\documentclass[10pt, reqno]{amsart}
\usepackage{amssymb}
\usepackage{eucal}
\usepackage{amsmath}
\usepackage{amscd}
\usepackage[dvips]{color}
\usepackage{multicol}
\usepackage[all]{xy}           
\usepackage{graphicx}
\usepackage{color}
\usepackage{colordvi}
\usepackage{xspace}
\usepackage{makecell}
\usepackage{tikz}

\usepackage{ifpdf}
\ifpdf
 \usepackage[colorlinks,final,backref=page,hyperindex]{hyperref}
\else
 \usepackage[colorlinks,final,backref=page,hyperindex,hypertex]{hyperref}
\fi

\begin{document}
\newcommand {\emptycomment}[1]{} 

\newcommand{\nc}{\newcommand}
\newcommand{\delete}[1]{}
\nc{\mfootnote}[1]{\footnote{#1}} 
\nc{\todo}[1]{\tred{To do:} #1}

\delete{
\nc{\mlabel}[1]{\label{#1}}  
\nc{\mcite}[1]{\cite{#1}}  
\nc{\mref}[1]{\ref{#1}}  
\nc{\meqref}[1]{\eqref{#1}} 
\nc{\mbibitem}[1]{\bibitem{#1}} 
}

\nc{\mlabel}[1]{\label{#1}  
{\hfill \hspace{1cm}{\bf{{\ }\hfill(#1)}}}}
\nc{\mcite}[1]{\cite{#1}{{\bf{{\ }(#1)}}}}  
\nc{\mref}[1]{\ref{#1}{{\bf{{\ }(#1)}}}}  
\nc{\meqref}[1]{\eqref{#1}{{\bf{{\ }(#1)}}}} 
\nc{\mbibitem}[1]{\bibitem[\bf #1]{#1}} 

\newtheorem{thm}{Theorem}[section]
\newtheorem{lem}[thm]{Lemma}
\newtheorem{cor}[thm]{Corollary}
\newtheorem{pro}[thm]{Proposition}
\newtheorem{conj}[thm]{Conjecture}
\theoremstyle{definition}
\newtheorem{defi}[thm]{Definition}
\newtheorem{ex}[thm]{Example}
\newtheorem{rmk}[thm]{Remark}
\newtheorem{pdef}[thm]{Proposition-Definition}
\newtheorem{condition}[thm]{Condition}

\renewcommand{\labelenumi}{{\rm(\alph{enumi})}}
\renewcommand{\theenumi}{\alph{enumi}}

\nc{\tred}[1]{\textcolor{red}{#1}}
\nc{\tblue}[1]{\textcolor{blue}{#1}}
\nc{\tgreen}[1]{\textcolor{green}{#1}}
\nc{\tpurple}[1]{\textcolor{purple}{#1}}
\nc{\btred}[1]{\textcolor{red}{\bf #1}}
\nc{\btblue}[1]{\textcolor{blue}{\bf #1}}
\nc{\btgreen}[1]{\textcolor{green}{\bf #1}}
\nc{\btpurple}[1]{\textcolor{purple}{\bf #1}}

\nc{\rp}[1]{\textcolor{blue}{Ruipu:#1}}
\nc{\cm}[1]{\textcolor{red}{Chengming:#1}}
\nc{\li}[1]{\textcolor{blue}{#1}}
\nc{\lir}[1]{\textcolor{blue}{Li:#1}}


\nc{\twovec}[2]{\left(\begin{array}{c} #1 \\ #2\end{array} \right )}
\nc{\threevec}[3]{\left(\begin{array}{c} #1 \\ #2 \\ #3 \end{array}\right )}
\nc{\twomatrix}[4]{\left(\begin{array}{cc} #1 & #2\\ #3 & #4 \end{array} \right)}
\nc{\threematrix}[9]{{\left(\begin{matrix} #1 & #2 & #3\\ #4 & #5 & #6 \\ #7 & #8 & #9 \end{matrix} \right)}}
\nc{\twodet}[4]{\left|\begin{array}{cc} #1 & #2\\ #3 & #4 \end{array} \right|}

\nc{\rk}{\mathrm{r}}
\newcommand{\g}{\mathfrak g}
\newcommand{\h}{\mathfrak h}
\newcommand{\pf}{\noindent{$Proof$.}\ }
\newcommand{\frkg}{\mathfrak g}
\newcommand{\frkh}{\mathfrak h}
\newcommand{\Id}{\rm{Id}}
\newcommand{\gl}{\mathfrak {gl}}
\newcommand{\ad}{\mathrm{ad}}
\newcommand{\add}{\frka\frkd}
\newcommand{\frka}{\mathfrak a}
\newcommand{\frkb}{\mathfrak b}
\newcommand{\frkc}{\mathfrak c}
\newcommand{\frkd}{\mathfrak d}
\newcommand {\comment}[1]{{\marginpar{*}\scriptsize\textbf{Comments:} #1}}

\nc{\tforall}{\text{ for all }}

\nc{\svec}[2]{{\tiny\left(\begin{matrix}#1\\
#2\end{matrix}\right)\,}}  
\nc{\ssvec}[2]{{\tiny\left(\begin{matrix}#1\\
#2\end{matrix}\right)\,}} 

\nc{\typeI}{local cocycle $3$-Lie bialgebra\xspace}
\nc{\typeIs}{local cocycle $3$-Lie bialgebras\xspace}
\nc{\typeII}{double construction $3$-Lie bialgebra\xspace}
\nc{\typeIIs}{double construction $3$-Lie bialgebras\xspace}

\nc{\bia}{{$\mathcal{P}$-bimodule ${\bf k}$-algebra}\xspace}
\nc{\bias}{{$\mathcal{P}$-bimodule ${\bf k}$-algebras}\xspace}

\nc{\rmi}{{\mathrm{I}}}
\nc{\rmii}{{\mathrm{II}}}
\nc{\rmiii}{{\mathrm{III}}}
\nc{\pr}{{\mathrm{pr}}}
\newcommand{\huaA}{\mathcal{A}}

\nc{\mcdot}{{}}

\nc{\OT}{constant $\theta$-}
\nc{\T}{$\theta$-}
\nc{\IT}{inverse $\theta$-}


\nc{\asi}{ASI\xspace}
\nc{\dualp}{transposed Poisson\xspace}
\nc{\Dualp}{Transposed Poisson\xspace}
\nc{\dualpop}{{\bf TPois}\xspace}
\nc{\ldualp}{derivation-transposed Poisson\xspace}

\nc{\spdualp}{sp-dual Poisson \xspace} \nc{\aybe}{AYBE\xspace}

\nc{\admset}{\{\pm x\}\cup K^{\times} x^{-1}}

\nc{\dualrep}{gives a dual representation\xspace}
\nc{\admt}{admissible to\xspace}

\nc{\ciri}{\circ_{\rm I}}
\nc{\cirii}{\circ_{\rm II}}
\nc{\ciriii}{\circ_{\rm III}}

\nc{\opa}{\cdot_A}
\nc{\opb}{\cdot_B}

\nc{\post}{positive type\xspace}
\nc{\negt}{negative type\xspace}
\nc{\invt}{inverse type\xspace}

\nc{\pll}{\beta}
\nc{\plc}{\epsilon}

\nc{\ass}{{\mathit{Ass}}}
\nc{\comm}{{\mathit{Comm}}}
\nc{\dend}{{\mathit{Dend}}}
\nc{\zinb}{{\mathit{Zinb}}}
\nc{\tdend}{{\mathit{TDend}}}
\nc{\prelie}{{\mathit{preLie}}}
\nc{\postlie}{{\mathit{PostLie}}}
\nc{\quado}{{\mathit{Quad}}}
\nc{\octo}{{\mathit{Octo}}}
\nc{\ldend}{{\mathit{ldend}}}
\nc{\lquad}{{\mathit{LQuad}}}

 \nc{\adec}{\check{;}} \nc{\aop}{\alpha}
\nc{\dftimes}{\widetilde{\otimes}} \nc{\dfl}{\succ} \nc{\dfr}{\prec}
\nc{\dfc}{\circ} \nc{\dfb}{\bullet} \nc{\dft}{\star}
\nc{\dfcf}{{\mathbf k}} \nc{\apr}{\ast} \nc{\spr}{\cdot}
\nc{\twopr}{\circ} \nc{\tspr}{\star} \nc{\sempr}{\ast}
\nc{\disp}[1]{\displaystyle{#1}}
\nc{\bin}[2]{ (_{\stackrel{\scs{#1}}{\scs{#2}}})}  
\nc{\binc}[2]{ \left (\!\! \begin{array}{c} \scs{#1}\\
    \scs{#2} \end{array}\!\! \right )}  
\nc{\bincc}[2]{  \left ( {\scs{#1} \atop
    \vspace{-.5cm}\scs{#2}} \right )}  
\nc{\sarray}[2]{\begin{array}{c}#1 \vspace{.1cm}\\ \hline
    \vspace{-.35cm} \\ #2 \end{array}}
\nc{\bs}{\bar{S}} \nc{\dcup}{\stackrel{\bullet}{\cup}}
\nc{\dbigcup}{\stackrel{\bullet}{\bigcup}} \nc{\etree}{\big |}
\nc{\la}{\longrightarrow} \nc{\fe}{\'{e}} \nc{\rar}{\rightarrow}
\nc{\dar}{\downarrow} \nc{\dap}[1]{\downarrow
\rlap{$\scriptstyle{#1}$}} \nc{\uap}[1]{\uparrow
\rlap{$\scriptstyle{#1}$}} \nc{\defeq}{\stackrel{\rm def}{=}}
\nc{\dis}[1]{\displaystyle{#1}} \nc{\dotcup}{\,
\displaystyle{\bigcup^\bullet}\ } \nc{\sdotcup}{\tiny{
\displaystyle{\bigcup^\bullet}\ }} \nc{\hcm}{\ \hat{,}\ }
\nc{\hcirc}{\hat{\circ}} \nc{\hts}{\hat{\shpr}}
\nc{\lts}{\stackrel{\leftarrow}{\shpr}}
\nc{\rts}{\stackrel{\rightarrow}{\shpr}} \nc{\lleft}{[}
\nc{\lright}{]} \nc{\uni}[1]{\tilde{#1}} \nc{\wor}[1]{\check{#1}}
\nc{\free}[1]{\bar{#1}} \nc{\den}[1]{\check{#1}} \nc{\lrpa}{\wr}
\nc{\curlyl}{\left \{ \begin{array}{c} {} \\ {} \end{array}
    \right .  \!\!\!\!\!\!\!}
\nc{\curlyr}{ \!\!\!\!\!\!\!
    \left . \begin{array}{c} {} \\ {} \end{array}
    \right \} }
\nc{\leaf}{\ell}       
\nc{\longmid}{\left | \begin{array}{c} {} \\ {} \end{array}
    \right . \!\!\!\!\!\!\!}
\nc{\ot}{\otimes} \nc{\sot}{{\scriptstyle{\ot}}}
\nc{\otm}{\overline{\ot}}
\nc{\ora}[1]{\stackrel{#1}{\rar}}
\nc{\ola}[1]{\stackrel{#1}{\la}}
\nc{\pltree}{\calt^\pl}
\nc{\epltree}{\calt^{\pl,\NC}}
\nc{\rbpltree}{\calt^r}
\nc{\scs}[1]{\scriptstyle{#1}} \nc{\mrm}[1]{{\rm #1}}
\nc{\dirlim}{\displaystyle{\lim_{\longrightarrow}}\,}
\nc{\invlim}{\displaystyle{\lim_{\longleftarrow}}\,}
\nc{\mvp}{\vspace{0.5cm}} \nc{\svp}{\vspace{2cm}}
\nc{\vp}{\vspace{8cm}} \nc{\proofbegin}{\noindent{\bf Proof: }}
\nc{\proofend}{$\blacksquare$ \vspace{0.5cm}}
\nc{\freerbpl}{{F^{\mathrm RBPL}}}
\nc{\sha}{{\mbox{\cyr X}}}  
\nc{\ncsha}{{\mbox{\cyr X}^{\mathrm NC}}} \nc{\ncshao}{{\mbox{\cyr
X}^{\mathrm NC,\,0}}}
\nc{\shpr}{\diamond}    
\nc{\shprm}{\overline{\diamond}}    
\nc{\shpro}{\diamond^0}    
\nc{\shprr}{\diamond^r}     
\nc{\shpra}{\overline{\diamond}^r}
\nc{\shpru}{\check{\diamond}} \nc{\catpr}{\diamond_l}
\nc{\rcatpr}{\diamond_r} \nc{\lapr}{\diamond_a}
\nc{\sqcupm}{\ot}
\nc{\lepr}{\diamond_e} \nc{\vep}{\varepsilon} \nc{\labs}{\mid\!}
\nc{\rabs}{\!\mid} \nc{\hsha}{\widehat{\sha}}
\nc{\lsha}{\stackrel{\leftarrow}{\sha}}
\nc{\rsha}{\stackrel{\rightarrow}{\sha}} \nc{\lc}{\lfloor}
\nc{\rc}{\rfloor}
\nc{\tpr}{\sqcup}
\nc{\nctpr}{\vee}
\nc{\plpr}{\star}
\nc{\rbplpr}{\bar{\plpr}}
\nc{\sqmon}[1]{\langle #1\rangle}
\nc{\forest}{\calf}
\nc{\altx}{\Lambda_X} \nc{\vecT}{\vec{T}} \nc{\onetree}{\bullet}
\nc{\Ao}{\check{A}}
\nc{\seta}{\underline{\Ao}}
\nc{\deltaa}{\overline{\delta}}
\nc{\trho}{\tilde{\rho}}

\nc{\rpr}{\circ}
\nc{\dpr}{{\tiny\diamond}}
\nc{\rprpm}{{\rpr}}

\nc{\mmbox}[1]{\mbox{\ #1\ }} \nc{\ann}{\mrm{ann}}
\nc{\Aut}{\mrm{Aut}} \nc{\can}{\mrm{can}}
\nc{\twoalg}{{two-sided algebra}\xspace}
\nc{\colim}{\mrm{colim}}
\nc{\Cont}{\mrm{Cont}} \nc{\rchar}{\mrm{char}}
\nc{\cok}{\mrm{coker}} \nc{\dtf}{{R-{\rm tf}}} \nc{\dtor}{{R-{\rm
tor}}}
\renewcommand{\det}{\mrm{det}}
\nc{\depth}{{\mrm d}}
\nc{\Div}{{\mrm Div}} \nc{\End}{\mrm{End}} \nc{\Ext}{\mrm{Ext}}
\nc{\Fil}{\mrm{Fil}} \nc{\Frob}{\mrm{Frob}} \nc{\Gal}{\mrm{Gal}}
\nc{\GL}{\mrm{GL}} \nc{\Hom}{\mrm{Hom}} \nc{\hsr}{\mrm{H}}
\nc{\hpol}{\mrm{HP}} \nc{\id}{\mrm{id}} \nc{\im}{\mrm{im}}
\nc{\incl}{\mrm{incl}} \nc{\length}{\mrm{length}}
\nc{\LR}{\mrm{LR}} \nc{\mchar}{\rm char} \nc{\NC}{\mrm{NC}}
\nc{\mpart}{\mrm{part}} \nc{\pl}{\mrm{PL}}
\nc{\ql}{{\QQ_\ell}} \nc{\qp}{{\QQ_p}}
\nc{\rank}{\mrm{rank}} \nc{\rba}{\rm{RBA }} \nc{\rbas}{\rm{RBAs }}
\nc{\rbpl}{\mrm{RBPL}}
\nc{\rbw}{\rm{RBW }} \nc{\rbws}{\rm{RBWs }} \nc{\rcot}{\mrm{cot}}
\nc{\rest}{\rm{controlled}\xspace}
\nc{\rdef}{\mrm{def}} \nc{\rdiv}{{\rm div}} \nc{\rtf}{{\rm tf}}
\nc{\rtor}{{\rm tor}} \nc{\res}{\mrm{res}} \nc{\SL}{\mrm{SL}}
\nc{\Spec}{\mrm{Spec}} \nc{\tor}{\mrm{tor}} \nc{\Tr}{\mrm{Tr}}
\nc{\mtr}{\mrm{sk}}

\nc{\ab}{\mathbf{Ab}} \nc{\Alg}{\mathbf{Alg}}
\nc{\Algo}{\mathbf{Alg}^0} \nc{\Bax}{\mathbf{Bax}}
\nc{\Baxo}{\mathbf{Bax}^0} \nc{\RB}{\mathbf{RB}}
\nc{\RBo}{\mathbf{RB}^0} \nc{\BRB}{\mathbf{RB}}
\nc{\Dend}{\mathbf{DD}} \nc{\bfk}{{K}} \nc{\bfone}{{\bf 1}}
\nc{\base}[1]{{a_{#1}}} \nc{\detail}{\marginpar{\bf More detail}
    \noindent{\bf Need more detail!}
    \svp}
\nc{\Diff}{\mathbf{Diff}} \nc{\gap}{\marginpar{\bf
Incomplete}\noindent{\bf Incomplete!!}
    \svp}
\nc{\FMod}{\mathbf{FMod}} \nc{\mset}{\mathbf{MSet}}
\nc{\rb}{\mathrm{RB}} \nc{\Int}{\mathbf{Int}}
\nc{\Mon}{\mathbf{Mon}}
\nc{\remarks}{\noindent{\bf Remarks: }}
\nc{\OS}{\mathbf{OS}} 
\nc{\Rep}{\mathbf{Rep}}
\nc{\Rings}{\mathbf{Rings}} \nc{\Sets}{\mathbf{Sets}}
\nc{\DT}{\mathbf{DT}}

\nc{\BA}{{\mathbb A}} \nc{\CC}{{\mathbb C}} \nc{\DD}{{\mathbb D}}
\nc{\EE}{{\mathbb E}} \nc{\FF}{{\mathbb F}} \nc{\GG}{{\mathbb G}}
\nc{\HH}{{\mathbb H}} \nc{\LL}{{\mathbb L}} \nc{\NN}{{\mathbb N}}
\nc{\QQ}{{\mathbb Q}} \nc{\RR}{{\mathbb R}} \nc{\BS}{{\mathbb{S}}} \nc{\TT}{{\mathbb T}}
\nc{\VV}{{\mathbb V}} \nc{\ZZ}{{\mathbb Z}}


\nc{\calao}{{\mathcal A}} \nc{\cala}{{\mathcal A}}
\nc{\calc}{{\mathcal C}} \nc{\cald}{{\mathcal D}}
\nc{\cale}{{\mathcal E}} \nc{\calf}{{\mathcal F}}
\nc{\calfr}{{{\mathcal F}^{\,r}}} \nc{\calfo}{{\mathcal F}^0}
\nc{\calfro}{{\mathcal F}^{\,r,0}} \nc{\oF}{\overline{F}}
\nc{\calg}{{\mathcal G}} \nc{\calh}{{\mathcal H}}
\nc{\cali}{{\mathcal I}} \nc{\calj}{{\mathcal J}}
\nc{\call}{{\mathcal L}} \nc{\calm}{{\mathcal M}}
\nc{\caln}{{\mathcal N}} \nc{\calo}{{\mathcal O}}
\nc{\calp}{{\mathcal P}} \nc{\calq}{{\mathcal Q}} \nc{\calr}{{\mathcal R}}
\nc{\calt}{{\mathcal T}} \nc{\caltr}{{\mathcal T}^{\,r}}
\nc{\calu}{{\mathcal U}} \nc{\calv}{{\mathcal V}}
\nc{\calw}{{\mathcal W}} \nc{\calx}{{\mathcal X}}
\nc{\CA}{\mathcal{A}}

\nc{\fraka}{{\mathfrak a}} \nc{\frakB}{{\mathfrak B}}
\nc{\frakb}{{\mathfrak b}} \nc{\frakd}{{\mathfrak d}}
\nc{\oD}{\overline{D}}
\nc{\frakF}{{\mathfrak F}} \nc{\frakg}{{\mathfrak g}}
\nc{\frakm}{{\mathfrak m}} \nc{\frakM}{{\mathfrak M}}
\nc{\frakMo}{{\mathfrak M}^0} \nc{\frakp}{{\mathfrak p}}
\nc{\frakS}{{\mathfrak S}} \nc{\frakSo}{{\mathfrak S}^0}
\nc{\fraks}{{\mathfrak s}} \nc{\os}{\overline{\fraks}}
\nc{\frakT}{{\mathfrak T}}
\nc{\oT}{\overline{T}}
\nc{\frakX}{{\mathfrak X}} \nc{\frakXo}{{\mathfrak X}^0}
\nc{\frakx}{{\mathbf x}}
\nc{\frakTx}{\frakT}      
\nc{\frakTa}{\frakT^a}        
\nc{\frakTxo}{\frakTx^0}   
\nc{\caltao}{\calt^{a,0}}   
\nc{\ox}{\overline{\frakx}} \nc{\fraky}{{\mathfrak y}}
\nc{\frakz}{{\mathfrak z}} \nc{\oX}{\overline{X}}

\font\cyr=wncyr10

\nc{\al}{\alpha}
\nc{\lam}{\lambda}
\nc{\lr}{\longrightarrow}


\title[Diassociative bialgebras]{Skew-symmetric Frobenius dialgebras, diassociative bialgebras and Yang-Baxter equations}

\author{Dilei Lu}
\address{College of Applied Science, Beijing Information Science and Technology University, Beijing 100192, China}
         \email{ludyray@bistu.edu.cn}

\date{\today}

\begin{abstract}

A diassociative algebra (or dialgebra), a generalization of associative algebras with two associative products, is a fundamental algebraic structure of great importance and wide-ranging applications. In this paper, we establish a bialgebra theory for dialgebras completely and systematically. Explicitly, we introduce the notion of diassociative bialgebras (bi-dialgebras) which are equivalent to double constructions of Frobenius dialgebras as well as matched pairs of dialgebras. Moreover, we show that a bi-dialgebra gives rise to a Leibniz bialgebra, thereby lifting the classical relation  between dialgebras and Leibniz algebras to the bialgebra level. In the  coboundary cases, we introduce the diassociative Yang-Baxter equation (DAYBE) in a 
dialgebra, and solutions whose skew-symmetric part is invariant give rise to the so-called quasi-triangular bi-dialgebras. Moreover, weighted  $\mathcal{O}$-operators on dialgebras are defined, which afford an operator interpretation for such solutions of the DAYBE.
 Finally, we also develop the triangular and factorizable theories for bi-dialgebras, and introduce skew-symmetric Rota-Baxter Frobenius dialgebras of arbitrary weight. We show that  skew-symmetric Rota-Baxter Frobenius dialgebras of weight  zero give rise to triangular bi-dialgebras, whereas those of nonzero weight induce factorizable bi-dialgebras.  As an application,  we provide concrete examples of factorizable bi-dialgebras from associative algebras and dialgebras, respectively.
\end{abstract}

\subjclass[2020]{
16T10,  
16T25,  
16W99, 
17A32,  
17B38,  
17B62,  
17B60,  
81R60.  
}

\keywords{Associative algebra; dialgebra; Frobenius dialgebra;    diassociative bialgebra; diassociative Yang-Baxter equation; $\mathcal{O}$-operator}

\maketitle


\tableofcontents

\allowdisplaybreaks

\section{Introduction}

 The aim of this paper is to study the bialgebra theory for dialgebras with the motivation from the great importance of Lie bialgebras and infinitesimal bialgebras.  By introducing analogs of the classical Yang-Baxter equation  and 
$\mathcal{O}$-operators, we also develop the quasi-triangular theory for bi-dialgebras, including the triangular and factorizable cases.

\subsection{Dialgebras} 
The notion of  diassociative algebras (or dialgebras, for
short)  was first introduced by Loday in the early 1990s in connection with  periodicity phenomena in algebraic $K$-theory \cite{L4,Loday4}. As a generalization of associative algebras, dialgebras are equipped with two associative products satisfying three compatibility conditions, and they naturally give rise to a Leibniz algebra, which can be regarded as a non-commutative generalization of Lie algebras \cite{Loday3,Loday4}.  In particular, just as associative algebras play the role of universal enveloping algebras for Lie algebras, dialgebras serve as the universal enveloping algebras for Leibniz algebras \cite{Loday5}.  Moreover, a class of dialgebras arises from averaging operators on associative algebras \cite{MA}. Surprisingly, \cite{MAMG2} showed that the dialgebra cohomology of a dialgebra admits a $G$-algebra structure, analogous to that on the Hochschild cohomology of an associative algebra \cite{GM1,GM2}.  Since then,  remarkable progress has been made in the study of dialgebras, including their structural theory \cite{GCM,VRFR}, deformation theory \cite{FM,MAMG1,DonY},  Gröbner-Shirshov bases \cite{BCL} and algebraic operads  \cite{MRB,JBGN,UK,VB1}.  Furthermore, dialgebras have arisen in numerous other branches of pure and applied mathematics,  such as the Yang-Baxter equations \cite{FLVO},  conformal algebraic structures \cite{PSKXY}, functional analysis \cite{FR} and so on. Several further recent developments on dialgebras can be found in \cite{BMR,KPS2,KPS1}. Especially, the operad of dialgebras is the operadic Koszul duality of the operad of dendriform algebras which are
equipped with an associative product  written as a linear combination of
nonassociative compositions \cite{B2,L4}.

Apart from Leibniz algebras and dendriform algebras, dialgebras are closely related to other important classes of algebras. This is a well-known diagram introduced by Chapoton \cite{CF23}, which illustrates the close relationships among Lie algebras, associative algebras, commutative algebras, pre-Lie algebras, dendriform algebras, Zinbiel algebras, Leibniz algebras, dialgebras and perm algebras. 
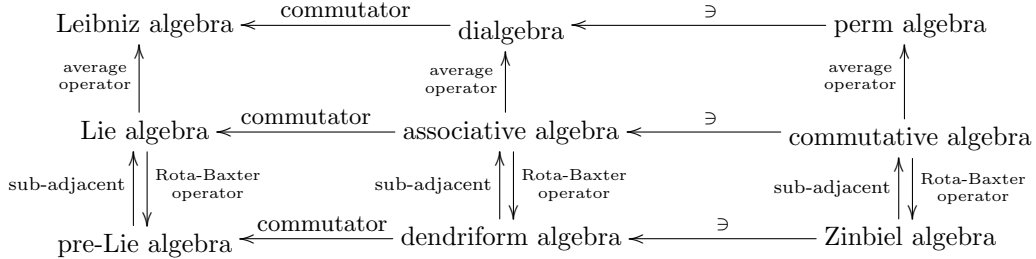
\begin{equation*}    
\xymatrix@C=2cm@R=0.8cm{ 
\text{Leibniz algebra}     &  \ar[l]_-{\small\txt{commutator}}  \txt{dialgebra}  &   \text{perm  algebra}\ar@{->}[l]_{\ni} \\    
\text{Lie algebra}  \ar@<.5ex>[u]^{\tiny{\txt{average \\ operator}}} \ar@<.2ex>[d]^{\tiny{\txt{Rota-Baxter \\ operator}}}  &\ar[l]_-{\small\txt{commutator}}    \text{associative algebra}   \ar@<.2ex>[d]^{\tiny{\txt{Rota-Baxter \\ operator}}}  \ar@<.5ex>[u]^-{\tiny{\txt{average \\ operator}}}     &   \ar[l]_{\ni} \txt{commutative  algebra} \ar@<.5ex>[u]^{\tiny{\txt{average \\ operator}}} \ar@<.2ex>[d]^-{\tiny{\txt{Rota-Baxter \\ operator}}}     \\    
\txt{pre-Lie algebra}   \ar@<1ex>[u]^{\text{sub-adjacent}} &  \ar[l]_-{\;\;\small\txt{commutator}}  \text{dendriform algebra}   \ar@<1ex>[u]^{\text{sub-adjacent}} &  \ar[l]_-{\ni} \text{Zinbiel  algebra}\ar@<1ex>[u]^-{\text{sub-adjacent}}  }
\end{equation*}
where the symbol $\ni$ means the algebra on the right is a special case (commutative version) of that on the left.  Algebraically, the first-row algebras arise from  the second row via averaging operators, and the third-row algebras via Rota-Baxter operators  \cite{MA,BBG,PBG,JBGN}. At the operadic level, all operads corresponding to the algebras in the diagram are Koszul and the operadic Koszul duality  is given by central symmetry \cite{BGL,Gik,LB}. Furthermore, the operads of the first and third rows of algebras are the duplicator and disuccessor of those in the middle row, respectively \cite{BBG,JBGN}. More importantly, the operads of the algebras in the second row are cyclic, while those of the first and third rows are anticyclic \cite{CF1,GK}. 

In this paper, we first recall the basic properties and classical constructions of dialgebras. These include constructions via averaging operators on associative algebras (Proposition \ref{avonass}), the tensor product of perm algebras and associative algebras (Proposition \ref{ifdualpp}), and bimodules of associative algebras (Proposition \ref{difromrepAss}). After that, we introduce bimodules of dialgebras and show that every given bimodule of a dialgebra admits a natural dual bimodule (Lemma \ref{dualbimodofdias}). Due to the anticyclicity of the operad of dialgebras,
Chapoton introduced the invariance condition for bilinear forms on dialgebras   in \cite{CF1}, which motivates us to define (skew-symmetric) Frobenius dialgebras (Definition \ref{Frbendi}). Note that, unlike symmetric Frobenius algebras, for which the invariant bilinear form on the associative algebra is symmetric, this invariant bilinear form on dialgebras is skew-symmetric rather than symmetric (in the sense of Proposition \ref{symFrodi}). We then provide three explicit constructions of  skew-symmetric  Frobenius dialgebras from associative algebras, dialgebras, and the tensor products of  symmetric Frobenius algebras with quadratic perm algebras (Propositions \ref{p:ppplgphpl} and  \ref{ifqdualpp}). Furthermore, there is a quadratic Leibniz algebra arising from a skew-symmetric Frobenius dialgebra (Proposition \ref{subjecntqla}). More importantly, we prove the equivalence between the regular bimodule and the coregular bimodule in the  skew-symmetric Frobenius dialgebras (Proposition \ref{dualdj:x}). 
Finally, we investigate the invariant elements of dialgebras together with their equivalent characterizations (Definition \ref{invele1}) and give the tensor forms of invariant bilinear forms on dialgebras (Proposition \ref{tensorformIndias}).


\subsection{Diassociative bialgebras and diassociative Yang-Baxter equations}

A bialgebra structure consists of an algebra structure and a coalgebra structure related by appropriate compatibility conditions. They often play important roles in various fields and are closely connected with other structures arising from mathematics and physics. Well-known examples of such structures include Lie bialgebras and antisymmetric infinitesimal bialgebras. As structures equivalent to Manin triples of Lie algebras, Lie bialgebras are the infinitesimal versions of Poisson-Lie groups, and play a crucial role in the study of quantum groups \cite{CV,Dr}. Antisymmetric infinitesimal bialgebras, regarded as the associative analogues of Lie bialgebras, can be characterized as double constructions of Frobenius algebras. These structures have broad applications in two-dimensional topological field theory and string theory \cite{KJ,LAPH}.  
Note that, with the exception of dialgebras, all algebraic structures appearing in the preceding diagram possess well-established bialgebra theories, including Lie algebras \cite{CV,Dr}, (commutative) associative algebras \cite{B2}, pre-Lie algebras \cite{B1},    dendriform algebras \cite{B2}, Zinbiel algebras \cite{WNB}, 
Leibniz algebras \cite{TS},  perm algebras \cite{BYZ}.  In view of the absence of bialgebra structures for dialgebras, developing a suitable bialgebra theory for dialgebras is thus a natural and well-motivated undertaking. 

Additionally, our research motivation is also derived from a critical question: does there exist a bialgebra structure on dialgebras that induces a Leibniz bialgebra (introduced in \cite{TS}), thereby generalizing the link between dialgebras and Leibniz algebras to the bialgebra level? In this paper, we give an affirmative answer to this question. Explicitly, we introduce the notion of a bi-dialgebra, defined as a structure consisting of a dialgebra and a co-dialgebra that are compatible with each other. Such a structure is equivalent to a double construction of Frobenius dialgebras, as well as to a certain matched pair of dialgebras  (Theorem \ref{sandengjia}). Moreover, we show that a bi-dialgebra naturally gives rise to a Leibniz bialgebra via a commutator approach, thereby lifting the  connection that a dialgebra induces a Leibniz algebra to the level of bialgebras (Theorem \ref{bidiatoLeibbi}). We summarize these results in the following diagram:
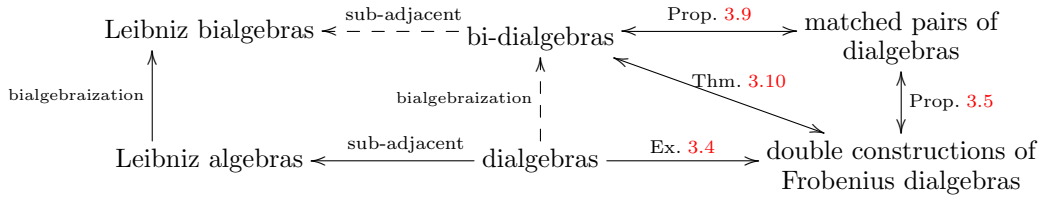
\begin{equation*}    
\xymatrix@C=1.8cm@R=0.8cm{ 
\text{Leibniz bialgebras}     &  \ar@{-->}[l]_-{\quad\text{sub-adjacent}}  \txt{bi-dialgebras}  \ar@{<->}[r]^-{\mathrm{Prop.\;}{\ref{mpdppba}}} \ar@{<->}[dr]^-{\mathrm{Thm.\;}{\ref{sandengjia}}}& \txt{matched pairs of \\dialgebras }\ar@{<->}[d]^-{ \mathrm{Prop.\;}{\ref{mpandmtofdppa}}} \\    
\text{Leibniz algebras}  \ar@<5ex>[u]^{\tiny{\txt{bialgebraization}}}  &\ar[l]_-{\quad\text{sub-adjacent}}    \text{dialgebras}  \ar@{-->}[u]^-{\tiny{\txt{bialgebraization}}} \ar@{->}[r]^-{\mathrm{Ex.\;}{\ref{doublefromdias}}} & \txt{double constructions of\\ Frobenius dialgebras} }
\end{equation*}

In the coboundary  case,  we introduce the diassociative Yang-Baxter equation (DAYBE), an analogue of the classical Yang-Baxter equation for dialgebras. Different types of bi-dialgebras arise from its solutions with distinct properties:
\begin{itemize}
\item Solutions whose skew-symmetric part is invariant induce quasi-triangular bi-dialgebras.  We investigate the operator form of the DAYBE (Proposition \ref{operaforDAYBE}) and accordingly introduce the notion of $\mathcal{O}$-operators with weights on dialgebras  to characterize those solutions of the DAYBE with invariant skew-symmetric parts (Theorem \ref{ilovethisthm}).

\item Symmetric solutions give rise to triangular bi-dialgebras. By introducing the notion of a pre-dialgebra as the disuccessor of dialgebras in the sense of operad theory, we provide explicit and concrete constructions for these symmetric solutions and, consequently, for triangular bi-dialgebras. (Corollary \ref{pretoBialgebra}).

\item Solutions whose skew-symmetric part is invariant  and nondegenerate  yield factorizable bi-dialgebras. We prove that every factorizable bi-dialgebra induces a factorization of its underlying bi-dialgebra (Proposition \ref{prop:fpt}). Moreover, we show that the Drinfeld classical double of a bi-dialgebra admits a canonical factorizable bi-dialgebra structure automatically (Theorem \ref{Dfincdouble}).
\end{itemize}

Movitvated by the study of quasi-triangular and factorizable theories of Lie bialgebras \cite{LS} and antisymmetric infinitesimal bialgebras \cite{SYWY}, we introduce the notion of a skew-symmetric Rota-Baxter Frobenius dialgebra  by equipping a Frobenius dialgebra   with a Rota-Baxter operator satisfying a compatibility condition. We show that a skew-symmetric Rota-Baxter Frobenius dialgebra  of zero weight can give rise to a triangular bi-dialgebra (Proposition \ref{rbfna2}). Moreover, we establish a one-to-one correspondence between factorizable  bi-dialgebras and skew-symmetric Rota-Baxter Frobenius dialgebras of nonzero weights (Theorem \ref{thm:fdiasb2qrpx}). As an application, we construct explicit examples of factorizable bi-dialgebras from associative algebras and dialgebras, respectively  (Corollary \ref{EX:jpx:ppplgphpl}).

{}

\subsection{Layout of the paper}
The paper is organized as follows.

 In Section \ref{Property},   we exhibit fundamental properties of dialgebras and provide several explicit constructions. We then introduce the notion of bimodules for dialgebras and give the dual bimodule associated with a given bimodule of a dialgebra. By equipping a dialgebra with a nondegenerate (skew-symmetric) invariant bilinear form, we introduce the notion of a (skew-symmetric) Frobenius dialgebra. We further present two natural constructions of skew-symmetric Frobenius dialgebras, arising from associative algebras and dialgebras respectively. Finally, we study invariant elements of dialgebras and establish an equivalent characterization for them, which serves as indispensable preliminaries for our subsequent investigations.

 In Section \ref{Bialgebra}, the notions of double constructions of Frobenius dialgebras and  diassociative bialgebras (bi-dialgebras) are introduced. Their equivalence is interpreted in terms of matched pairs of dialgebras.
 Furthermore, we prove that a bi-dialgebra naturally induces a Leibniz bialgebra via the sub-adjacent way. 
  Then we consider coboundary bi-dialgebras whose study leads to the introduction of diassociative Yang-Baxter equation (DAYBE). Explicitly, a solution whose skew-symmetric part is invariant gives a special coboundary bi-dialgebras, called quasi-triangular bi-dialgebras. Finally, we introduce the notion of  $\mathcal{O}$-operators on dialgebras with weights, which provides an operator-form interpretation for solutions to the DAYBE with invariant skew-symmetric part. 
 
 In Section \ref{CBialgebra}, we study two special classes of quasi-triangular bi-dialgebras. The triangular bi-dialgebra is constructed from symmetric solutions to the DAYBE, and we introduce pre-dialgebras to obtain such solutions in certain larger dialgebras.
The factorizable bi-dialgebra induces a factorization of the underlying dialgebra, and we show that the Drinfeld classical double of a bi-dialgebra naturally admits such a structure. Finally,
We define skew-symmetric Rota-Baxter Frobenius dialgebras of arbitrary weight and prove that those of weight zero induce triangular bi-dialgebras. A one-to-one correspondence is established between factorizable bi-dialgebras and those of nonzero weight.
As an application, we construct distinct factorizable bi-dialgebras from associative algebras and dialgebras, respectively.

 Unless otherwise specified,  we work with finite-dimensional vector spaces and algebras over a field $\mathbb{F}$ of characteristic $0$ throughout this paper, although many results and definitions remain valid in the infinite-dimensional case.
\begin{enumerate}
\item  For a  $\mathbb{F}$-vector space $V$, let  $V^{*}:=\operatorname{Hom}_{\mathbb{F}}(V, \mathbb{F})$  denote the dual  $\mathbb{F}$-vector space. Denote the usual pairing between  $V^{*}$  and  $V$  by
$$
\langle \ , \ \rangle: V^{*} \times V \rightarrow \mathbb{F}, \quad \left\langle u^{*}, v\right\rangle=u^{*}(v),\quad  \forall u^{*} \in V^{*}, v \in V .
$$

\item   Let  $A$  and  $V$  be vector spaces. For a linear map  $\varphi: A \rightarrow {\rm End}(V)$, define a linear map  $\varphi^{*}: A \rightarrow {\rm End}(V^*)$  by 
\begin{align}
\left\langle \varphi^{*}(x) v^{*}, u\right\rangle=-\left\langle v^{*}, \varphi(x) u\right\rangle, \quad \forall x \in A, u \in V, v^{*} \in V^{*}. \label{dualmap}
\end{align} 

\item  Let $A$ be a vector space with a binary
operation $\diamond$. Define linear maps $L_{\diamond},
R_{\diamond}:A\rightarrow {\rm End}(A)$ respectively by
\begin{eqnarray*}
L_{\diamond}(x)y:=x\diamond y =: R_{\diamond}(y)x, \;\;\; \forall x, y\in A.
\end{eqnarray*}
 If there is a binary operation $\odot$ on the dual space $A^*$, we denote the linear maps $\mathcal{L}_{\odot},
\mathcal{R}_{\odot}:A^*\rightarrow {\rm End}(A^*)$ respectively by
 \begin{eqnarray*}
\mathcal{L}_{\odot}(a^*)b^*:= a^* \odot b^* =: \mathcal{R}_{\odot}(b^*)a^*, \;\;\;\forall a^*,b^*\in A^*.
\end{eqnarray*}
\end{enumerate}

\section{Frobenius dialgebras}\label{Property}

 In this section, we first present the fundamental properties and interesting constructions of diassociative algebras (dialgebras). We then introduce the notion of the bimodules  of dialgebras.  The dual bimodule of a bimodule of a dialgebra  is also given. Equipped with a nondegenerate (skew-symmetric) invariant bilinear form on a dialgebra, we introduce the concept of a (skew-symmetric) Frobenius dialgebra. We further give two natural constructions of  skew-symmetric Frobenius dialgebras from associative algebras and dialgebras, respectively.  Finally, we investigate invariant elements for dialgebras and provide an equivalent characterization thereof, as indispensable preliminaries for our subsequent study.

\subsection{Some basic results on diassociative algebras}

\begin{defi}\label{def1}
A   \textbf{diassociative algebra (dialgebra)} is a triple $(A,  \rhd,\lhd)$ where $(A,  \rhd)$ and $(A, \lhd)$ are associative algebras  satisfying the following conditions 
\begin{align}
(x \lhd y) \lhd z&=x \lhd (y \rhd z), \label{dias1}\\
(x \rhd y) \lhd z&=x \rhd(y \lhd z), \label{dias2}\\
(x \lhd y) \rhd z&=x \rhd(y\rhd z), \quad \forall x,y,z \in A.\label{dias3} 
\end{align}
Let $(A, \rhd_A,\lhd_A)$ and $(B, \rhd_B,\lhd_B)$ be  dialgebras. A linear map $f: A \to B$ is called a {\bf homomorphism of   dialgebras} if it satisfies 
\begin{align}
 f(x \rhd_A y) = f(x) \rhd_B f(y),\quad  f(x \lhd_A y) = f(x) \lhd_B f(y),\quad \forall x,y \in A.
\end{align}
\end{defi}
\begin{rmk}
Although some results still hold in the unital case, throughout this article we only consider non-unital dialgebras and associative algebras.
\end{rmk}

Recall that a  {\bf  Leibniz algebra} is a pair  $(A,  [\cdot,\cdot])$ where $A$ is a vector space with a binary operation  $[\cdot,\cdot]: A \otimes A \to A$ such that  
\begin{eqnarray}
[x , [y ,z]] = [[x ,y] ,z] + [y, [x,z]], \quad \forall x,y,z \in A. \label{Leibniz}
\end{eqnarray}

\begin{pro}\cite{L4}\label{diatoLeib}
    Let $(A,  \rhd,\lhd)$ be a dialgebra. Then the binary operation $[\cdot,\cdot]: A \otimes A \to A$ given by 
\begin{align}
  [x, y]  = x \rhd y - y \lhd x, \quad \forall x,y \in A, \label{inducedLeib}
\end{align}
defines a Leibniz algebra $(A,[\cdot, \cdot])$, which is called the {\bf sub-adjacent Leibniz algebra} of $(A,  \rhd,\lhd)$.
\end{pro}

Apart from the known algebraic examples of dialgebras in  \cite{L4}, some low-dimensional classifications of complex diassociative algebras are given in \cite{RRB2}, while an infinite diassociative algebras structure over a polynomial algebra  is constructed in \cite{LY2010}. In what follows, we present several examples of dialgebras.

\begin{ex}\label{dwdpp}
(1) Let $(A,\cdot)$ be an associative algebra, then the following equations  
$$
x \rhd y=x \cdot y=x \lhd y, \quad \forall x,y \in A,
$$
define a dialgebra structure  on  $A$.   Thus any associative algebras are dialgebras. 

(2)  Let  $V$  be a  $\mathbb{F}$-vector space and  $\varphi \in V^{*} \backslash\{0\}$  be a nonzero linear functional. Define two binary  operations on $V$ by 
\begin{align*}
u \rhd v =\varphi(u)v, \quad u \lhd v =\varphi(v)u, \quad \forall u,v \in V.
\end{align*}
 Then $(V,\rhd, \lhd)$ is a dialgebra.

(3) Let  $A:= \mathbb{F}\left[t^{ \pm}\right] \partial_{1} \oplus \mathbb{F}\left[t^{ \pm}\right] \partial_{2} $ be the  algebra of Laurent polynomials in the indeterminate  $t$. The set  $\left\{t^{i} \partial_{m} \mid i \in \mathbb{Z}, m \in\{1,2\}\right\}$  forms a basis of  $A$ over  $\mathbb{F}$. We define two binary  operations  $\rhd,\lhd$ on  $A$  respectively by
\begin{align}
t^{i}\partial_{m}   \rhd t^{j}\partial_{n}  =t^{i+j+\delta_{m,1}}\partial_{n}, \quad t^{i}\partial_{m}  \lhd t^{j}\partial_{n}  = t^{i+j+\delta_{n,1}}\partial_{m}, \quad \forall i, j \in \mathbb{Z},  m, n \in\{1,2\}, \label{ex:diainfinix}
\end{align}
where $\delta_{a,b}$ is the Kronecker symbol with $\delta_{a,b} =1 $ if $a=b$, $\delta_{a,b} =0 $ if $a\ne b$. Then  $(A,\rhd,\lhd)$  is an infinite-dimensional dialgebra.
\end{ex}

{}

An important class of dialgebras arises from averaging operators on associative algebras. We now recall the notion of an averaging operator on an associative algebra.

\begin{defi}
    Let $(A,  \cdot)$ be an  associative algebra and $P: A \to A$  be a linear map. If $P$ satisfies the following equations
\begin{align}
P(x) \cdot P(y) = P(P(x) \cdot y )= P( x \cdot  P(y)  ), \quad \forall x,y \in A,
\end{align} 
then  $P$ is called an {\bf averaging operator} on $(A,  \cdot)$.
\end{defi}

\begin{pro} \label{avonass}\cite{MA} Let $(A,  \cdot)$ be an associative algebra and $P: A \to A$ be an  averaging operator on $(A,  \cdot)$. Define new binary operations $\rhd$ and $\lhd$ on $A$ by
\begin{align}
 x \rhd y = P(x)\cdot y , \quad   x \lhd y = x \cdot P(y),\quad \forall x,y \in A.
\end{align} 
Then  $(A, \rhd ,\lhd)$ is a dialgebra.
\end{pro}

\begin{ex}\label{Ex:dias}
(1) Let  $(A, \cdot)$ be the 2-dimensional non-commutative  associative   algebra with a basis $\{e_1, e_2\}$ whose the  nonzero products of $\cdot$ are given by
\begin{align}
   e_1 \cdot e_2  =e_1,\quad    e_2 \cdot e_2  =e_2. \label{ass:1}
\end{align}
Let $P : A \to A$ be a linear map given by
\begin{align*}
P(e_1) = 0, \quad P(e_2) = e_1 + e_2.
\end{align*}
Then  $P$ is an averaging operator on $(A, \cdot)$. Thus by Proposition \ref{avonass}, there is a 2-dimensional dialgebra $(A, \rhd,\lhd)$ whose the nonzero products of $\rhd,\lhd$ are explicitly given by
{}
\begin{align}
e_2 \rhd e_2  = P(e_2) \cdot e_2 = e_1+e_2, \;e_1 \lhd e_2  = e_1 \cdot P(e_2)  =  e_1,\; e_2 \lhd e_2  = e_2 \cdot P(e_2)  =  e_2. \label{exfordias}
\end{align}

(2) Let  $M(n,\mathbb{F}) =\left\{\mathsf{A}=\left(a_{i j}\right) \mid a_{i j} \in \mathbb{F}, i, j=1, \ldots, n\right\}$ be the $n^2$-dimensional matrix vector space. It is known that $M(n,\mathbb{F})$ is an associative algebra under matrix multiplication. Let $P : M(n,\mathbb{F}) \to M(n,\mathbb{F})$ be the diagonal projection which is given by
\begin{align*}
P((a_{ij}))   =\operatorname{diag}\left(a_{11}, a_{22}, \ldots, a_{n n}\right).   
\end{align*}
Then  $P$ is an averaging operator on $M(n,\mathbb{F})$. Thus by Proposition \ref{avonass}, there is a $n^2$-dimensional dialgebra $(M(n,\mathbb{F}), \rhd,\lhd)$ whose the products of $\rhd,\lhd$ are explicitly given by
\begin{align*}
 (a_{ij}) \rhd (b_{ij})  = (a_{ii}b_{ij}),   \quad  (a_{ij}) \lhd (b_{ij})  = (a_{ij}b_{jj}).
\end{align*}
\end{ex}

Moreover, the tensor product of a perm algebra and an associative algebra admits a dialgebra structure. Recall that a  {\bf perm algebra} is a pair  $(A,  \circ)$ where $A$ is a vector space with a binary operation  $\circ: A \otimes A \to A$ such that  
\begin{eqnarray}
 (x \circ y) \circ z = x \circ (y \circ z)   = y \circ (x \circ z), \quad \forall x,y,z \in A. \label{perm}
\end{eqnarray}

\begin{pro}\label{ifdualpp}\cite{L4}
 Let $(A,\circ)$ be a perm algebra and $(B, \cdot)$ be an associative algebra. Then $(A \otimes B , \rhd_{A \otimes B},\lhd_{A \otimes B})$ is a dialgebra where $\rhd_{A \otimes B},\lhd_{A \otimes B}$ are respectively  given by
 \begin{align}
 (x \otimes a) \rhd_{A \otimes B} (y \otimes b) &=  x \circ y  \otimes  a \cdot b,\\
 (x \otimes a) \lhd_{A \otimes B} (y \otimes b)&=   y \circ x  \otimes  a \cdot b,\quad \forall x,y \in A, a,b \in B.
\end{align} 
\end{pro}

We provide the following example to illustrate the above construction explicitly.

\begin{ex}\label{infdualppa}
Let $A=\mathbb{F}\left[t^{ \pm}\right] \partial_{1} \oplus \mathbb{F}\left[t^{ \pm}\right] \partial_{2}$, and define a  binary   operation   $\star: A \otimes A \rightarrow A $ by
$$
 t^{i} \partial_{m}  \star t^{j} \partial_{n} =  t^{i+j+\delta_{m,1}}\partial_{n},  \quad \forall i, j \in \mathbb{Z},\; m,n \in \{1,2\}.
$$
Then $ (A, \star)$  is a   perm algebra. Moreover, let  $B= M(n,\mathbb{F})$ be the  matrix algebra.  By Proposition \ref{ifdualpp}, there is an  infinite-dimensional dialgebra $ (A \otimes B, \rhd_{A \otimes B},\lhd_{A \otimes B})$ with $\rhd_{A \otimes B},\lhd_{A \otimes B}$ respectively given by 
\begin{align*}
(t^{i} \partial_{m} \otimes \mathsf{A}) \rhd_{A \otimes B} (t^{j} \partial_{n} \otimes \mathsf{B}) &=  t^{i+j+\delta_{m,1}}\partial_{n} \otimes \mathsf{A}\mathsf{B},\\
(t^{i}\partial_{m}  \otimes \mathsf{A}) \lhd_{A \otimes B} (t^{j}\partial_{n}  \otimes \mathsf{B}) &= t^{i+j+\delta_{n,1}}\partial_{m}  \otimes \mathsf{A}\mathsf{B},
\end{align*}
 for all $i, j\in \mathbb{Z},\; m,n \in \{1,2\},\; \mathsf{A},\mathsf{B} \in B$.
\end{ex}

It is well known that the operad of dialgebras is the duplicator of the operad of associative algebras \cite{JBGN}.  Consequently, one obtains a natural construction of dialgebras from bimodules of associative algebras.

 Recall \cite{B2} that a \textbf{bimodule} of an associative algebra $(A,\cdot)$ is a triple $(V;l,r)$  where $V$ is a vector space and  $l,r: A \rightarrow \mathrm{End}(V)$ are linear maps satisfying 
\begin{align*}
 l(x\cdot y)  =l(x) l(y)  , \quad r(x\cdot y)  =r(y) r(x) , \quad l(x) r(y)  =r(y) l(x) , \quad \forall x,y \in A.
\end{align*}
 Then $(V^*;r^*,l^*)$ is an  bimodule of $(A,\cdot)$, called the {\bf dual bimodule} of $(V;l,r)$.

\begin{pro}\label{difromrepAss}
 Let $(A,\cdot)$ be an associative  algebra and $(V;l,r)$ be a bimodule of $(A,\cdot)$. Then there is a dialgebra structrue on the direct sum  $A \oplus V$  of vector spaces given by
 \begin{align}
(x+u) \rhd_{A \oplus V}(y+v) &:=x \cdot  y+ l(x) v  ,\\
(x+u) \lhd_{A \oplus V}(y+v)&:=x \cdot y+ r(y) u, \quad \forall x, y \in A, u, v \in V.
\end{align}
We denote this dialgebra $(A \oplus V,\rhd_{A \oplus V}, \lhd_{A \oplus V})$ by $A \rightthreetimes_{l,r} V$. Moreover, there is another dialgebra $A \rightthreetimes_{r^*,l^*} V^*$  by considering the dual bimodule $(V^*;r^*,l^*)$ of  $(V;l,r)$.
\end{pro}
\begin{proof}
It is straightforward to check that $(A \oplus V, \rhd_{A \oplus V})$  and $(A \oplus V, \lhd_{A \oplus V})$ are associative  algebras. Moreover, for all $x, y,z \in A, u, v,w \in V$, we have
\begin{align*}
&(x +u) \lhd_{A\oplus V} \big((y+v)\lhd_{A\oplus V} (z+w)\big)  - (x +u) \lhd_{A\oplus V} \big((y+v) \rhd_{A\oplus V} (z+w)\big)\\
&=(x +u) \lhd_{A\oplus V} \big(y \cdot z + r(z)v\big)  - (x +u) \lhd_{A\oplus V} \big(y \cdot z + l(y)w\big)\\
&= x\cdot  ( y \cdot z) + r(y \cdot z)u   -x\cdot  ( y \cdot z) - r(y \cdot z)u =0,\\
&(x +u) \rhd_{A\oplus V} \big((y+v)\lhd_{A\oplus V} (z+w)\big)  - \big((x +u) \rhd_{A\oplus V} (y+v)\big)\lhd_{A\oplus V} (z+w) \\
&=(x +u) \rhd_{A\oplus V} \big(y \cdot z + r(z)v\big)  -  \big(x \cdot y + l(x)v\big)\lhd_{A\oplus V} (z +w)\\
&= x\cdot  ( y \cdot z) + l(x)r(  z)v   -(x\cdot y) \cdot z  - r(  z)l(x)v =0,\\
&\big((x +u) \lhd_{A\oplus V} (y+v)\big)\rhd_{A\oplus V} (z+w)  - \big((x +u) \rhd_{A\oplus V} (y+v)\big)\rhd_{A\oplus V} (z+w) \\
&=  \big(x \cdot y + r(y)u\big)\rhd_{A\oplus V} (z +w) -  \big(x \cdot y + l(x)v\big)\rhd_{A\oplus V} (z +w)\\
&= (x\cdot y) \cdot z + l(x\cdot y)w   -(x\cdot y) \cdot z  - l(x\cdot y)w  =0.
\end{align*}
Thus $(A \oplus V,\rhd_{A \oplus V}, \lhd_{A \oplus V})$ is a  dialgebra.
\end{proof}

\begin{ex}\label{amaExspe}
(1) Let $(A, \cdot)$ be an associative algebra.  Then $(A;  L_{\cdot},R_{\cdot})$ and $(A;  -R_{\cdot}^*,-L_{\cdot}^*)$ are   bimodules of  $(A, \cdot)$ \cite{B2}. Thus by Proposition \ref{difromrepAss}, there are two dialgebras $A \rightthreetimes_{L_{\cdot},R_{\cdot}} A$ and $A \rightthreetimes_{-R_{\cdot}^*,-L_{\cdot}^*} A^*$ on the direct sum  $A \oplus A$ and $A \oplus A^*$ of vector spaces, respectively.
 
(2) In Example \ref{Ex:dias} (1), let $\{e_1^*, e_2^*\}$ be the dual basis of $\{e_1, e_2\}$. By Proposition \ref{difromrepAss}, there is a 4-dimensional 
 dialgebra $A \rightthreetimes_{-R_{\cdot}^*,-L_{\cdot}^*} A^*$ explicitly given by 
 \begin{align}
 e_1 \rhd e_2 &= e_1,  \quad e_1 \lhd e_2  = e_1,\quad   e_2 \rhd e_2 = e_2 ,\quad e_2 \lhd e_2  = e_2,\label{ex:diasexZ1}\\
 e_2 \rhd  e_1^* &=e_1^*,\quad  e_2 \rhd  e_2^*  =e_2^*,\quad   e_1^* \lhd  e_1 =e_2^*,\quad e_2^* \lhd  e_2 =e_2^*.\label{ex:diasexZ2}
\end{align}
\end{ex}


{}

\subsection{Bimodules of dialgebras and dual bimodules}

In this subsection, we introduce the notion  of bimodules of dialgebras.

\begin{defi} \label{defrep} A \textbf{bimodule} of a dialgebra $(A,\rhd,\lhd)$  is a quintuple $(V;l_{\rhd}, r_{\rhd}, l_{\lhd},  r_{\lhd})$ where $(V;l_{\rhd}, r_{\rhd})$  and $(V;l_{\lhd},  r_{\lhd})$ are bimodules of associative algebras $(A,\rhd)$ and $(A,\lhd)$ respectively  satisfying  the following equations
\begin{align}
r_{\lhd}(x \lhd y) & =r_{\lhd}(x \rhd y), &&l_{\lhd}(x)r_{\lhd}(y)  =l_{\lhd}(x)r_{\rhd}(y), && l_{\lhd}(x)l_{\lhd}(y)  =l_{\lhd}(x)l_{\rhd}(y),  \\
r_{\lhd}(x ) r_{\rhd}(y)& =r_{\rhd}(y \lhd x), &&r_{\lhd}(x)l_{\rhd}(y)  =l_{\rhd}(y)r_{\lhd}(x), && l_{\lhd}(x \rhd y)  =l_{\rhd}(x)l_{\lhd}(y),  \\
r_{\rhd}(x ) r_{\lhd}(y)& =r_{\rhd}(x )r_{\rhd}(y), &&r_{\rhd}(x)l_{\lhd}(y)  =r_{\rhd}(x)l_{\rhd}(y), && l_{\rhd}(x \lhd y)  =l_{\rhd} (x \rhd y) .
\end{align}
Bimodules  $(V_1;l_{\rhd_1}, r_{\rhd_1}, l_{\lhd_1}, r_{\lhd_1})$  and $(V_2;l_{\rhd_2}, r_{\rhd_2}, l_{\lhd_2}, r_{\lhd_2})$ of a dialgebra $(A,\rhd,\lhd)$ are
\textbf{equivalent} if there exists a linear isomorphism $\varphi:
V_{1} \rightarrow V_{2}$  such that
\begin{align*}
\varphi\left(l_{\rhd_1}(x) v\right) &=l_{\rhd_2}(x) \varphi(v),&&
\varphi\left(r_{\rhd_1}(x) v\right)=r_{\rhd_2}(x) \varphi(v), \\
\varphi\left(l_{\lhd_1}(x) v\right) &=l_{\lhd_2}(x) \varphi(v),&&\varphi\left(r_{\lhd_1}(x) v\right)=r_{\lhd_2}(x) \varphi(v),
\;\;\forall x \in A, v \in V_{1}.
\end{align*}
\end{defi}

It is straightforward to  obtain the following conclusion.
\begin{pro} \label{repofdppa}
Let  $(A,\rhd,\lhd)$  be a dialgebra. Let $V$  be a vector space and  $l_{\rhd}, r_{\rhd}, l_{\lhd}, r_{\lhd}: A \rightarrow \End(V)$ be linear maps. Then $(V;l_{\rhd}, r_{\rhd}, l_{\lhd}, r_{\lhd})$ is a bimodule of $(A, \rhd,\lhd)$ if and only if the direct sum  $A \oplus V$ of vector spaces is turned into a  dialgebra (the semidirect sum) by defining multiplications in $A \oplus V$ by
\begin{align}
(x+u) \rhd_{A \oplus V}(y+v) &:=x \rhd  y+ l_{\rhd}(x) v+r_{\rhd}(y) u,\\
(x+u) \lhd_{A \oplus V}(y+v)  &:=x \lhd  y+ l_{\lhd}(x) v+r_{\lhd}(y) u, \quad \forall x, y \in A, u, v \in V.
\end{align}
We denote the dialgebra  $\left(A \oplus V,  \rhd_{A \oplus V}, \lhd_{A \oplus V} \right)$  by  $A \ltimes_{l_{\rhd}, r_{\rhd}, l_{\lhd}, r_{\lhd}} V$.
\end{pro}



\begin{lem}\label{dualbimodofdias}
Let  $(V;l_{\rhd}, r_{\rhd}, l_{\lhd}, r_{\lhd})$  be a bimodule of a dialgebra $(A,\rhd,\lhd)$.
 \begin{enumerate}
\item\label{repdi:1} $(V;l_{\rhd}, 0, l_{\lhd}, 0)$ and $(V;0,r_{\rhd},0, r_{\lhd})$   are bimodules of $(A,\rhd,\lhd)$.
\item\label{repdi:2}  $\left(V^*;-r_{\lhd}^{*}, l_{\lhd}^{*}-l_{\rhd}^{*}, r_{\rhd}^{*}-r_{\lhd}^{*},-l_{\rhd}^{*}\right)$ is a  bimodule  of $(A,\rhd,\lhd)$, called the {\bf dual bimodule} of $(V;l_{\rhd}, r_{\rhd},$ $ l_{\lhd}, r_{\lhd})$. Moreover, there is a dialgebra structrue $A \ltimes_{-r_{\lhd}^{*}, l_{\lhd}^{*}-l_{\rhd}^{*}, r_{\rhd}^{*}-r_{\lhd}^{*},-l_{\rhd}^{*}} V^*$  on the direct sum  $A \oplus V^*$ of vector spaces.
\item\label{repdi:3}  $\left(V^*;0, l_{\lhd}^{*}-l_{\rhd}^{*}, 0,-l_{\rhd}^{*}\right)$  and $\left(V^*;-r_{\lhd}^{*}, 0, r_{\rhd}^{*}-r_{\lhd}^{*},0\right)$  are   bimodules  of $(A,\rhd,\lhd)$.
 \end{enumerate}
 
\end{lem}
\begin{proof}\eqref{repdi:1}. It is straightforward.

 \eqref{repdi:2}. For all $x, y  \in A, v\in V,  w^* \in V^*$, we have
\begin{align*}
&\left \langle \Big( -r_{\lhd}^{*}(x \rhd y)  - r_{\lhd}^{*}(x)  r_{\lhd}^{*}(y) \Big)w^*,v \right \rangle   =  \left \langle w^*,\Big(  r_{\lhd} (x \rhd y)  -   r_{\lhd} (y)r_{\lhd} (x) \Big)v \right \rangle  =0,\\
&\left \langle \Big( (l_{\lhd}^{*}-l_{\rhd}^{*})(x \rhd y)  - (l_{\lhd}^{*}-l_{\rhd}^{*})(y)  (l_{\lhd}^{*}-l_{\rhd}^{*})(x) \Big)w^*,v \right \rangle   \\
&=  \left \langle w^*,\Big( (l_{\rhd} -l_{\lhd} )(x \rhd y)  -   (l_{\lhd} -l_{\rhd} )(x)(l_{\lhd} -l_{\rhd} )(y) \Big)v \right \rangle \\
&=  \left \langle w^*,\Big(   -l_{\lhd}  (x \rhd y)  -    l_{\lhd}(x)l_{\lhd}(y) + l_{\lhd}(x)l_{\rhd}(y) + l_{\rhd}(x)l_{\lhd}(y)   \Big)v \right \rangle =0,\\
&\left \langle \Big(-r_{\lhd}^{*}(x)  (l_{\lhd}^{*}-l_{\rhd}^{*})(y)  +(l_{\lhd}^{*}-l_{\rhd}^{*})(y) r_{\lhd}^{*}(x) \Big)w^*,v \right \rangle\\
   &=  \left \langle w^*, \Big(-  (l_{\lhd}-l_{\rhd})(y)r_{\lhd} (x)  + r_{\lhd} (x)(l_{\lhd} -l_{\rhd} )(y) \Big)v \right \rangle\\
      &=  \left \langle w^*, \Big(l_{\rhd} (y)r_{\lhd} (x)  - r_{\lhd} (x)   l_{\rhd}  (y) \Big)v \right \rangle =0.
\end{align*}
Thus $\left(V^*;-r_{\lhd}^{*}, l_{\lhd}^{*}-l_{\rhd}^{*}\right)$ is a  bimodule  of associative algebra  $(A,\rhd)$. Moreover, we have
\begin{align*}
&\left \langle \Big( (r_{\rhd}^{*}-r_{\lhd}^{*})(x \lhd y)  - (r_{\rhd}^{*}-r_{\lhd}^{*})(x)  (r_{\rhd}^{*}-r_{\lhd}^{*})(y) \Big)w^*,v \right \rangle   \\
&=  \left \langle w^*,\Big( (r_{\lhd}-r_{\rhd})(x \lhd y)  -   (r_{\rhd} -r_{\lhd} )(y)(r_{\rhd} -r_{\lhd} )(x) \Big)v \right \rangle   \\
&=  \left \langle w^*,\Big( (r_{\lhd}-r_{\rhd})(x \lhd y)  -    r_{\rhd}(y)(r_{\rhd}(x) -r_{\lhd}(x))  +r_{\lhd}(y)  (r_{\rhd}(x) -r_{\lhd}(x)) \Big)v \right \rangle   \\
&=  \left \langle w^*,\Big(  -r_{\rhd} (x \lhd y)  -    r_{\rhd}(y)(r_{\rhd}(x) -r_{\lhd}(x))  +r_{\lhd}(y)   r_{\rhd}(x)   \Big)v \right \rangle =0,   \\
&\left \langle \Big( -l_{\rhd}^{*}(x \lhd y)  -  l_{\rhd}^{*}(y)   l_{\rhd}^{*}(x) \Big)w^*,v \right \rangle = \left \langle w^*,\Big(  l_{\rhd} (x \lhd y)  -    l_{\rhd} (x)l_{\rhd} (y)  \Big)v \right \rangle   =0, \\
&\left \langle \Big(-(r_{\rhd}^{*}-r_{\lhd}^{*})(x)   l_{\rhd}^{*}(y)  + l_{\rhd}^{*}(y) (r_{\rhd}^{*}-r_{\lhd}^{*})(x) \Big)w^*,v \right \rangle \\
&=\left \langle w^*,\Big(   -l_{\rhd} (y)(r_{\rhd} -r_{\lhd} )(x)  +  (r_{\rhd} -r_{\lhd} )(x) l_{\rhd} (y)\Big)v \right \rangle =0.
\end{align*}
Thus $\left(V^*;r_{\rhd}^{*}-r_{\lhd}^{*},-l_{\rhd}^{*}\right)$ is a  bimodule  of associative algebra  $(A,\lhd)$. Futhermore,
\begin{align*}
&\left \langle \Big( -l_{\rhd}^{*}(x \lhd y) + l_{\rhd}^{*}(x \rhd y) \Big)w^*,v \right \rangle  =  \left \langle w^*, \Big(  l_{\rhd} (x \lhd y) - l_{\rhd} (x \rhd y) \Big)v \right \rangle  =0,\\
&\left \langle \Big(-(r_{\rhd}^{*}-r_{\lhd}^{*})(x) l_{\rhd}^{*}(y)  -(r_{\rhd}^{*}-r_{\lhd}^{*})(x)(l_{\lhd}^{*}-l_{\rhd}^{*})(y) \Big)w^*,v \right \rangle   \\
&= \left \langle \Big( -  (r_{\rhd}^{*}-r_{\lhd}^{*})(x) l_{\lhd}^{*}  (y) \Big)w^*,v \right \rangle  = \left \langle w^*,\Big( -  l_{\lhd}  (y)(r_{\rhd} -r_{\lhd} )(x)  \Big)v \right \rangle =0,  \\
&\left \langle \Big(  (r_{\rhd}^{*}-r_{\lhd}^{*})(x)(r_{\rhd}^{*}-r_{\lhd}^{*})(y)  +(r_{\rhd}^{*}-r_{\lhd}^{*})(x) r_{\lhd}^{*}(y) \Big)w^*,v \right \rangle   \\
&=\left \langle \Big(  (r_{\rhd}^{*}-r_{\lhd}^{*})(x) r_{\rhd}^{*} (y)    \Big)w^*,v \right \rangle =\left \langle w^*,\Big(  r_{\rhd} (y)  (r_{\rhd} -r_{\lhd} )(x)   \Big)v \right \rangle =0, \\
&\left \langle \Big(-l_{\rhd}^{*}(x ) (l_{\lhd}^{*}-l_{\rhd}^{*})(y)-(l_{\lhd}^{*}-l_{\rhd}^{*})(y \lhd x)\Big)w^*,v \right \rangle  \\
&= \left \langle w^*,\Big(- (l_{\lhd} -l_{\rhd} )(y)l_{\rhd} (x )+(l_{\lhd} -l_{\rhd} )(y \lhd x)\Big)v \right \rangle =0, \\
&\left \langle \Big( l_{\rhd}^{*}(x) r_{\lhd}^{*}(y) -r_{\lhd}^{*}(y) l_{\rhd}^{*}(x)\Big)w^*,v \right \rangle = \left \langle w^*,\Big(  r_{\lhd} (y)l_{\rhd} (x) - l_{\rhd} (x)r_{\lhd} (y)\Big)v \right \rangle =0,\\
&\left \langle \Big((r_{\rhd}^{*}-r_{\lhd}^{*})(x \rhd y)  +r_{\lhd}^{*}(x)(r_{\rhd}^{*}-r_{\lhd}^{*})(y)\Big)w^*,v \right \rangle\\
&=\left \langle w^*,\Big(-(r_{\rhd} -r_{\lhd} )(x \rhd y)  +(r_{\rhd} -r_{\lhd} )(y)r_{\lhd} (x)\Big)v \right \rangle = 0,\\
&\left \langle \Big(-(l_{\lhd}^{*}-l_{\rhd}^{*})(x ) l_{\rhd}^{*}(y)-(l_{\lhd}^{*}-l_{\rhd}^{*})(x )(l_{\lhd}^{*}-l_{\rhd}^{*})(y)\Big)w^*,v \right \rangle\\
&=\left \langle \Big(-(l_{\lhd}^{*}-l_{\rhd}^{*})(x )l_{\lhd}^{*}(y)\Big)w^*,v \right \rangle = \left \langle w^*,\Big(-l_{\lhd} (y)(l_{\lhd} -l_{\rhd} )(x )\Big)v \right \rangle=0,\\
&\left \langle \Big((l_{\lhd}^{*}-l_{\rhd}^{*})(x)(r_{\rhd}^{*}-r_{\lhd}^{*})(y)  +(l_{\lhd}^{*}-l_{\rhd}^{*})(x)r_{\lhd}^{*}(y)\Big)w^*,v \right \rangle =\left \langle \Big((l_{\lhd}^{*}-l_{\rhd}^{*})(x) r_{\rhd}^{*}  (y) \Big)w^*,v \right \rangle \\
&= \left \langle w^*,\Big( r_{\rhd}  (y) (l_{\lhd} -l_{\rhd} )(x)\Big)v \right \rangle=0,\\
&\left \langle \Big(-r_{\lhd}^{*}(x \lhd y) +r_{\lhd}^{*}(x \rhd y)\Big)w^*,v \right \rangle = \left \langle w^*,\Big( r_{\lhd} (x \lhd y)-r_{\lhd} (x \rhd y)\Big)v \right \rangle =0. 
\end{align*}
Thus $\left(V^*;-r_{\lhd}^{*}, l_{\lhd}^{*}-l_{\rhd}^{*}, r_{\rhd}^{*}-r_{\lhd}^{*},-l_{\rhd}^{*}\right)$ is a  bimodule  of $(A,\rhd,\lhd)$. The remaining statement follows from Proposition \ref{repofdppa}.

\eqref{repdi:3}. It follows from \eqref{repdi:1} and \eqref{repdi:2}.
\end{proof}

\begin{ex}\label{coregularrep}
(1) Let $(A,\rhd,\lhd)$ be a dialgebra.   Then $(A;L_{\rhd}, R_{\rhd}, L_{\lhd}, R_{\lhd})$ and $(A^*;-R_{\lhd}^{*}, $ $L_{\lhd}^{*}-L_{\rhd}^{*}, R_{\rhd}^{*}-R_{\lhd}^{*},-L_{\rhd}^{*})$ are  bimodules of $(A,\rhd,\lhd)$. The former is called a {\bf regular bimodule} of $(A,\rhd,\lhd)$, and the latter is called a {\bf coregular bimodule} of it. In this case,   there is a dialgebra structrue $A \ltimes_{-R_{\lhd}^{*},  L_{\lhd}^{*}-L_{\rhd}^{*}, R_{\rhd}^{*}-R_{\lhd}^{*},-L_{\rhd}^{*}} A^*$ on the direct sum  $A \oplus A^*$  of vector spaces. Moreover,  $(A;L_{\rhd}, 0, L_{\lhd}, 0)$, $(A;0,R_{\rhd},0, R_{\lhd})$, $(A^*;0, L_{\lhd}^{*}-L_{\rhd}^{*}, $ $0,-L_{\rhd}^{*})$  and $(A^*;-R_{\lhd}^{*}, 0, R_{\rhd}^{*}-R_{\lhd}^{*},0)$ are  bimodules of $(A,\rhd,\lhd)$, too.  

(2) In Example \ref{Ex:dias} (1), let $\{e_1^*, e_2^*\}$ be the dual basis of $\{e_1, e_2\}$. Then there is a 4-dimensional  dialgebra $A
\ltimes_{-R_{\lhd}^{*},  L_{\lhd}^{*}-L_{\rhd}^{*}, R_{\rhd}^{*}-R_{\lhd}^{*},-L_{\rhd}^{*}}
A^*$ explicitly given by  Eq.~\eqref{exfordias} and the following equations 
 \begin{align}
 e_2 \rhd e_1^*& =e_1^*,\quad  e_2 \rhd e_2^* =e_2^*,\quad  e_2 \lhd e_1^* =e_1^*-e_2^*,\quad  e_1^* \rhd e_1 =-e_2^*,\label{ex:diasexZ3}\\
 e_1^* \rhd e_2&=e_1^*,\quad  e_1^* \lhd e_2 =e_2^*,\quad e_2^* \lhd e_2 =e_2^*.\label{ex:diasexZ4}
\end{align}
\end{ex}

\subsection{Frobenius dialgebras and their constructions
}

There is an important   bilinear form  on a dialgebra given as follows.

 \begin{defi}\label{Frbendi}
A bilinear form $ \mathfrak{B} : A \otimes A \to \mathbb{F}$ on a dialgebra $(A,  \rhd,\lhd)$ is called {\bf invariant} if
\begin{align}
\mathfrak{B}(x \rhd y, z)&=\mathfrak{B}(x,y \rhd z-y \lhd z ), \quad 
\mathfrak{B}(x \lhd y, z)= \mathfrak{B}( x,y \rhd z), \quad \forall x,y,z \in A.\label{diinv1}
\end{align}
A {\bf (skew-symmetric) Frobenius dialgebra} $(A,  \rhd,\lhd,\mathfrak{B})$ is a  dialgebra $(A,  \rhd,\lhd)$ with a nondegenerate (skew-symmetric) invariant bilinear form $\mathfrak{B}$. 
\end{defi}

\begin{rmk}The  Eq.~\eqref{diinv1} first appeared in \cite{CF1} as the invariance condition for bilinear forms on dialgebras. This is motivated by the fact that the operad of associative algebras is cyclic, whereas the operad of dialgebras is anticyclic. Indeed, as observed by Chapoton in \cite{CF1} via operad theory, the appropriate invariant bilinear form on a dialgebra is precisely the skew-symmetric one mentioned above.
\end{rmk}

\begin{pro}\label{symFrodi}
Let $(A,  \rhd,\lhd,\mathfrak{B})$ be a Frobenius dialgebra. If $\mathfrak{B}$ is symmetric, then $(A,  \rhd,\lhd)$ is  a trivial dialgebra, that is, $x \rhd y = x \lhd y =0$ for all $x,y \in A$.
\end{pro}
 \begin{proof}
For all $x, y, z \in A$, we have
\begin{align*}
\mathfrak{B}(x \rhd y, z) &=\mathfrak{B}(x,y \rhd z-y \lhd z ) = \mathfrak{B}(x \lhd  y,z ) - \mathfrak{B}(x, y \lhd z )  
 = \mathfrak{B}(x \lhd  y,z ) - \mathfrak{B}( y,z \rhd x )  \\
 & = \mathfrak{B}(x \lhd  y,z ) - \mathfrak{B}( z,x \rhd y - x \lhd y ).
\end{align*}
With the nondegeneracy of $\mathfrak{B}$ and Eq.~\eqref{diinv1}, we have $x \rhd y = x \lhd y = 0$ for all $x,y \in A$. Thus $(A,  \rhd,\lhd)$ is  a trivial dialgebra.
\end{proof}

\begin{ex} In Example \ref{dwdpp} (3), let $\mathfrak{B}$ be a  bilinear form on $A:= \mathbb{F}\left[t^{ \pm}\right] \partial_{1} \oplus \mathbb{F}\left[t^{ \pm}\right] \partial_{2} $ defined by 
 \begin{align}
\mathfrak{B}\left(t^{i}  \partial_{1}, t^{j}   \partial_{2}\right) & =-\mathfrak{B}\left(t^{j}  \partial_{2}, t^{i}   \partial_{1}\right)=\delta_{i+j, 0},\quad
\mathfrak{B}\left(t^{i}   \partial_{1}, t^{j}  \partial_{1}\right)   =\mathfrak{B}\left(t^{i}   \partial_{2}, t^{j}   \partial_{2}\right)=0, \quad \forall i, j \in \mathbb{Z}.
\end{align}
 Then  $(A,  \rhd,\lhd,\mathfrak{B}) $ is a  skew-symmetric Frobenius dialgebra with $\rhd, \lhd$ given by  Eq.~\eqref{ex:diainfinix}.
\end{ex}

There are two natural constructions of  skew-symmetric Frobenius dialgebras from associative algebras and dialgebras  respectively.

\begin{pro}\label{p:ppplgphpl}
 Let $A$ be a vector space and $\mathfrak{B}_{d}$ be the  bilinear form on the direct sum  $A \oplus A^*$  of vector spaces  defined by 
\begin{align}
\mathfrak{B}_d\left(x+a^{*}, y+b^{*}\right)=  \left\langle  y,a^{*}\right\rangle - \left\langle x, b^{*}\right\rangle, \quad  \forall x, y \in A, a^{*}, b^{*} \in A^{*}. \label{mtorbl}
\end{align}
 \begin{enumerate}
 \item\label{sfdifromAs} If $(A,\cdot)$ is an associative  algebra, then $( A \rightthreetimes_{-R_{\cdot}^*,-L_{\cdot}^*} A^* ,\mathfrak{B}_d)$ is a  skew-symmetric Frobenius dialgebra.
 \item\label{sfdifromdi} If $(A,\rhd,\lhd)$ is a  dialgebra, then $(A
\ltimes_{-R_{\lhd}^{*},  L_{\lhd}^{*}-L_{\rhd}^{*}, R_{\rhd}^{*}-R_{\lhd}^{*},-L_{\rhd}^{*}}
A^*,\mathfrak{B}_d)$ is a  skew-symmetric Frobenius dialgebra.
  \end{enumerate}
 \end{pro}
 
\begin{proof} \eqref{sfdifromAs}. Let $x,y,z \in A, a^*,b^*,c^* \in A^*$. Then we have
\begin{align*}
&\mathfrak{B}_d((x +a^*)\rhd (y +b^*), z +c^*) 
 \\
 &=\mathfrak{B}_d( x \cdot y    -R_{\cdot}^{*}(x)b^*, z +c^*) 
 =  \left\langle z,   -R_{\cdot}^{*}(x)b^*   \right\rangle  - \left\langle x \cdot y   , c^*   \right\rangle    =  \left\langle x,   -L_{\cdot}^{*}(z)b^*   \right\rangle  + \left\langle x     ,  R_{\cdot}^{*}(y)c^*   \right\rangle  \\
 &=\mathfrak{B}_d( x +a^* ,y \cdot z   -R_{\cdot}^{*}(y)c^*   -y \cdot z  +L_{\cdot}^{*}(z)b^*    )\\
       &=\mathfrak{B}_d( x +a^* ,(y +b^*)\rhd ( z +c^*) - (y +b^*)\lhd ( z +c^*)),\\
       &\mathfrak{B}_d((x +a^*)\lhd (y +b^*), z +c^*) 
 =\mathfrak{B}_d(  x \cdot y -L_{\cdot}^{*} (y)a^*   , z +c^*) 
 =\langle  z , -L_{\cdot}^{*} (y)a^*\rangle   -  \langle   x \cdot y,c^* \rangle  \\
&=\langle  y \cdot z ,  a^*\rangle  +  \langle   x ,R_{\cdot}^{*} (y)c^* \rangle   =\mathfrak{B}_d( x +a^*,y \cdot z -R_{\cdot}^{*} (y)c^*   )\\
&=\mathfrak{B}_d( x +a^*,(y +b^*) \rhd (z +c^*)).
\end{align*}
Thus $\mathfrak{B}_d$ is invariant on $ A \rightthreetimes_{-R_{\cdot}^*,-L_{\cdot}^*} A^*$. Thus $( A \rightthreetimes_{-R_{\cdot}^*,-L_{\cdot}^*} A^* ,\mathfrak{B}_d)$ is a  skew-symmetric Frobenius dialgebra.

\eqref{sfdifromdi}. Let $x,y,z \in A, a^*,b^*,c^* \in A^*$. Then we have
\begin{align*}
&\mathfrak{B}_d((x +a^*)\rhd (y +b^*), z +c^*)\\
 &=\mathfrak{B}_d( x \rhd y    -R_{\lhd}^{*}(x)b^*+(L_{\lhd}^{*}-L_{\rhd}^{*})(y)a^*, z +c^*)\\
 &=  \langle  z, -R_{\lhd}^{*}(x)b^*+(L_{\lhd}^{*}-L_{\rhd}^{*})(y)a^* \rangle - \langle  x \rhd y,c^* \rangle    \\
  &=  \langle  x, -L_{\lhd}^{*}(z)b^*\rangle +\langle  y \rhd z - y \lhd z,  a^* \rangle - \langle  x  ,-R_{\rhd}^{*}(y)c^* \rangle    \\
    &=   \mathfrak{B}_d( x +a^* , y \rhd z - y \lhd z + L_{\lhd}^{*}(z)b^*-R_{\rhd}^{*}(y)c^* )\\
        &=   \mathfrak{B}_d( x +a^* , y \rhd z - y \lhd z + (L_{\lhd}^{*}-L_{\rhd}^{*})(z)b^*-R_{\lhd}^{*}(y)c^* -(R_{\rhd}^{*}-R_{\lhd}^{*})(y)c^* +   L_{\rhd}^{*} (z)b^*  )\\
       &=\mathfrak{B}_d( x +a^* ,(y +b^*)\rhd ( z +c^*) - (y +b^*)\lhd ( z +c^*)),\\
&\mathfrak{B}_d((x +a^*)\lhd (y +b^*), z +c^*)\\
&=\mathfrak{B}_d(  x \lhd y +(R_{\rhd}^{*}-R_{\lhd}^{*})(x)b^* -  L_{\rhd}^{*} (y)a^*  , z +c^*)\\
&=\langle  z , (R_{\rhd}^{*}-R_{\lhd}^{*})(x)b^* -  L_{\rhd}^{*} (y)a^* \rangle   -  \langle  x \lhd y,c^* \rangle  \\
&=\langle  x , (L_{\rhd}^{*}-L_{\lhd}^{*})(z)b^*  \rangle + \langle y \rhd z ,     a^* \rangle    + \langle  x  ,R_{\lhd}^{*} (y)c^* \rangle  \\
&=\mathfrak{B}_d( x +a^*,y \rhd z +(L_{\lhd}^{*}-L_{\rhd}^{*})(z)b^* - R_{\lhd}^{*} (y)c^*   )\\
&=\mathfrak{B}_d( x +a^*,(y +b^*) \rhd (z +c^*)).
\end{align*}
Thus $\mathfrak{B}_d$ is invariant on $A
\ltimes_{-R_{\lhd}^{*},  L_{\lhd}^{*}-L_{\rhd}^{*}, R_{\rhd}^{*}-R_{\lhd}^{*},-L_{\rhd}^{*}}
A^*$. Then $(A
\ltimes_{-R_{\lhd}^{*},  L_{\lhd}^{*}-L_{\rhd}^{*}, R_{\rhd}^{*}-R_{\lhd}^{*},-L_{\rhd}^{*}}
A^*,\mathfrak{B}_d)$ is a  skew-symmetric Frobenius dialgebra.
\end{proof}

\begin{ex}(1) Continuing with Example \ref{amaExspe} (2),  there is an  invariant bilinear form $\mathfrak{B}_d$ on the 4-dimensional dialgebra  $ A \rightthreetimes_{-R_{\cdot}^*,-L_{\cdot}^*} A^*$  given by 
\begin{align}\label{Exforininf}
	\mathfrak{B}_d(e_i, e_j) = \mathfrak{B}_d(e_i^*, e_j^*) = 0, \quad \mathfrak{B}_d(e_i^*, e_j) = -\mathfrak{B}_d(e_i, e_j^*) = \begin{cases}
		1, i=j\\
		0, i \neq j
	\end{cases}   i, j \in  \{1, 2\}, 
\end{align}
such that $( A \rightthreetimes_{-R_{\cdot}^*,-L_{\cdot}^*} A^* ,\mathfrak{B}_d)$ is a  skew-symmetric Frobenius dialgebra.

(2) Continuing with Example \ref{coregularrep} (2),  there is  another 4-dimensional skew-symmetric Frobenius dialgebra $(A
\ltimes_{-R_{\lhd}^{*},  L_{\lhd}^{*}-L_{\rhd}^{*}, R_{\rhd}^{*}-R_{\lhd}^{*},-L_{\rhd}^{*}}
A^*,\mathfrak{B}_d)$ with $\mathfrak{B}_d$  explicitly given by Eq.~\eqref{Exforininf}.
\end{ex}

Recall \cite{BYZ} that a {\bf quadratic perm algebra} is a triple $(A, \star,\mathfrak{B})$ where $(A,  \star )$ is a perm algebra and $\mathfrak{B}$ is a nondegenerate skew-symmetric  bilinear form satisfying the following conditions
 \begin{align*}
 \mathfrak{B}(x \star y,z) =   \mathfrak{B}(x , y \star z - z \star y),\quad \forall x,y,z \in A.
 \end{align*}
 
 Recall that a {\bf (symmetric) Frobenius  algebra} $(A,  \cdot,\mathfrak{B})$ is an  associative algebra $(A,\cdot)$ with a nondegenerate (symmetric)  bilinear form $\mathfrak{B}$ satisfying the following invariant condition
\begin{align}
\mathfrak{B}(x \cdot y, z)&= \mathfrak{B}( x,y \cdot z), \quad \forall x,y,z \in A.\label{Assinv}
\end{align}

Next we extend the construction of dialgebras in Proposition \ref{ifdualpp} to the Frobenius  case.
 
\begin{pro}\label{ifqdualpp}
 Let $(A,\circ,\mathfrak{B}_A)$ be a quadratic perm algebra and $(B, \cdot ,\mathfrak{B}_B)$ be a  symmetric  Frobenius  algebra. Then $(A \otimes B , \rhd_{A \otimes B},\lhd_{A \otimes B},\mathfrak{B}_{A \otimes B})$ is a skew-symmetric  Frobenius  dialgebra where $\rhd_{A \otimes B}$ and $\lhd_{A \otimes B}$ are   given in Proposition \ref{ifdualpp} and $\mathfrak{B}_{A \otimes B}$ is defined by 
 \begin{align}
 \mathfrak{B}_{A \otimes B}( x \otimes a,  y \otimes b) &:=  \mathfrak{B}_{A }(x ,y)\mathfrak{B}_{B}(a,b), \quad \forall x,y \in A, a,b \in B.
\end{align} 
\end{pro}
\begin{proof}
By Proposition \ref{ifdualpp}, it is enough to prove that $\mathfrak{B}_{A \otimes B}$ is a nondegenerate skew-symmetric invariant bilinear form on the dialgebra $ (A \otimes B, \rhd_{A \otimes B},\lhd_{A \otimes B})$. For all $x,y,z \in A, a,b,c \in B$, we have
\begin{align*}
&\mathfrak{B}_{A \otimes B}( x \otimes a,  y \otimes b) =   \mathfrak{B}_{A }(x ,y)\mathfrak{B}_{B}(a,b) =   -\mathfrak{B}_{A }(y ,x)\mathfrak{B}_{B}(b,a) = \mathfrak{B}_{A \otimes B}( y \otimes b,  x \otimes a).
\end{align*}
Thus $\mathfrak{B}_{A \otimes B}$ is skew-symmetric and it is straightforward to verify that $\mathfrak{B}_{A \otimes B}$ is  nondegenerate. Moreover   
\begin{align*}
&\mathfrak{B}_{A \otimes B}( (x \otimes a) \rhd_{A \otimes B} (y \otimes b), z \otimes c)  =\mathfrak{B}_{A \otimes B}( (x \circ y) \otimes (a \cdot b), z \otimes c) \\
& =\mathfrak{B}_{A}(  x \circ y ,z)\mathfrak{B}_{B}( a \cdot b ,   c)  = \mathfrak{B}_{A}(x, y \circ z - z \circ y)\mathfrak{B}_{B}( a, b \cdot  c)\\
&= \mathfrak{B}_{A}(x, y \circ z  )\mathfrak{B}_{B}( a, b \cdot  c) - \mathfrak{B}_{A}(x,     z \circ y)\mathcal{B}_{B}( a, b \cdot  c)\\
&=\mathcal{B}_{A \otimes B}( x \otimes a , (y \otimes b)\rhd_{A \otimes B} (z \otimes c)  
- (y \otimes b)\lhd_{A \otimes B}(z \otimes c) ), \\
&\mathfrak{B}_{A \otimes B}( (x \otimes a) \lhd_{A \otimes B} (y \otimes b), z \otimes c)  =\mathfrak{B}_{A \otimes B}( (y \circ x) \otimes (a \cdot b), z \otimes c) \\
&= \mathfrak{B}_{A}(  y \circ x ,z)\mathfrak{B}_{B}( a \cdot b ,   c) = \mathfrak{B}_{A}(  x,y \circ z )\mathfrak{B}_{B}( a  ,   b\cdot c) = \mathfrak{B}_{A \otimes B}(  x \otimes a ,   (y \circ z ) \otimes (b \cdot c))\\
&= \mathfrak{B}_{A \otimes B}(  x \otimes a,   (y \otimes b) \rhd_{A \otimes B}     (z \otimes c) ).
\end{align*}
 That is, $\mathfrak{B}_{A \otimes B}$ is invariant on  $(A \otimes B , \rhd_{A \otimes B},\lhd_{A \otimes B})$. 
\end{proof}

\begin{ex} Continuing with Example \ref{infdualppa}, there is a  nondegenerate skew-symmetric invariant bilinear form $\mathfrak{B}_{A \otimes B}$ on the  dialgebra $ (A \otimes B, \rhd_{A \otimes B},\lhd_{A \otimes B})$ defined by
 \begin{align}
\mathfrak{B}_{A}\left(t^{i}  \partial_{1}, t^{j}   \partial_{2}\right) & =-\mathfrak{B}_{A}\left(t^{j}  \partial_{2}, t^{i}   \partial_{1}\right)=\delta_{i+j, 0},\quad
\mathfrak{B}_{A}\left(t^{i}   \partial_{1}, t^{j}  \partial_{1}\right)   =\mathfrak{B}_{A}\left(t^{i}   \partial_{2}, t^{j}   \partial_{2}\right)=0,  \\
\mathfrak{B}_{B}(\mathsf{A},\mathsf{B}) &= \mathrm{Tr}(\mathsf{A}\mathsf{B}),  \quad \forall i, j\in \mathbb{Z}, \mathsf{A},\mathsf{B} \in B.
\end{align}
 Thus   by Proposition \ref{ifqdualpp}, $(A \otimes B ,  \rhd_{A \otimes B},\lhd_{A \otimes B},\mathfrak{B}_{A \otimes B})$ is a  skew-symmetric  Frobenius dialgebra.
\end{ex}

\begin{lem}\label{invxzyil}
Let $\mathfrak{B}$ be a skew-symmetric invariant bilinear form on a dialgebra $(A, \rhd,\lhd)$. Then $\mathfrak{B}$ satisfies the following properties
\begin{align}
\mathfrak{B}(x \lhd y, z)&=\mathfrak{B}(y, z \lhd x-z \rhd x)  , \label{diinv3}\\
\mathfrak{B}(x \rhd y, z)& =\mathfrak{B}(y,z \lhd x). \label{diinv4} 
\end{align}
\end{lem}
\begin{proof}
By applying Eq.~\eqref{diinv1},  it implies that Eq.~\eqref{diinv4} holds.  For all $x, y, z \in A$, we have
\begin{align*}
\mathfrak{B}(x \lhd y, z) &= \mathfrak{B}(x,y \rhd z) = - \mathfrak{B}(y \rhd z,x) =-\mathfrak{B}( y,z \rhd x-z \lhd x)=\mathfrak{B}(y, z \lhd x-z \rhd x).
\end{align*}
This completes the proof.
\end{proof}

Recall \cite{TS} that a {\bf quadratic Leibniz algebra} is a triple $(A, [\cdot,\cdot],\mathfrak{B})$ where $(A, [\cdot,\cdot])$ is a Leibniz algebra and $\mathfrak{B}$ is a nondegenerate skew-symmetric  bilinear form satisfying the following conditions
 \begin{align*}
 \mathfrak{B}([x , y],z) =   \mathfrak{B}(x , [y ,z ] +  [z ,y]),\quad \forall x,y,z \in A.
 \end{align*}
 
 \begin{pro}\label{subjecntqla}
 Let $(A, \rhd,\lhd,\mathfrak{B})$  be a  skew-symmetric  Frobenius dialgebra and $(A, [\cdot,\cdot])$  be the  sub-adjacent Leibniz algebra  of $(A,  \rhd,\lhd)$.  Then $(A, [\cdot,\cdot],\mathfrak{B})$ is a quadratic Leibniz algebra.
 \end{pro}
\begin{proof}
By Lemma \ref{invxzyil}, for all $x, y, z \in A$, we have
\begin{align*}
 \mathfrak{B}([x , y],z)  &=  \mathfrak{B}(x \rhd y - y \lhd x,z)  =\mathfrak{B}(x,y \rhd z-y \lhd z ) - \mathfrak{B}(x, z \lhd y-z \rhd y)\\
 &=  \mathfrak{B}(x , [y ,z ] +  [z ,y]).
\end{align*}
Thus $(A, [\cdot,\cdot],\mathfrak{B})$ is a quadratic Leibniz algebra.
\end{proof}

\begin{pro}\label{dualdj:x}
 Let $(A,  \rhd,\lhd)$  be a dialgebra. If there is a nondegenerate skew-symmetric invariant bilinear form  $\mathfrak{B}$  such that $(A, \rhd,\lhd,\mathfrak{B})$ is a skew-symmetric  Frobenius dialgebra, then the two bimodules $(A;L_{\rhd}, R_{\rhd}, L_{\lhd}, R_{\lhd})$ and $(A^*;-R_{\lhd}^{*},  L_{\lhd}^{*}-L_{\rhd}^{*}, R_{\rhd}^{*}-R_{\lhd}^{*},-L_{\rhd}^{*})$  of the dialgebra $(A,  \rhd,\lhd)$ are equivalent. 
 
 Conversely, if the the two bimodules $(A;L_{\rhd}, R_{\rhd}, L_{\lhd}, R_{\lhd})$ and $(A^*;-R_{\lhd}^{*},  L_{\lhd}^{*}-L_{\rhd}^{*}, R_{\rhd}^{*}-R_{\lhd}^{*},-L_{\rhd}^{*})$  of the dialgebra $(A,  \rhd,\lhd)$ are equivalent, then there exists a nondegenerate invariant bilinear form  $\mathfrak{B}$  on $A$.
\end{pro}
\begin{proof}
With the nondegeneracy of $\mathfrak{B}$, there exists a linear isomorphism $\varphi_\mathfrak{B}: A \to A^*$  defined by
\begin{align}
\langle \varphi_\mathfrak{B}(x) ,y\rangle := \mathfrak{B}(x,y),\quad \forall x,y \in A.\label{mapinbyB}
\end{align}
By Lemma \ref{invxzyil}, for all $x, y, z \in A$, we have
\begin{align*}
\langle \varphi_\mathfrak{B}(L_{\rhd}(x)y) ,z\rangle &= \mathfrak{B}(x \rhd y,z) =\mathfrak{B}(y,z \lhd x)=   \langle -R^*_{\lhd}(x)\varphi_\mathfrak{B}(y),  z\rangle ,\\
\langle \varphi_\mathfrak{B}(R_{\rhd}(y)x) ,z\rangle &= \mathfrak{B}(x \rhd y,z) =\mathfrak{B}(x,y \rhd z-y \lhd z )
 = \langle (L_{\lhd}^*-L_{\rhd}^* )(y)\varphi_\mathfrak{B}(x), z\rangle,\\
\langle \varphi_\mathfrak{B}(L_{\lhd}(x)y) ,z\rangle &= \mathfrak{B}(x \lhd y,z)=\mathfrak{B}(y, z \lhd x-z \rhd x) = \langle (R_{\rhd}^*-R_{\lhd}^* )(x)\varphi_\mathfrak{B}(y), z\rangle,\\
\langle \varphi_\mathfrak{B}(R_{\lhd}(y)x) ,z\rangle &= \mathfrak{B}(x \lhd y,z) = \mathfrak{B}( x,y \rhd z)=   \langle  -L^*_{\rhd} (y)\varphi_\mathfrak{B}(x),  z\rangle.
\end{align*}
Thus the two bimodules $(A;L_{\rhd}, R_{\rhd}, L_{\lhd}, R_{\lhd})$ and $(A^*;-R_{\lhd}^{*},  L_{\lhd}^{*}-L_{\rhd}^{*}, R_{\rhd}^{*}-R_{\lhd}^{*},-L_{\rhd}^{*})$  of the dialgebra $(A,  \rhd,\lhd)$ are equivalent.  The converse can be proved similarly. We omit the details.
\end{proof}

Next  we introduce an invariant element for a dialgebra. Before proceeding, we list the necessary notations as follows.
 
 \begin{enumerate}
\item Let $A$ be  a vector space. Any $r \in A \otimes A$ can be identified as maps from $A^*$ to $A$, which we denote by $\widetilde{r}: A^* \to A$ and $\widetilde{r}^t: A^* \to A$, respectively and explicitly,
\begin{align*}
\left\langle \widetilde{r}\left(u^{*}\right), v^{*}\right\rangle :=\left\langle r, u^{*} \otimes v^{*}\right\rangle =:\left\langle \widetilde{r}^t\left(v^{*}\right), u^{*}\right\rangle   ,     \quad \forall u^{*}, v^{*} \in A^{*} .
\end{align*}
Note that $\widetilde{\tau(r)} =  \widetilde{r}^t$. We say that   $r \in V \otimes V$  is {\bf nondegenerate}  if the linear map  $\widetilde{r}$  is an isomorphism. 

\item Let $(A,\rhd,\lhd)$ be a dialgebra and define a binary operation $\star: A \otimes A \to A$ by
\begin{align}
x \star y = x \rhd y - x \lhd y,\quad \forall x,y \in A. \label{jianfayuns}
 \end{align}
 For all  $r \in A \otimes A$, we have
\begin{align*}
r = \frac{r + \tau(r)}{2} + \frac{r - \tau(r)}{2}: = \frac{\beta + \alpha}{2}
\end{align*}
where $\alpha = r - \tau(r)$, $\beta = r + \tau(r)$. That is, $\alpha$ and $\beta$ are the {\bf  skew-symmetric part} and {\bf symmetric part} of $r$ respectively. Note that $\widetilde{\alpha} = \widetilde{r} - \widetilde{r}^t$ and $\widetilde{\beta} = \widetilde{r} + \widetilde{r}^t$.

\item Let $A$ be a vector space and $\mathfrak{B}$  be a nondegenerate bilinear form. Denote by  $\varphi_{\mathfrak{B}}: A\rightarrow A^{*}$   the induced linear isomorphism defined by   Eq.~\eqref{mapinbyB}.  Moreover, denote by $r_{\mathfrak{B}} \in  A \otimes A$ the 2-tensor form of $\varphi_{\mathfrak{B}}^{-1}$, that is,
 \begin{align}
\langle r_{\mathfrak{B}}, a^{*} \otimes b^{*} \rangle := \langle \varphi_{\mathfrak{B}}^{-1}(a^*), b^*\rangle, \quad \forall a^{*}, b^{*} \in A^{*}.  \label{inducedr}
\end{align}
 \end{enumerate}
 
 Now we introduce the notion of  invariance of a 2-tensor  $r \in A \otimes A$, which is the main ingredient in the definition of a quasi-triangular bi-dialgebra in section \ref{Bialgebra}.

\begin{defi}\label{invele1}
Let $(A,  \rhd,\lhd)$  be a dialgebra. Define two linear maps $E, F: A \rightarrow \operatorname{End}(A \otimes A)$  by
\begin{align}
E(x)&:=    R_{\star}(x) \otimes  \mathrm{id}+ \mathrm{id} \otimes L_{\lhd}(x) , \label{eq:e}\\
F(x)&:= R_{\rhd}(x) \otimes \mathrm{id} - \mathrm{id} \otimes L_{\star}(x).\label{eq:f}
\end{align}
An element  $r \in A \otimes A$  is called {\bf invariant} if
\begin{align}
E(x) r = F(x) r=0, \quad \forall x \in A .
\end{align}
\end{defi}

Now we give another characterization of the  invariant condition.

\begin{lem}\label{semiinequ}
Let $(A,  \rhd,\lhd)$  be a dialgebra and  $r \in A \otimes A$. Let $\alpha$ be the skew-symmetric  part of $r$. Then the following conditions are equivalent:
	\begin{enumerate}
	\item \label{semi:1}   $\alpha$  is invariant.
	\item \label{semi:2}   The following equations hold:
	\begin{align*}
	 R_{\lhd}^{*}\left(\widetilde{\alpha}\left(a^{*}\right)\right) b^{*} = L_{\star}^{*}\left(\widetilde{\alpha}\left(b^{*}\right)\right) a^{*},\quad  R_{\star}^{*}\left(\widetilde{\alpha}\left(a^{*}\right)\right) b^{*} = -L_{\rhd}^{*}\left(\widetilde{\alpha}\left(b^{*}\right)\right) a^{*},\quad \forall   a^{*}, b^{*} \in A^{*}.
	\end{align*}
	\item \label{semi:3}   The following equations hold:
	\begin{align*}
	 \widetilde{\alpha}\left( L_{\lhd}^*(x) a^{*}\right) =\widetilde{\alpha}\left(a^{*}\right) \star x ,\quad  \widetilde{\alpha}\left(L_{\star}^*(x)a^{*}\right)= -\widetilde{\alpha}\left(a^{*}\right)\rhd x ,\quad  \forall x \in A,  a^{*}  \in A^{*}.
	\end{align*}
		\item \label{semi:4}    The following equations hold:
	\begin{align*}
	 \widetilde{\alpha}\left( R_{\star}^*(x) a^{*}\right) =x \lhd \widetilde{\alpha}\left(a^{*}\right)  ,\quad  \widetilde{\alpha}\left(R_{\rhd}^*(x)a^{*}\right)=  -x 
	 \star \widetilde{\alpha}\left(a^{*}\right)  , \quad \forall x \in A,  a^{*}  \in A^{*}.
	\end{align*}
		\end{enumerate}
\end{lem}

 \begin{proof}
 $\eqref{semi:1}   \Longrightarrow  \eqref{semi:2}$. For all $x \in A$, $a^{*}, b^{*} \in A^{*}$, we have
\begin{align*}
\left\langle(E(x) \alpha, a^{*} \otimes b^{*}\right\rangle & =-\left\langle\alpha,\left(R_{\star}^{*}(x)a^{*}\right)  \otimes  b^{*} + a^{*}  \otimes \left(L_{\lhd}^{*}(x)b^{*}\right)\right\rangle   
 = \left\langle\alpha,  b^{*}  \otimes  \left(R_{\star}^{*}(x)a^{*}\right) -a^{*}  \otimes \left(L_{\lhd}^{*}(x)b^{*}\right)\right\rangle \\
& =- \left\langle a^{*},  \widetilde{\alpha}\left(b^{*}\right) \star x \right\rangle+\left\langle x \lhd \widetilde{\alpha}\left(a^{*}\right) , b^{*}\right\rangle  
 =\left\langle L_{\star}^{*}\left(\widetilde{\alpha}\left(b^{*}\right)\right) a^{*}-R_{\lhd}^{*}\left(\widetilde{\alpha}\left(a^{*}\right)\right) b^{*}, x\right\rangle.\\
\left\langle(F(x) \alpha, a^{*} \otimes b^{*}\right\rangle & =-\left\langle\alpha,\left(R_{\rhd}^{*}(x)a^{*}\right)  \otimes  b^{*} - a^{*}  \otimes \left(L_{\star}^{*}(x)b^{*}\right)\right\rangle  
  = \left\langle\alpha, b^{*}  \otimes  \left(R_{\rhd}^{*}(x)a^{*}\right) +a^{*}  \otimes \left(L_{\star}^{*}(x)b^{*}\right)\right\rangle\\
& =-\left\langle a^{*}, \widetilde{\alpha}\left(b^{*}\right)\rhd x \right\rangle-\left\langle    x \star \widetilde{\alpha}\left(a^{*}\right), b^{*}\right\rangle  
 =\left\langle L_{\rhd}^{*}\left(\widetilde{\alpha}\left(b^{*}\right)\right) a^{*}+R_{\star}^{*}\left(\widetilde{\alpha}\left(a^{*}\right)\right) b^{*}, x\right\rangle.
\end{align*}
 $\eqref{semi:2}   \Longrightarrow  \eqref{semi:3} $. For all $x \in A$, $a^{*}, b^{*} \in A^{*}$, we have
\begin{align*}
\left\langle L_{\star}^{*}\left(\widetilde{\alpha}\left(b^{*}\right)\right) a^{*}-R_{\lhd}^{*}\left(\widetilde{\alpha}\left(a^{*}\right)\right) b^{*}, x\right\rangle &=- \left\langle a^{*},  \widetilde{\alpha}\left(b^{*}\right) \star x \right\rangle+\left\langle x \lhd \widetilde{\alpha}\left(a^{*}\right) , b^{*}\right\rangle \\
&=- \left\langle a^{*},  \widetilde{\alpha}\left(b^{*}\right) \star x \right\rangle-\left\langle   \widetilde{\alpha}\left(a^{*}\right) , L_{\lhd}^*(x) b^{*}\right\rangle \\
&=   \left\langle a^{*},  \widetilde{\alpha}\left( L_{\lhd}^*(x) b^{*}\right) -\widetilde{\alpha}\left(b^{*}\right) \star x \right\rangle, \\
\left\langle L_{\rhd}^{*}\left(\widetilde{\alpha}\left(b^{*}\right)\right) a^{*}+R_{\star}^{*}\left(\widetilde{\alpha}\left(a^{*}\right)\right) b^{*}, x\right\rangle  &=-\left\langle a^{*}, \widetilde{\alpha}\left(b^{*}\right)\rhd x \right\rangle-\left\langle    x \star \widetilde{\alpha}\left(a^{*}\right), b^{*}\right\rangle \\
&=-\left\langle a^{*}, \widetilde{\alpha}\left(b^{*}\right)\rhd x \right\rangle+\left\langle      \widetilde{\alpha}\left(a^{*}\right), L_{\star}^*(x)b^{*}\right\rangle \\
&= \left\langle a^{*}, -\widetilde{\alpha}\left(b^{*}\right)\rhd x - \widetilde{\alpha}\left(L_{\star}^*(x)b^{*}\right) \right\rangle.
\end{align*}
Similarly,   $\eqref{semi:3}   \Longrightarrow  \eqref{semi:4}$ and   $\eqref{semi:4}   \Longrightarrow  \eqref{semi:1}$. This completes the proof.
\end{proof}

Finally, we present the tensor forms of invariant bilinear forms on dialgebras.

\begin{pro}\label{tensorformIndias}
Let $(A,  \rhd,\lhd)$  be a dialgebra and $\mathfrak{B}$ be a nondegenerate bilinear form on $(A,  \rhd,\lhd)$.  Let $\varphi_{\mathfrak{B}}: A\rightarrow A^{*}$  be the induced linear isomorphism by $\mathfrak{B}$ and $r_{\mathfrak{B}} \in A \otimes A$ be the 2-tensor form of $\varphi_{\mathfrak{B}}^{-1}$ given by Eq. \eqref{inducedr}. Then    $(A,\rhd,\lhd,\mathfrak{B})$ is a   skew-symmetric  Frobenius dialgebra if and only if   $r_{\mathfrak{B}}$ is skew-symmetric and invariant.
\end{pro}
\begin{proof}
Suppose $(A,\rhd,\lhd,\mathfrak{B})$ is a   skew-symmetric  Frobenius dialgebra, then 
 \begin{align*}
\langle r_{\mathfrak{B}}, a^{*} \otimes b^{*} \rangle  = \langle \varphi_{\mathfrak{B}}^{-1}(a^*), b^*\rangle = -\langle \varphi_{\mathfrak{B}}^{-1}(b^*), a^*\rangle = -\langle r_{\mathfrak{B}}, b^{*} \otimes a^{*} \rangle  = -\langle \tau(r_{\mathfrak{B}}), a^{*} \otimes b^{*} \rangle.
\end{align*}
Thus $r_{\mathfrak{B}}$ is skew-symmetric. Moreover,  for all $a^*,b^*,c^* \in A^*$, there exist $x,y,z \in A$ such that $a^* = \varphi_{\mathfrak{B}}(x)$,  $b^* = \varphi_{\mathfrak{B}}(y)$, $c^* = \varphi_{\mathfrak{B}}(z)$. Thus we have
\begin{align*}
\mathfrak{B}(x \lhd y, z)- \mathfrak{B}(y, z \lhd x-z \rhd x) &= \langle\varphi_{\mathfrak{B}}(x \lhd y)  -  R_{\star}^*(x)\varphi_{\mathfrak{B}}(  y), z\rangle   = \langle \varphi_{\mathfrak{B}}(x \lhd y) -  R_{\star}^*(x)b^*, \varphi_{\mathfrak{B}}^{-1}(c^*)\rangle      \\
&= -\langle  c^*, \varphi_{\mathfrak{B}}^{-1}(\varphi_{\mathfrak{B}}(x \lhd y) \!-\!  R_{\star}^*(x)b^*)\rangle      
 =  \langle  c^*,   \varphi_{\mathfrak{B}}^{-1}(R_{\star}^*(x)b^*)\!-\! x \lhd \varphi_{\mathfrak{B}}^{-1}(b^*)    \rangle,      \\
&=0.\\
\mathfrak{B}(x \rhd y, z)- \mathfrak{B}(y,z \lhd x) &= \langle\varphi_{\mathfrak{B}}(x \rhd y)  \!+\!  R_{\lhd}^*(x)b^*, \varphi_{\mathfrak{B}}^{-1}(c^*)\rangle   
 = \!-\langle c^* ,  x \rhd \varphi_{\mathfrak{B}}^{-1}(b^*)   \!+ \! \varphi_{\mathfrak{B}}^{-1}(R_{\lhd}^*(x)b^*)\rangle\\
&=0.
\end{align*}
Then we obtain $\varphi_{\mathfrak{B}}^{-1}(R_{\star}^*(x)a^*)= x \lhd \varphi_{\mathfrak{B}}^{-1}(a^*)$ and $    \varphi_{\mathfrak{B}}^{-1}(R_{\lhd}^*(x)a^*) =- x \rhd \varphi_{\mathfrak{B}}^{-1}(a^*)  $ for all $a^* \in A^*$. Furthermore, 
\begin{align*}
 \varphi_{\mathfrak{B}}^{-1} \left(R_{\rhd}^*(x)a^{*}\right) &=     \varphi_{\mathfrak{B}}^{-1}\left( R_{\star}^*(x) a^{*}\right)  +  \varphi_{\mathfrak{B}}^{-1}(R_{\lhd}^*(x)a^*) = x \lhd \varphi_{\mathfrak{B}}^{-1}\left(a^{*}\right)  -x \rhd \varphi_{\mathfrak{B}}^{-1}(a^*)  
  =-x \star \varphi_{\mathfrak{B}}^{-1}(a^*).
\end{align*}
Then by Proposition \ref{semiinequ}, $r_{\mathfrak{B}}$  is   invariant. The converse statement can be proved by reversing the argument. 
\end{proof} 

 \section{Diassociative bialgebras and diassociative  Yang-Baxter equations}\label{Bialgebra}

In this section, we introduce the double construction of Frobenius dialgebras associated to an invariant bilinear form, as well as the notion of a diassociative bialgebra. The equivalence between them is characterized in terms of certain matched pairs of dialgebras. Furthermore, we prove that a diassociative bialgebra naturally gives rise to a Leibniz bialgebra, which is the bialgebra structure of Leibniz algebras through the sub-adjacent way.
The study of coboundary diassociative bialgebras leads to the introduction of the diassociative Yang–Baxter equation (DAYBE) in a dialgebra. Explicitly, a solution whose skew-symmetric part is invariant gives a special class of diassociative bialgebras, referred to as quasi-triangular diassociative bialgebras. Moreover, we introduce the notion of $\mathcal{O}$-operators on dialgebras with   weights,  which provides an operator form interpretation for solutions  of the DAYBE whose skew-symmetric part is invariant.

   \subsection{Double constructions of Frobenius dialgebras and bi-dialgebras }

  We first give the concept of a matched pair of dialgebras. Recall that  a {\bf matched pair} of associative algebras is a sextuple $(A,B,l_A,r_A,l_B,r_B)$ where $(B;l_A,r_A)$ is a bimodule of associative algebra  $(A,\cdot_A)$  and $(A;l_B,r_B)$  is a bimodule of associative algebra  $(B,\cdot_B)$ satisfying the following equations.
\begin{align}
l_{A}(x)(a \cdot_B b) & =l_{A}\left(r_{B}(a) x\right) b+\left(l_{A}(x) a\right) \cdot_B b,\label{mpofAss1}\\
r_{A}(x)(a \cdot_B b) & =r_{A}\left(l_{B}(b) x\right) a+a \cdot_B\left(r_{A}(x) b\right), \label{mpofAss2}\\
l_{B}(a)(x \cdot_A y) & =l_{B}\left(r_{A}(x) a\right) y+\left(l_{B}(a) x\right) \cdot_A y, \label{mpofAss3}\\
r_{B}(a)(x \cdot_A y) & =r_{B}\left(l_{A}(y) a\right) x+x \cdot_A\left(r_{B}(a) y\right), \label{mpofAss4}\\
l_{A}\left(l_{B}(a) x\right) b+\left(r_{A}(x) a\right) \cdot_B b & = r_{A}\left(r_{B}(b) x\right) a +a \cdot_B\left(l_{A}(x) b\right) , \label{mpofAss5}\\
l_{B}\left(l_{A}(x) a\right) y+\left(r_{B}(a) x\right) \cdot_A y & = r_{B}\left(r_{A}(y) a\right) x + x \cdot_A\left(l_{B}(a) y\right),\label{mpofAss6} \quad \forall x,y \in A, a,b \in B.
\end{align}

 \begin{defi}\label{mpofdiass}
 Let $(A,\rhd_A, \lhd_A)$ and $(B,\rhd_B, \lhd_B)$  be two dialgebras. Suppose that there are linear maps  $l_{\rhd_{A}}, r_{\rhd_{A}},l_{\lhd_{A}}, r_{\lhd_{A}}: A \rightarrow \operatorname{End}\left(B\right)$  and  $l_{\rhd_{B}}, r_{\rhd_{B}},l_{\lhd_{B}}, r_{\lhd_{B}}: B \rightarrow \operatorname{End}\left(A\right)$   such that  $(B;l_{\rhd_{A}}, r_{\rhd_{A}},l_{\lhd_{A}}, r_{\lhd_{A}})$  is a bimodule of  dialgebra $(A,\rhd_A, \lhd_A)$,  $(A;l_{\rhd_{B}}, r_{\rhd_{B}},l_{\lhd_{B}}, r_{\lhd_{B}})$  is a bimodule of  $(B,\rhd_B, \lhd_B)$, $(A,B,l_{\rhd_{A}}, r_{\rhd_{A}},$ $l_{\rhd_{B}}, r_{\rhd_{B}})$ and $(A,B,l_{\lhd_{A}}, r_{\lhd_{A}},l_{\lhd_{B}}, r_{\lhd_{B}})$ are matched pairs of associative algebras and they satisfy the following
conditions:
  \begin{align}
x \lhd_A \big(l_{\lhd_{B}}(a)(y)\big) + r_{\lhd_B}(r_{\lhd_A}(y)a)x &= x \lhd_A \big(l_{\rhd_{B}}(a)(y)\big) + r_{\lhd_B}(r_{\rhd_A}(y)a)x, \label{mpofDias1} \\
x \lhd_A \big(r_{\lhd_{B}}(a)(y)\big) + r_{\lhd_B}(l_{\lhd_A}(y)a)x &= x \lhd_A \big(r_{\rhd_{B}}(a)(y)\big) + r_{\lhd_B}(l_{\rhd_A}(y)a)x, \label{mpofDias2} \\
l_{\lhd_A}(x)(a \lhd_B b) &= l_{\lhd_A}(x)(a \rhd_B b),\label{mpofDias3} \\
a \lhd_B \big(l_{\lhd_{A}}(x)(b)\big) + r_{\lhd_A}(r_{\lhd_B}(b)x)a &= a \lhd_B \big(l_{\rhd_{A}}(x)(b)\big) + r_{\lhd_A}(r_{\rhd_B}(b)x)a, \label{mpofDias4} \\
a \lhd_B \big(r_{\lhd_{A}}(x)(b)\big) + r_{\lhd_A}(l_{\lhd_B}(b)x)a &= a \lhd_B \big(r_{\rhd_{A}}(x)(b)\big) + r_{\lhd_A}(l_{\rhd_B}(b)x)a, \label{mpofDias5} \\
l_{\lhd_B}(a)(x \lhd_A y) &= l_{\lhd_B}(a)(x \rhd_A y),\label{mpofDias6} \\
l_{\rhd_B}(a)(x \lhd_A y)&= \big(l_{\rhd_B}(a)x\big) \lhd_A y + l_{\lhd_B}(r_{\rhd_A}(x)a)y,\label{mpofDias7} \\
\big(r_{\rhd_B}(a)x\big) \lhd_A y+ l_{\lhd_{B}}(l_{\rhd_A}(x)a)y &= x \rhd_A \big(l_{\lhd_B}(a)y\big) +  r_{\rhd_{B}}(r_{\lhd_A}(y)a)x,\label{mpofDias8} \\ 
r_{\lhd_B}(a)(x \rhd_A y) &=  x \rhd_A \big(r_{\lhd_B}(a)y\big)  + r_{\rhd_{B}}(l_{\lhd_A}(y)a)x,\label{mpofDias9} \\
l_{\rhd_A}(x)(a \lhd_B b)&= \big(l_{\rhd_A}(x)a\big) \lhd_B b + l_{\lhd_A}(r_{\rhd_B}(a)x)b,\label{mpofDias10} \\
\big(r_{\rhd_A}(x)a\big) \lhd_B b+ l_{\lhd_{A}}(l_{\rhd_B}(a)x)b &= a \rhd_B \big(l_{\lhd_A}(x)b\big) +  r_{\rhd_{A}}(r_{\lhd_B}(b)x)a, \label{mpofDias11} \\ 
r_{\lhd_A}(x)(a \rhd_B b) &=  a \rhd_B \big(r_{\lhd_A}(x)b\big)  + r_{\rhd_{A}}(l_{\lhd_B}(b)x)a,\label{mpofDias12} \\
\big(l_{\lhd_B}(a)x\big) \rhd_A y+ l_{\rhd_{B}}(r_{\lhd_A}(x)a)y &=\big(l_{\rhd_B}(a)x\big) \rhd_A y+ l_{\rhd_{B}}(r_{\rhd_A}(x)a)y , \label{mpofDias13} \\ 
\big(r_{\lhd_B}(a)x\big) \rhd_A y+ l_{\rhd_{B}}(l_{\lhd_A}(x)a)y &=\big(r_{\rhd_B}(a)x\big) \rhd_A y+ l_{\rhd_{B}}(l_{\rhd_A}(x)a)y , \label{mpofDias14}  \\ 
r_{\rhd_B}(a)(x \lhd_A y) &= r_{\rhd_B}(a)(x \rhd_A y), \label{mpofDias15} \\
\big(l_{\lhd_A}(x)a\big) \rhd_B b+ l_{\rhd_{A}}(r_{\lhd_B}(a)x)b &=\big(l_{\rhd_A}(x)a\big) \rhd_B b+ l_{\rhd_{A}}(r_{\rhd_B}(a)x)b , \label{mpofDias16} \\ 
\big(r_{\lhd_A}(x)a\big) \rhd_B b+ l_{\rhd_{A}}(l_{\lhd_B}(a)x)b &=\big(r_{\rhd_A}(x)a\big) \rhd_B b+ l_{\rhd_{A}}(l_{\rhd_B}(a)x)b , \label{mpofDias17}  \\ 
r_{\rhd_A}(x)(a \lhd_B b) &= r_{\rhd_A}(x)(a \rhd_B b), \quad  \forall x,y \in A, a,b \in B. \label{mpofDias18} 
\end{align}
 Then  $(A,B,l_{\rhd_{A}}, r_{\rhd_{A}},l_{\lhd_{A}}, r_{\lhd_{A}},l_{\rhd_{B}}, r_{\rhd_{B}},l_{\lhd_{B}}, r_{\lhd_{B}})$  is called a   {\bf matched pair of dialgebras}.
 \end{defi}
 
 \begin{pro}\label{djdandm}
 Let $(A,\rhd_A, \lhd_A)$ and $(B,\rhd_B, \lhd_B)$  be two dialgebras. Suppose that there are linear maps  $l_{\rhd_{A}}, r_{\rhd_{A}},l_{\lhd_{A}}, r_{\lhd_{A}}: A \rightarrow \operatorname{End}\left(B\right)$  and  $l_{\rhd_{B}}, r_{\rhd_{B}},l_{\lhd_{B}}, r_{\lhd_{B}}: B \rightarrow \operatorname{End}\left(A\right)$. Define two binary operations  $\rhd_{A\oplus B},\lhd_{A\oplus B}:\left(A  \oplus B\right) \otimes\left(A  \oplus B\right) \rightarrow A \oplus B$  on  $A  \oplus B$  by
\begin{align}
(x+a) \rhd_{A\oplus B} (y+b) &=x \rhd_A y+l_{\rhd_B}(a) y+r_{\rhd_B}(b) x+l_{\rhd_A}(x) b+r_{\rhd_A}(y) a+a \rhd_B b,\\
(x+a) \lhd_{A\oplus B} (y+b) &=x \lhd_A y+l_{\lhd_B}(a) y+r_{\lhd_B}(b) x+l_{\lhd_A}(x) b+r_{\lhd_A}(y) a+a \lhd_B b, 
\end{align}
where  $x, y \in A, a, b \in B$. Then $(A  \oplus B , \rhd_{A\oplus B},\lhd_{A\oplus B})$ is a dialgebra  if and only if $(A,B,l_{\rhd_{A}}, r_{\rhd_{A}},l_{\lhd_{A}}, r_{\lhd_{A}},$ $l_{\rhd_{B}}, r_{\rhd_{B}},l_{\lhd_{B}}, r_{\lhd_{B}})$  is  a   matched pair of dialgebras. We denote this dialgebra by  $A  \bowtie_{l_{\rhd_{A}},  r_{\rhd_{A}},l_{\lhd_{A}}, r_{\lhd_{A}}}^{l_{\rhd_{B}}, r_{\rhd_{B}},l_{\lhd_{B}}, r_{\lhd_{B}}} B$  or simply  $A  \bowtie B$. On the other hand, every dialgebra with a decomposition into the direct sum of the underlying vector spaces of two diassociative subalgebras can be obtained in this way.
 \end{pro}
\begin{proof}
The proof is similar to the one of \cite[Theorem 2.1.4]{B2}.
\end{proof}

Next we introduce a notion of double construction of Frobenius dialgebras which is an analogue of the notion of double construction for Frobenius  algebras \cite{B2}.

\begin{defi} 
Let  $ (A, \rhd_A,\lhd_A)$ and $(A^{*},\rhd_{A^*},\lhd_{A^*})$ be two dialgebras. Suppose that there is  a dialgebra structure $\rhd_{A \oplus A^{*}}$, $\lhd_{A \oplus A^{*}}$ on the direct sum  $A \oplus A^{*}$  of the underlying vector spaces of  $A$  and  $A^{*}$  which contains both $ (A, \rhd_A,\lhd_A)$ and $(A^{*},\rhd_{A^*},\lhd_{A^*})$ as diassociative  subalgebras.  
If the bilinear form $\mathfrak{B}_{d}$ given by Eq.~\eqref{mtorbl}  is invariant  such that $(A \oplus A^{*}, \rhd_{A \oplus A^{*}},\lhd_{A \oplus A^{*}}, \mathfrak{B}_{d})$  is a skew-symmetric Frobenius dialgebra, then the triple $((A \oplus A^{*}, \rhd_{A \oplus A^{*}}, \lhd_{A \oplus A^{*}}, \mathfrak{B}_{d}), (A, \rhd_A,\lhd_A), (A^{*},\rhd_{A^*},\lhd_{A^*}))$ is called a {\bf double construction of Frobenius dialgebra} associated to $ (A, \rhd_A,\lhd_A)$ and $(A^{*},\rhd_{A^*},\lhd_{A^*})$. We denote it by  $(A \oplus A^{*},A,A^*, \mathfrak{B}_{d})$.
\end{defi}

\begin{ex} \label{doublefromdias}
In Proposition \ref{p:ppplgphpl},  $(A
\ltimes_{-R_{\lhd}^{*},  L_{\lhd}^{*}-L_{\rhd}^{*}, R_{\rhd}^{*}-R_{\lhd}^{*},-L_{\rhd}^{*}}
A^*,\mathfrak{B}_d)$ is a skew-symmetric Frobenius dialgebra  with  $\mathfrak{B}_d$  given by Eq.~\eqref{mtorbl}. Moreover, $(A
\ltimes_{-R_{\lhd}^{*},  L_{\lhd}^{*}-L_{\rhd}^{*}, R_{\rhd}^{*}-R_{\lhd}^{*},-L_{\rhd}^{*}}
A^*,A,A^*,\mathfrak{B}_d)$ is a double construction of Frobenius dialgebra associated to $ (A, \rhd_A,\lhd_A)$ and $(A^{*},\rhd_{A^*},\lhd_{A^*})$ where $(A^{*},\rhd_{A^*},\lhd_{A^*})$ is regarded as a  trivial dialgebra.
\end{ex}

\begin{pro}\label{mpandmtofdppa}
Let  $ (A, \rhd_A,\lhd_A)$ and $(A^{*},\rhd_{A^*},\lhd_{A^*})$ be two dialgebras.  Then   there is a double
construction of a Frobenius dialgebra associated to $(A, \rhd_A,\lhd_A)$ and $(A^{*},\rhd_{A^*},\lhd_{A^*})$ if and only if $(A,A^*,-R_{\lhd}^{*},  L_{\lhd}^{*}-L_{\rhd}^{*}, R_{\rhd}^{*}-R_{\lhd}^{*},-L_{\rhd}^{*},-\mathcal{R}_{\lhd}^{*},  \mathcal{L}_{\lhd}^{*}-\mathcal{L}_{\rhd}^{*}, \mathcal{R}_{\rhd}^{*}-\mathcal{R}_{\lhd}^{*},-\mathcal{L}_{\rhd}^{*})$  is a matched pair of dialgebras.
\end{pro}
\begin{proof}If  $(A,A^*,-R_{\lhd}^{*},  L_{\lhd}^{*}-L_{\rhd}^{*}, R_{\rhd}^{*}-R_{\lhd}^{*},-L_{\rhd}^{*},-\mathcal{R}_{\lhd}^{*},  \mathcal{L}_{\lhd}^{*}-\mathcal{L}_{\rhd}^{*}, \mathcal{R}_{\rhd}^{*}-\mathcal{R}_{\lhd}^{*},-\mathcal{L}_{\rhd}^{*})$  is a matched pair of dialgebras, then there is a dialgebra  $A \bowtie  A^{*}$  obtained from Proposition \ref{djdandm}, which includes  $ (A, \rhd_A,\lhd_A)$ and $(A^{*},\rhd_{A^*},\lhd_{A^*})$  as  diassociative subalgebras. Moreover, it
is straightforward to show that the skew-symmetric  bilinear form $\mathfrak{B}_d$ given by Eq.~\eqref{mtorbl} is invariant on $A \bowtie  A^{*}$, following a proof similar to that of Proposition \ref{p:ppplgphpl}.

Conversely, if $((A \oplus A^{*}, \rhd_{A \oplus A^{*}}, \lhd_{A \oplus A^{*}}, \mathfrak{B}_{d}), (A, \rhd_A,\lhd_A), (A^{*},\rhd_{A^*},\lhd_{A^*}))$  is a double construction of Frobenius dialgebra associated to $ (A, \rhd_A,\lhd_A)$ and $(A^{*},\rhd_{A^*},\lhd_{A^*})$, then we set 
\begin{align*}
x \rhd a^{*}&=l_{\rhd_{A}}(x) a^{*}+r_{\rhd_{A^{*}}}(a^{*}) x, && a^{*}\rhd x=l_{\rhd_{A^{*}}}(a^{*}) x+r_{\rhd_A}(x) a^{*}, \\
x \lhd a^{*}&=l_{\lhd_{A}}(x) a^{*}+r_{\lhd_{A^{*}}}(a^{*}) x, &&  a^{*} \lhd x =l_{\lhd_{A^{*}}}(a^{*}) x+r_{\lhd_A}(x) a^{*},\; \forall x \in A, a^{*} \in A^{*}.
\end{align*}
By Proposition \ref{djdandm},  $(A, A^{*}, l_{\rhd_{A}}, r_{\rhd_{A}},l_{\lhd_{A}}, r_{\lhd_{A}}, l_{\rhd_{A^*}}, r_{\rhd_{A^*}},l_{\lhd_{A^*}}, r_{\lhd_{A^*}})$  is a matched pair of dialgebras. Note that
\begin{align*}
\left\langle  y,l_{\rhd_{A}}(x) a^{*}\right\rangle &= \mathfrak{B}_d( l_{\rhd_{A}}(x) a^{*}, y) =\mathfrak{B}_d( x \rhd a^{*} - r_{\rhd_{A^{*}}}(a^{*}) x, y)=\mathfrak{B}_d( x \rhd a^{*}, y) \\
&= \mathfrak{B}_d(  a^{*},   y  \lhd_{A} x) =\left\langle    y,  -R_{\lhd}^{*}(x)(a^{*}) \right\rangle.\\
\left\langle  y,r_{\rhd_{A}}(x) a^{*}\right\rangle &= \mathfrak{B}_d( r_{\rhd_{A}}(x) a^{*}, y) =\mathfrak{B}_d(a^{*} \rhd x - l_{\rhd_{A^{*}}}(a^{*}) x, y)=\mathfrak{B}_d( a^{*} \rhd x , y) \\
&= \mathfrak{B}_d( a^{*}  , x \rhd_{A}  y - x \lhd_{A} y)   =  \left\langle  y, ( L_{\lhd}^*-L_{\rhd}^*)(x)a^{*}\right\rangle.\\
\left\langle  y,l_{\lhd_{A}}(x) a^{*}\right\rangle &= \mathfrak{B}_d( l_{\lhd_{A}}(x) a^{*}, y) =\mathfrak{B}_d( x \lhd a^{*} - r_{\lhd_{A^{*}}}(a^{*}) x, y)=\mathfrak{B}_d( x \lhd a^{*}, y) \\
&= \mathfrak{B}_d(   a^{*}, y\lhd x - y \rhd x) =  \left\langle  y, ( R_{\rhd}^*-R_{\lhd}^*)(x)a^{*}\right\rangle.\\
\left\langle  y,r_{\lhd_{A}}(x) a^{*}\right\rangle &= \mathfrak{B}_d( r_{\lhd_{A}}(x) a^{*}, y) =\mathfrak{B}_d(a^{*} \lhd x- l_{\lhd_{A^{*}}}(a^{*}) x, y)=\mathfrak{B}_d( a^{*} \lhd x , y) \\
&=\mathfrak{B}_d( a^{*} ,x\rhd   y) =  \left\langle  y,-  L_{\rhd}^* (x)a^{*}\right\rangle,
\end{align*}
where $x, y \in A$, $a^{*}  \in A^{*}$. Hence  $l_{\rhd_{A}}=-R_{\lhd}^{*},\; r_{\rhd_{A}}=L_{\lhd}^{*}-L_{\rhd}^{*},\; l_{\lhd_{A}}=R_{\rhd}^{*}-R_{\lhd}^{*},\; r_{\lhd_{A}}=-L_{\rhd}^{*}$. Similarly, we have
 $l_{\rhd_{A^{*}}}=-\mathfrak{R}_{\lhd}^{*},\; r_{\rhd_{A^{*}}}=\mathfrak{L}_{\lhd}^{*}-\mathfrak{L}_{\rhd}^{*},\; l_{\lhd_{A^{*}}}=\mathfrak{R}_{\rhd}^{*}-\mathfrak{R}_{\lhd}^{*},\; r_{\lhd_{A^{*}}}=-\mathfrak{L}_{\rhd}^{*}$. Thus $(A,A^*,-R_{\lhd}^{*},  L_{\lhd}^{*}-L_{\rhd}^{*}, R_{\rhd}^{*}-R_{\lhd}^{*},-L_{\rhd}^{*},-\mathfrak{R}_{\lhd}^{*},  \mathfrak{L}_{\lhd}^{*}-\mathfrak{L}_{\rhd}^{*}, \mathfrak{R}_{\rhd}^{*}-\mathfrak{R}_{\lhd}^{*},-\mathfrak{L}_{\rhd}^{*})$  is a matched pair of dialgebras.
\end{proof}

Now we give the definitions of a diassociative coalgebra and a diassociative bialgebra.
   

 Recall that an {\bf associative coalgebra} is pair $(A,\delta)$ that  $A$ is a vector space  and $\delta: A \to A \otimes A$ satisfying the coassociativity condition.
\begin{align}
(\id \otimes \delta)\delta = (\delta  \otimes \id )\delta.
\end{align}
 
\begin{defi} \label{defco1}
A   \textbf{diassociative coalgebra} (or {\bf co-dialgebra}) is a triple $(A, \delta_{\rhd},\delta_{\lhd})$ where  $(A, \delta_{\rhd})$ and $(A, \delta_{\lhd})$ are associative coalgebras satisfying  the following conditions.
\begin{align}
(\id \otimes \delta_{\lhd})\delta_{\lhd}  &= (\id \otimes \delta_{\rhd})\delta_{\lhd} ,\label{codiass1} \\
( \delta_{\rhd} \otimes \id)\delta_{\lhd}  &= (\id \otimes \delta_{\lhd})\delta_{\rhd}, \label{codiass2}\\
( \delta_{\lhd} \otimes \id)\delta_{\rhd}   &= ( \delta_{\rhd} \otimes \id)\delta_{\rhd}.\label{codiass3}
\end{align}
\end{defi}

\begin{pro}
Let $A$ be a finite-dimensional vector space and $\delta_{\rhd},\delta_{\lhd}: A \to A \otimes A$ be linear maps. If  the binary operations $\rhd_{A^*}, \lhd_{A^*}: A^* \otimes A^* \to A^*$ are the linear duals of $\delta_{\rhd},\delta_{\lhd}$ respectively, that is, $\rhd_{A^*}$ and $\lhd_{A^*}$ are respectively defined by 
\begin{align*}
\langle a^* \rhd_{A^{*}} b^*, x \rangle := \langle a^* \otimes b^*, \delta_{\rhd}(x) \rangle, \; \langle a^* \lhd_{A^{*}} b^*, x \rangle := \langle a^* \otimes b^*, \delta_{\lhd}(x) \rangle,\;\forall x \in A, a^*,b^* \in A^{*}.
\end{align*}
Then $(A,\delta_{\rhd},\delta_{\lhd})$ is a  co-dialgebra if and only if $(A^*,\rhd_{A^*},\lhd_{A^*})$ is a dialgebra.
\end{pro}

 For a vector space $A$, let  $\tau: A \otimes A \rightarrow A \otimes A$  be the exchange operator defined as  
$$\tau(x \otimes y)=y \otimes x,\quad \forall x, y \in A.$$ 

Now we give the definition of a diassociative bialgebra as an analogue of a Lie bialgebra of Drinfeld.

   \begin{defi} \label{defbi1}
A   \textbf{diassociative bialgebra } (or {\bf bi-dialgebra})  is a quintuple $(A,\rhd,\lhd, \delta_{\rhd},$ $\delta_{\lhd})$ where $(A,\rhd,\lhd)$ is a dialgebra and $(A,\delta_{\rhd},\delta_{\lhd})$ is a co-dialgebra satisfying the following compatibility conditions.
\begin{align}
 &  (\operatorname{id} \otimes L_{\lhd}(x)  ) \delta_{\lhd}(y) = ( R_{\rhd}(y) \otimes \operatorname{id}) \delta_{\rhd}(x),\label{bidicd1}\\
 &( L_{\lhd}(x) \otimes \operatorname{id}  ) \delta_{\rhd}(y)=
 (\operatorname{id} \otimes L_{\lhd}(y)  ) \tau \delta_{\rhd}(x),\label{bidicd2}\\ 
 & ( R_{\rhd}(x) \otimes \operatorname{id}  ) \tau\delta_{\lhd}(y)= (\operatorname{id} \otimes R_{\rhd}(y)  )  \delta_{\lhd}(x),\label{bidicd3}\\
&\delta_{\rhd}(x \rhd y)=  (\operatorname{id}  \otimes  L_{\rhd}(x)  ) \delta_{\rhd}(y) -( R_{\star}(y)  \otimes    \operatorname{id}) \delta_{\star}(x),\label{bidicd4}\\
&\delta_{\rhd}(x \lhd y)=  (\operatorname{id} \otimes L_{\lhd}(x)  ) \delta_{\star}(y)+( R_{\lhd}(y) \otimes \operatorname{id}) \delta_{\rhd}(x),\label{bidicd5}\\
&\delta_{\lhd}(x \rhd y)=  (\operatorname{id} \otimes L_{\star}(x)  ) \delta_{\lhd}(y)+( R_{\rhd}(y) \otimes \operatorname{id}) \delta_{\lhd}(x),\label{bidicd6}\\
 &\delta_{\lhd}(x \lhd y)=   ( R_{\lhd}(y)  \otimes    \operatorname{id}) \delta_{\lhd}(x)-(\operatorname{id}  \otimes  L_{\star}(x)  ) \delta_{\star}(y),\label{bidicd7}\\
&( L_{\star}(x) \otimes \operatorname{id}  ) \delta_{\rhd}(y)+(\operatorname{id} \otimes L_{\lhd}(y)  ) \tau\delta_{\lhd}(x) = 
( R_{\star}(y) \otimes \operatorname{id}  ) \tau\delta_{\star}(x)+ (\operatorname{id} \otimes R_{\lhd}(x)  ) \delta_{\rhd}(y),\label{bidicd8}\\
 &( R_{\lhd}(x) \otimes \operatorname{id}  ) \tau\delta_{\lhd}(y)+(\operatorname{id} \otimes R_{\rhd}(y)  ) \delta_{\star}(x) = 
 ( L_{\star}(y) \otimes \operatorname{id}  )  \delta_{\star}(x)+ (\operatorname{id} \otimes L_{\rhd}(x)  ) \tau\delta_{\lhd}(y),\label{bidicd9}
\end{align}
where $x \star y = x \rhd y -  x \lhd y$,  $\delta_{\star} = \delta_{\rhd}- \delta_{\lhd}$  for all  $x,y \in A$.

\end{defi}

\begin{pro}\label{mpdppba}
Let  $ (A, \rhd_A,\lhd_A)$ and $(A^{*},\rhd_{A^*},\lhd_{A^*})$ be two dialgebras. Let  linear maps  $\delta_{\rhd},\delta_{\lhd}: A \rightarrow A \otimes A $ be the linear duals of  $\rhd_{A^{*}}$ and  $\lhd_{A^{*}}$ respectively. Then $(A,A^*,-R_{\lhd}^{*},  L_{\lhd}^{*}-L_{\rhd}^{*}, R_{\rhd}^{*}-R_{\lhd}^{*},-L_{\rhd}^{*},-\mathcal{R}_{\lhd}^{*},  \mathcal{L}_{\lhd}^{*}-\mathcal{L}_{\rhd}^{*}, \mathcal{R}_{\rhd}^{*}-\mathcal{R}_{\lhd}^{*},-\mathcal{L}_{\rhd}^{*})$  is a matched pair of dialgebras if and only if  $(A,\rhd_{A},\lhd_{A}, \delta_{\rhd},\delta_{\lhd})$  is a bi-dialgebra. 
\end{pro}
\begin{proof}Define  binary operation $\star:  A \otimes A \rightarrow A$ and linear map $ \delta_{\star}: A \rightarrow A \otimes A $ by Eq.~\eqref{jianfayuns}. It is straightforward that  $(A^{*};-R_{\lhd}^{*},  -L_{\star}^{*}, R_{\star}^{*},-L_{\rhd}^{*})$  and $(A;-\mathcal{R}_{\lhd}^{*}, - \mathcal{L}_{\star}^{*}, \mathcal{R}_{\star}^{*},-\mathcal{L}_{\rhd}^{*})$ are bimodules of dialgebras $(A,\rhd_A, \lhd_A)$ and $(A^{*},\rhd_{A^{*}}, \lhd_{A^{*}})$ respectively.

We rewrite Eqs.~(\ref{mpofAss1})--(\ref{mpofAss6}) in the cases that
$l_{A}=-R_{\triangleleft}^{*}$, $r_{A}=-L_{\star}^{*}$, $l_{B}=-\mathcal{R}_{\triangleleft}^{*}$, $r_{B}=-\mathcal{L}_{\star}^{*}
$
and
$
l_{A}=R_{\star}^{*}$, $r_{A}=-L_{\triangleright}^{*}$, $l_{B}=\mathcal{R}_{\star}^{*}$, $r_{B}=-\mathcal{L}_{\triangleright}^{*}
$ and label the corresponding equations as Eqs.~(\ref{mpofAss1}$^\prime$)--(\ref{mpofAss6}$^\prime$) and Eqs.~(\ref{mpofAss1}$^{\prime\prime}$)--(\ref{mpofAss6}$^{\prime\prime}$) respectively.  Let $x,y \in A$, $a^*,b^* \in A^*$, we have
\begin{align*}
&\left \langle -R_{\lhd}^*(x)(a^{*} \rhd_{A^{*}} b^{*}) - R_{\lhd}^*\left( \mathcal{L}_{\star}^*(a^{*}) x\right) b^{*}+\left( R_{\lhd}^*(x) a^{*}\right) \rhd_{A^{*}} b^{*} , y \right \rangle\\
&=\left \langle   a^{*} \rhd_{A^{*}} b^{*}, y \lhd_A x \right \rangle + \left \langle a^{*}  \star_{A^*} L_{\lhd}^*(y)b^{*},    x  \right \rangle  +\left \langle\left( R_{\lhd}^*(x) a^{*}\right) \otimes b^{*} , \delta_{\rhd}(y) \right \rangle\\
&=\left \langle   a^{*} \otimes b^{*}, \delta_{\rhd}(y \lhd_A x)-(\id \otimes L_{\lhd} (y))\delta_{\star}(x)-(R_{\lhd}(x) \otimes \id) \delta_{\rhd}(y)\right \rangle,\\
&\left \langle -L_{\star}^*(x)(a^{*} \rhd_{A^{*}} b^{*}) - L_{\star}^*\left( \mathcal{R}_{\lhd}^*(b^{*}) x\right) a^{*}+a^{*} \rhd_{A^{*}}\left( L_{\star}^*(x) b^{*}\right), y \right \rangle\\
&=\left \langle   a^{*} \otimes b^{*}, \delta_{\rhd}(x \star_A y)-(R_{\star}(x) \otimes \id) \delta_{\lhd}(y)-(\id \otimes L_{\star} (x))\delta_{\rhd}(y)\right \rangle,\\
&\left \langle -\mathcal{R}_{\lhd}^*(a^{*})(x \rhd_A y) - \mathcal{R}_{\lhd}^*\left( L_{\star}^*(x) a^{*}\right) y+\left( \mathcal{R}_{\lhd}^*(a^{*}) x\right) \rhd_A y, b\right \rangle\\
&=\left \langle   a^{*} \otimes b^{*}, \delta_{\lhd}(x \rhd_A y)-(\id \otimes L_{\star} (x))\delta_{\lhd}(y)-(R_{\rhd}(y) \otimes \id) \delta_{\lhd}(x)\right \rangle,\\
&\left \langle -\mathcal{L}_{\star}^*(a^{*})(x \rhd_A y) - \mathcal{L}_{\star}^*\left( R_{\lhd}^*(y) a^{*}\right) x +x \rhd_A\left( \mathcal{L}_{\star}^*(a^{*}) y\right), b\right \rangle\\
&=\left \langle   a^{*} \otimes b^{*}, \delta_{\star}(x \rhd_A y)-(R_{\lhd}(y) \otimes \id) \delta_{\star}(x)-(\id \otimes L_{\rhd} (x))\delta_{\star}(y)\right \rangle,\\
&\left \langle  R_{\lhd}^*\left( \mathcal{R}_{\lhd}^*(a^{*}) x\right) b^{*}-\left( L_{\star}^*(x) a^{*}\right) \rhd_{A^{*}} b^{*} -  L_{\star}^*\left( \mathcal{L}_{\star}^*(b^{*}) x\right) a^{*} +a^{*} \rhd_{A^{*}}\left( R_{\lhd}^*(x) b^{*}\right) , y\right \rangle\\
&=\left \langle   a^{*} \!\otimes\! b^{*}, (L_{\star}(x) \!\otimes\! \id) \delta_{\rhd}(y)+( \id \!\otimes\! L_{\lhd}(y) ) \tau\delta_{\lhd}(x)-(R_{\star}(y) \!\otimes\! \id) \tau\delta_{\star}(x)-( \id \!\otimes\! R_{\lhd}(x) )  \delta_{\rhd}(y)\right \rangle,\\
&\left \langle   \mathcal{R}_{\lhd}^*\left( R_{\lhd}^*(x) a^{*}\right) y-\left( \mathcal{L}_{\star}^*(a^{*}) x\right) \rhd_A y -  \mathcal{L}_{\star}^*\left( L_{\star}^*(y) a^{*}\right) x + x \rhd_A\left( \mathcal{R}_{\lhd}^*(a^{*}) y\right) , b\right \rangle\\
&=\left \langle   a^{*} \!\otimes\! b^{*}, (R_{\lhd}(x) \!\otimes\! \id) \tau\delta_{\lhd}(y)+( \id \!\otimes\! R_{\rhd}(y) )  \delta_{\star}(x)-(L_{\star}(y) \!\otimes\! \id)  \delta_{\star}(x)-( \id \!\otimes\! L_{\rhd}(x) )  \tau\delta_{\lhd}(y)\right \rangle.
\end{align*}
Futhermore, we also have
\begin{align*}
&\left \langle R_{\star}^*(x)(a^{*} \lhd_{A^{*}} b^{*})  +R_{\star}^*\left( \mathcal{L}_{\rhd}^*(a^{*}) x\right) b^{*}-\left(R_{\star}^*(x) a^{*}\right) \lhd_{A^{*}} b^{*}, y \right \rangle\\
&=-\left \langle  a^{*} \otimes b^{*}  ,\delta_{\lhd}(y \star x) \right \rangle  -\left \langle     a^{*}  \rhd_{A^{*}} L_{\star}^*(y) b^{*}, x \right \rangle +\left \langle  a^{*}  \otimes b^{*}, (R_{\star} (x)\otimes  \id) \delta_{\lhd}(y)\right \rangle\\
&= \left \langle  a^{*} \otimes b^{*}  ,-\delta_{\lhd}(y \star x)+(\id \otimes L_{\star} (y))\delta_{\rhd}(x)  +(R_{\star} (x)\otimes  \id) \delta_{\lhd}(y)\right \rangle,\\
&\left \langle -L_{\rhd}^*(x)(a^{*} \lhd_{A^{*}} b^{*}) +L_{\rhd}^*\left(\mathcal{R}_{\star}^*(b^{*}) x\right) a^{*}+a^{*} \lhd_{A^{*}}\left( L_{\rhd}^*(x) b^{*}\right), y \right \rangle\\
&= \left \langle  a^{*} \otimes b^{*}  , \delta_{\lhd}(x \rhd y) +(R_{\rhd} (y)\otimes  \id) \delta_{\star}(x)-(\id \otimes L_{\rhd} (x))\delta_{\lhd}(y) \right \rangle,\\
&\left \langle \mathcal{R}_{\star}^*(a^{*})(x \lhd_A y) +\mathcal{R}_{\star}^*\left(L_{\rhd}^*(x) a^{*}\right) y-\left(\mathcal{R}_{\star}^*(a^{*}) x\right) \lhd_A y, b \right \rangle\\
&= \left \langle  a^{*} \otimes b^{*}  , -\tau\delta_{\star}(x \lhd y) +(L_{\rhd} (x) \otimes \id )\tau\delta_{\star}(y) +( \id \otimes  R_{\lhd} (y)) \tau\delta_{\star}(x)\right \rangle,\\
&\left \langle -\mathcal{L}_{\rhd}^*(a^{*})(x \lhd_A y) +\mathcal{L}_{\rhd}^*\left(R_{\star}^*(y) a^{*}\right) x+x \lhd_A\left( \mathcal{L}_{\rhd}^*(a^{*}) y\right), b \right \rangle\\
&= \left \langle  a^{*} \otimes b^{*}  ,\delta_{\rhd}(x \lhd y) +(R_{\star} (y)\otimes  \id) \delta_{\rhd}(x)-(\id \otimes L_{\lhd} (x))\delta_{\rhd}(y) \right \rangle,\\
&\left \langle R_{\star}^*\left(\mathcal{R}_{\star}^*(a^{*}) x\right) b^{*}-\left( L_{\rhd}^*(x) a^{*}\right) \lhd_{A^{*}} b^{*} - L_{\rhd}^*\left(\mathcal{L}_{\rhd}^*(b^{*}) x\right) a^{*} -a^{*} \lhd_{A^{*}}\left(R_{\star}^*(x) b^{*}\right), y \right \rangle\\
&=\left \langle   a^{*} \!\otimes\! b^{*}, ( \id \!\otimes\! L_{\star}(y) ) \tau \delta_{\star}(x)+(L_{\rhd}(x) \!\otimes\! \id) \delta_{\lhd}(y)-(R_{\rhd}(y) \!\otimes\! \id) \tau\delta_{\rhd}(x) + ( \id \!\otimes\! R_{\star}(x) )  \delta_{\lhd}(y)\right \rangle,\\
&\left \langle \mathcal{R}_{\star}^*\left(R_{\star}^*(x) a^{*}\right) y-\left( \mathcal{L}_{\rhd}^*(a^{*}) x\right) \lhd_A y -  \mathcal{L}_{\rhd}^*\left( L_{\rhd}^*(y) a^{*}\right) x - x \lhd_A\left(\mathcal{R}_{\star}^*(a^{*}) y\right),   b \right \rangle\\
&=\left \langle   a^{*} \!\otimes\! b^{*}, (R_{\star}(x) \!\otimes\! \id) \tau\delta_{\star}(y)+( \id \!\otimes\! R_{\lhd}(y) ) \delta_{\rhd}(x) -(L_{\rhd}(y) \!\otimes\! \id)  \delta_{\rhd}(x) + ( \id \!\otimes\!L_{\lhd}(x) )  \tau\delta_{\star}(y)\right \rangle.
\end{align*}
Note that
 \begin{align*}
&\text { Eq. }(\ref{mpofAss1}^\prime)   \Longleftrightarrow  \text { Eq. }(\ref{bidicd5}), &&\text { Eq. }(\ref{mpofAss3}^\prime)   \Longleftrightarrow  \text { Eq. }(\ref{bidicd6}), \\
& \text { Eq. }(\ref{mpofAss5}^\prime)   \Longleftrightarrow \text { Eq. }(\ref{bidicd8}), &&\text { Eq. }(\ref{mpofAss6}^\prime)  \Longleftrightarrow \text { Eq. }(\ref{bidicd9}).
 \end{align*}
 On one hand,
 \begin{align*}
 &\text { Eq. }(\ref{mpofAss1}^\prime) \text { and Eq. }(\ref{mpofAss4}^{\prime\prime})  \Longrightarrow \text { Eq. }(\ref{bidicd1}) \Longleftarrow  \text { Eq. }(\ref{mpofAss3}^\prime) \text { and Eq. }(\ref{mpofAss2}^{\prime\prime}),\\
 &\text { Eq. }(\ref{mpofAss5}^\prime) \text { and Eq. }(\ref{mpofAss6}^{\prime\prime})  \Longrightarrow \text { Eq. }(\ref{bidicd2}), \\
 & \text { Eq. }(\ref{mpofAss6}^\prime) \text { and Eq. }(\ref{mpofAss5}^{\prime\prime})  \Longrightarrow \text { Eq. }(\ref{bidicd3}), \\
 &\text { Eq. }(\ref{mpofAss1}^\prime) \text {, Eq. }(\ref{mpofAss2}^{\prime})   \text { and Eq. }(\ref{mpofAss4}^{\prime\prime})   \Longrightarrow \text { Eq. }(\ref{bidicd4})  \Longleftarrow \text { Eq. }(\ref{mpofAss3}^\prime) \text {, Eq. }(\ref{mpofAss4}^{\prime})   \text { and Eq. }(\ref{mpofAss2}^{\prime\prime}), \\
 &\text { Eq. }(\ref{mpofAss1}^{\prime\prime}) \text {, Eq. }(\ref{mpofAss2}^{\prime\prime})   \text { and Eq. }(\ref{mpofAss3}^\prime)   \Longrightarrow \text { Eq. }(\ref{bidicd7})\Longleftarrow \text { Eq. }(\ref{mpofAss3}^{\prime\prime}) \text {, Eq. }(\ref{mpofAss4}^{\prime\prime})   \text { and Eq. }(\ref{mpofAss1}^\prime).
\end{align*}
 On the other hand,
 \begin{align*}
 \text { Eq. }(\ref{bidicd1}) \text {, Eq. }(\ref{bidicd4}) \text { and Eq. }(\ref{bidicd5})  &\Longrightarrow \text { Eq. }(\ref{mpofAss2}^{\prime}), \\
\text { Eq. }(\ref{bidicd1}) \text {, Eq. }(\ref{bidicd4}) \text { and Eq. }(\ref{bidicd6})  &\Longrightarrow \text { Eq. }(\ref{mpofAss4}^{\prime}) ,\\
\text { Eq. }(\ref{bidicd1}) \text {, Eq. }(\ref{bidicd6})   \text { and Eq. }(\ref{bidicd7})  &\Longrightarrow \text { Eq. }(\ref{mpofAss1}^{\prime\prime}),\\
\text { Eq. }(\ref{bidicd1}) \text { and Eq. }(\ref{bidicd6}) &\Longrightarrow \text { Eq. }(\ref{mpofAss2}^{\prime\prime}), \\
  \text { Eq. }(\ref{bidicd1}) \text {, Eq. }(\ref{bidicd5})   \text { and Eq. }(\ref{bidicd7}) &\Longrightarrow \text { Eq. }(\ref{mpofAss3}^{\prime\prime}),\\
\text { Eq. }(\ref{bidicd1}) \text { and Eq. }(\ref{bidicd5}) &\Longrightarrow \text { Eq. }(\ref{mpofAss4}^{\prime\prime}), \\
\text { Eq. }(\ref{bidicd3}) \text { and Eq. }(\ref{bidicd9}) &\Longrightarrow \text { Eq. }(\ref{mpofAss5}^{\prime\prime}), \\
\text { Eq. }(\ref{bidicd2}) \text { and Eq. }(\ref{bidicd8})  &\Longrightarrow \text { Eq. }(\ref{mpofAss6}^{\prime\prime}).
\end{align*}
Thus $(A,A^*,-R_{\lhd}^{*},  -L_{\star}^{*},  -\mathcal{R}_{\lhd}^{*}, - \mathcal{L}_{\star}^{*})$  and $(A,A^*, R_{\star}^{*},-L_{\rhd}^{*}, \mathcal{R}_{\star}^{*},-\mathcal{L}_{\rhd}^{*})$  are matched pairs of  associative algebras if and only if Eqs.~\eqref{bidicd1}-\eqref{bidicd9} holds.
 
Moreover, in Eqs.~\eqref{mpofDias1}-\eqref{mpofDias18},  we take 
 \begin{align*}
 &l_{\rhd_{A}} =  -R_{\lhd}^{*}, &&r_{\rhd_{A}} =  -L_{\star}^{*}, && l_{\lhd_{A}}   = R_{\star}^{*},  && r_{\lhd_{A}} = -L_{\rhd}^{*},  \\
 &l_{\rhd_{B}} =  -\mathcal{R}_{\lhd}^{*}, && r_{\rhd_{B}} = - \mathcal{L}_{\star}^{*}, && l_{\lhd_{B}}   = \mathcal{R}_{\star}^{*},&&  r_{\lhd_{B}} =-\mathcal{L}_{\rhd}^{*},
 \end{align*}
which will be also  labeled by Eqs.~\eqref{mpofDias1}-\eqref{mpofDias18} respectively.  By a direct calculation, we have
  \begin{align*}
 \text { Eq. }(\ref{bidicd2}) \Longleftrightarrow \text { Eq. }(\ref{mpofDias1})&\Longleftrightarrow \text { Eq. }(\ref{mpofDias17}), \\
  \text { Eq. }(\ref{bidicd1}) \Longleftrightarrow \text { Eq. }(\ref{mpofDias2}) &\Longleftrightarrow \text { Eq. }(\ref{mpofDias5}) \Longleftrightarrow \text { Eq. }(\ref{mpofDias13}) \Longleftrightarrow \text { Eq. }(\ref{mpofDias16}), \\
   \text { Eqs. }(\ref{bidicd4}) \text {-}(\ref{bidicd7}) \Longrightarrow \text { Eq. }(\ref{mpofDias3}) &\Longleftrightarrow  \text { Eq. }(\ref{mpofDias18}), \\
    \text { Eq. }(\ref{bidicd3}) \Longleftrightarrow \text { Eq. }(\ref{mpofDias4}) &\Longleftrightarrow \text { Eq. }(\ref{mpofDias14}), \\ 
       \text { Eqs. }(\ref{bidicd4}) \text {-}(\ref{bidicd7}) \Longrightarrow \text { Eq. }(\ref{mpofDias6}) &\Longleftrightarrow \text { Eq. }(\ref{mpofDias15}), \\
         \text { Eq. }(\ref{bidicd7}) \Longleftrightarrow \text { Eq. }(\ref{mpofDias7}) &\Longleftrightarrow \text { Eq. }(\ref{mpofDias10}), \\   
   \text { Eqs. }(\ref{bidicd2}),\; (\ref{bidicd8}) \text {-}(\ref{bidicd9}) \Longrightarrow \text { Eq. }(\ref{mpofDias8})&, \\
       \text { Eq. }(\ref{bidicd4}) \Longleftrightarrow \text { Eq. }(\ref{mpofDias9})&\Longleftrightarrow \text { Eq. }(\ref{mpofDias12}), \\   
         \text { Eqs. }(\ref{bidicd3}),\; (\ref{bidicd8}) \text {-}(\ref{bidicd9}) \Longrightarrow \text { Eq. }(\ref{mpofDias11})&.     
\end{align*}
Thus $(A,A^*,-R_{\lhd}^{*},  -L_{\star}^{*}, R_{\star}^{*},-L_{\rhd}^{*},-\mathcal{R}_{\lhd}^{*}, - \mathcal{L}_{\star}^{*}, \mathcal{R}_{\star}^{*},-\mathcal{L}_{\rhd}^{*})$  is a matched pair of dialgebras if and only if  $(A,\rhd_{A},$ $\lhd_{A}, \delta_{\rhd},\delta_{\lhd})$  is a bi-dialgebra.  
\end{proof}

 \begin{thm}\label{sandengjia}
Let  $(A,\rhd_A,\lhd_A)$ and $(A^*,\rhd_{A^*},\lhd_{A^*})$  be   dialgebras. Let  linear maps  $\delta_{\rhd},\delta_{\lhd}: A \rightarrow A \otimes A $ be the linear duals of  $\rhd_{A^{*}}$ and  $\lhd_{A^{*}}$ respectively.  Then the following conditions are equivalent.
\begin{enumerate}
\item There is a double
construction of  Frobenius dialgebra associated to $(A, \rhd_A,\lhd_A)$ and $(A^{*},\rhd_{A^*},\lhd_{A^*})$.
\item $(A,A^*,-R_{\lhd}^{*},  L_{\lhd}^{*}-L_{\rhd}^{*}, R_{\rhd}^{*}-R_{\lhd}^{*},-L_{\rhd}^{*},-\mathcal{R}_{\lhd}^{*},  \mathcal{L}_{\lhd}^{*}-\mathcal{L}_{\rhd}^{*}, \mathcal{R}_{\rhd}^{*}-\mathcal{R}_{\lhd}^{*},-\mathcal{L}_{\rhd}^{*})$  is a matched pair of dialgebras.
    \item $(A,\rhd_{A},\lhd_{A}, \delta_{\rhd},\delta_{\lhd})$  is a bi-dialgebra.
\end{enumerate}
 \end{thm}
 \begin{proof}
 It follows from Propositions \ref{mpandmtofdppa} and \ref{mpdppba}.
  \end{proof}
  
To conclude this subsection, we establish the connection between Leibniz bialgebras and bi-dialgebras, lifting the well-known relation that a dialgebra gives rise to a Leibniz algebra to the bialgebra level.

Recall \cite{BLST} that a \textbf{Leibniz coalgebra} is a pair $(A, \delta_{[\cdot,\cdot]})$ where $A$ is a vector space and $\delta_{[\cdot,\cdot]}: A \to A \otimes A$ is a linear map satisfying
 \begin{eqnarray}
&(\id \otimes \delta_{[\cdot,\cdot]})\delta_{[\cdot,\cdot]} = ( \delta_{[\cdot,\cdot]}  \otimes \id )\delta_{[\cdot,\cdot]}+(\tau \otimes \id)(\id \otimes \delta_{[\cdot,\cdot]})\delta_{[\cdot,\cdot]}.
 \label{Leibnizco}
\end{eqnarray}
  
\begin{defi} \cite{BLST,TS} 
A \textbf{Leibniz bialgebra} is a triple $(A,[\cdot,\cdot],\delta_{[\cdot,\cdot]})$ such that 
  $(A,[\cdot,\cdot])$ is a Leibniz  algebra and  $(A,\delta_{[\cdot,\cdot]})$ is a Leibniz coalgebra satisfying the following compatibility conditions:
  \begin{align}
& \left(\mathrm{id} \otimes R_{[\cdot,\cdot]}(x)\right) \tau\delta_{[\cdot,\cdot]}(y)=\left(R_{[\cdot,\cdot]}(y) \otimes \mathrm{id}\right) \delta_{[\cdot,\cdot]}(x), \label{Leibbi1}\\
&\delta_{[\cdot,\cdot]}\left([x, y]\right)=\left(\mathrm{id}   \otimes   R_{[\cdot,\cdot]}(y)-L_{\square}(y)   \otimes   \mathrm{id}\right)  \delta_{\square}(x)+\left(\mathrm{id}   \otimes   L_{[\cdot,\cdot]}(x)+L_{[\cdot,\cdot]}(x)   \otimes   \mathrm{id}\right)\delta_{[\cdot,\cdot]}(y),\label{Leibbi2}
\end{align}
where $x \square y = [x ,y] + [y, x]$,  $\delta_{\square} = \delta_{[\cdot,\cdot]}+\tau\delta_{[\cdot,\cdot]}$  for all $x,y \in A$.
 \end{defi} 

We now extend Proposition \ref{diatoLeib} to the bialgebraic case.

\begin{thm}  \label{bidiatoLeibbi}
    Let $(A,  \rhd,\lhd,\delta_{\rhd},\delta_{\lhd})$ be a bi-dialgebra. Then $(A,[\cdot, \cdot],\delta_{[\cdot,\cdot]})$ is a Leibniz bialgebra where $[\cdot,\cdot]$ is given by Eq.~\eqref{inducedLeib} and  $\delta_{[\cdot,\cdot]}$ is defined by
  \begin{align}
 \delta_{[\cdot,\cdot]} =\delta_{\rhd} - \tau\delta_{\lhd}.
\end{align}  
We call $(A,[\cdot, \cdot],\delta_{[\cdot,\cdot]})$  the {\bf sub-adjacent Leibniz bialgebra} of $(A,  \rhd,\lhd,\delta_{\rhd},\delta_{\lhd})$.
\end{thm}
\begin{proof}
It follows from Proposition \ref{diatoLeib} that $(A,[\cdot, \cdot])$ is a Leibniz algebra. Furthermore, one can easily verify that   $(A, \delta_{[\cdot,\cdot]})$ is a Leibniz coalgebra. Let $x, y \in A$. Applying Eqs.~\eqref{bidicd1}-\eqref{bidicd3}, we have
\begin{align*}
&\left(\mathrm{id} \otimes R_{[\cdot,\cdot]}(x)\right) \tau\delta_{[\cdot,\cdot]}(y)-\left(R_{[\cdot,\cdot]}(y) \otimes \mathrm{id}\right) \delta_{[\cdot,\cdot]}(x)\\
 &=\left(\mathrm{id} \otimes (R_{\rhd}-L_{\lhd})(x)\right) (\tau\delta_{\rhd} -  \delta_{\lhd})(y)-\left((R_{\rhd}-L_{\lhd})(y) \otimes \mathrm{id}\right) (\delta_{\rhd} - \tau\delta_{\lhd})(x) =0.
\end{align*}
Thus Eq.~\eqref{Leibbi1}  holds. Similarly, by a direct verification, Eq.~\eqref{Leibbi2} also holds. This completes the proof.
\end{proof}

\begin{rmk}
In \cite{B2}, the author introduce antisymmetric infinitesimal bialgebras, the bialgebra structure for associative algebras, and prove that such a structure naturally induces a Lie bialgebra, thereby lifting the classcial relation from associative algebras to Lie algebras to the bialgebra level. As demonstrated by Theorem \ref{bidiatoLeibbi}, a bi-dialgebra is precisely the bialgebra structure for dialgebras, and it naturally gives rise to a Leibniz bialgebra. Consequently, this extends the classical correspondence between dialgebras and Leibniz algebras to the bialgebra context.
\end{rmk}
  
   \subsection{Coboundary  bi-dialgebras and diassociative  Yang-Baxter equations}  In this subsection, we investigate coboundary bi-dialgebras, which motivates us to introduce the diassociative Yang-Baxter equation (DAYBE). The DAYBE serves as an analogue of the classical Yang-Baxter equation for Lie algebras and the associative Yang-Baxter equation for associative algebras.

Let  $V$  be a vector space and  $r=\sum_{i} a_{i} \otimes b_{i} \in V \otimes V$. Set
\begin{align*}
r_{12}=\sum_{i} a_{i} \otimes b_{i} \otimes 1, \quad r_{13}=\sum_{i} a_{i} \otimes 1 \otimes b_{i},  \quad 
r_{23}=\sum_{i} 1 \otimes a_{i} \otimes b_{i}, 
\end{align*}
where $1$ is a symbol playing a similar role of the unit. If in addition, there exists a binary operation  $\diamond: V \otimes V \rightarrow V$  on  $V$, then the operation between two  $r$s is in an obvious way. For example,
\begin{align*}
r_{12} \diamond r_{13}=\sum_{i, j} a_{i} \diamond a_{j} \!\otimes\! b_{i} \!\otimes\! b_{j},& \;\;r_{13} \diamond r_{23}=\sum_{i, j} a_{i} \!\otimes\! a_{j} \!\otimes\! b_{i} \diamond b_{j},\;\;
r_{23} \diamond r_{12}=\sum_{i, j}  a_{j} \!\otimes\! a_{i} \diamond b_{j} \!\otimes\! b_{i}.
\end{align*}
Note that the above equation is independent of the existence of a unit.

\begin{defi}
Let  $(A,\rhd,\lhd)$   be a dialgebra and   $r  \in A \otimes A$.   Set
\begin{align}
\mathbf{A} (r):&=     r_{13} \rhd r_{23} -  r_{23} \lhd r_{12} -  r_{12} \star r_{13}.\label{DAYBE1}
\end{align}
 Then $r$ is called a solution of the {\bf diassociative Yang-Baxter equation (DAYBE)} in $(A,\rhd,\lhd)$ if $\mathbf{A} (r) = 0$.
\end{defi}

\begin{lem}\label{ryutaur}
Let  $(A,\rhd,\lhd)$   be a dialgebra and $r  \in A \otimes A$. Suppose the skew-symmetric part of $r$ is invariant. Then the following statements are equivalent:
\begin{enumerate}
\item\label{pqr:1} $r$ is a solution of the diassociative Yang-Baxter equation in  $(A,\rhd,\lhd)$.
\item\label{pqr:1.5} $\tau(r)$ is a solution of the diassociative Yang-Baxter equation in  $(A,\rhd,\lhd)$.
\item\label{pqr:2}  $\mathbf{B} (r):=    r_{13} \lhd r_{23} +  r_{23} \star r_{12} -  r_{12} \rhd r_{13} = 0$.
\item\label{pqr:3}  $\mathbf{C} (r):=    r_{13} \star r_{23} -  r_{23} \rhd r_{12} +  r_{12} \lhd r_{13}= 0$.
\end{enumerate}
\end{lem}
\begin{proof}
Let $r=\sum_{i} a_{i} \otimes b_{i}   \in A \otimes A$ and   $\alpha:= r -\tau(r)$ be the skew-symmetric part of $r$. Then  $\alpha$  is invariant.

\eqref{pqr:1} $ \Longleftrightarrow $ \eqref{pqr:1.5}. Let $\sigma_{123} \in \operatorname{End}(A \otimes A \otimes A)$ be the linear map  defined by  $\sigma_{123} (x \otimes y \otimes z) = z \otimes x \otimes y$ for all $x,y,z \in A$.  
\begin{align*}
    \sigma_{123}(\mathbf{A} (\tau(r)) )
    &= \sum_{i,j} a_i \rhd a_j\otimes b_i \otimes b_j -  a_j \otimes b_i \otimes b_j \lhd  a_i -   a_j \otimes b_i \star b_j \otimes a_i\\
      &=  \sigma_{123}(\mathbf{A} (r)) + \sum_{i}a_i \otimes  E(b_i)\alpha \\
       &=  \sigma_{123}(\mathbf{A} (r)).
\end{align*}
Thus $\mathbf{A} (r) = 0$ if and only if $\mathbf{A} (\tau(r)) = 0$.

\eqref{pqr:1} $ \Longleftrightarrow $ \eqref{pqr:2}. In fact, we have
\begin{align*}
    \sigma_{123}(\mathbf{A} (r)) 
    &= \sum_{i,j} b_i \rhd b_j\otimes a_i \otimes a_j -  b_i \otimes a_j \otimes a_i \lhd  b_j -   b_j \otimes a_i \star a_j \otimes b_i\\
      &=(\tau \otimes \id)(a_i \otimes \tau( (E(b_i)-F(b_i) )\alpha)-\mathbf{B} (r) -\sum_{i}F(a_i)\alpha \otimes b_i\\
       &= -\mathbf{B} (r).
\end{align*}
Thus $\mathbf{A} (r) = 0$  if and only if $\mathbf{B}(r) = 0$. 
	
	\eqref{pqr:2} $ \Longleftrightarrow $ \eqref{pqr:3}. It is straightforward to verify that  $\sigma_{132}( \mathbf{C} (r)) =  \mathbf{B} (r)$ where $\sigma_{132} \in \operatorname{End}(A \otimes A \otimes A)$ is defined by  $\sigma_{132} (x \otimes y \otimes z) = y \otimes z \otimes x$ for all $x,y,z \in A$.  Thus $\mathbf{B} (r) = 0$ if and only if $\mathbf{C} (r) = 0$.
\end{proof}

\begin{defi}A  bi-dialgebra $(A,\rhd,\lhd, \delta_{\rhd},\delta_{\lhd})$   is  called {\bf coboundary} if there exists an $r \in  A \otimes A$ such that 
\begin{align}
{\delta_{\rhd}}(x) :={\delta_{\rhd,r}}(x) := E(x)r ,\quad  {\delta_{\lhd}}(x) :={\delta_{\lhd,r}}(x):=
F(x)r, \quad \forall x \in A, \label{EF}
\end{align}
where  the linear maps  $E,F: A \rightarrow
\operatorname{End}(A \otimes A)$ are defined
respectively by Eqs.~\eqref{eq:e}-\eqref{eq:f}. We also denote the coboundary bi-dialgebra by $(A,\rhd,\lhd, \delta_{\rhd,r},\delta_{\lhd,r})$.
\end{defi}

\begin{pro}\label{cobdppb}
Let  $(A,\rhd,\lhd)$   be a dialgebra and  $r \in A \otimes A$. Define  linear maps ${\delta_{\rhd,r}}, {\delta_{\lhd,r}}: A \rightarrow A \otimes A$ by Eq.~\eqref{EF}.
\begin{enumerate}
\item\label{tdb:co1} $(A,\delta_{\rhd,r}, \delta_{\lhd,r})$ is a co-dialgebra if and only if the following equations hold:
    \begin{align}
    ( \id \otimes \id \otimes  L_{\lhd}(x)  )\mathbf{A} (r)+(  R_{\star}(x)\otimes   \id \otimes \id   )\mathbf{B} (r) &= 0, \label{tdb:eqco1}\\
          ( \id \otimes \id \otimes  L_{\star}(x)  )\mathbf{B} (r)+(  R_{\rhd}(x)\otimes   \id \otimes \id   )\mathbf{C} (r) &= 0,\label{tdb:eqco2}\\
   (  R_{\rhd}(x)\otimes   \id \otimes \id   )\mathbf{A} (r) &= 0,\label{tdb:eqco3}\\
       ( \id \otimes \id \otimes  L_{\star}(x)  )\mathbf{A} (r)+(  R_{\star}(x)\otimes   \id \otimes \id   )\mathbf{C} (r) &= 0,\label{tdb:eqco4}\\
             ( \id \otimes \id \otimes  L_{\lhd}(x)  )\mathbf{C} (r) &= 0,\quad \forall x \in A.\label{tdb:eqco5}
    \end{align}
    \item\label{tdb:co2} $(A, \rhd,\lhd, \delta_{\rhd,r}, \delta_{\lhd,r})$ is a coboundary bi-dialgebra if and only if Eqs.~\eqref{tdb:eqco1}-\eqref{tdb:eqco5} and the following equations hold:
    \begin{align}
      (     L_{\lhd}(x)\otimes \id   )E(y)(r -\tau(r)) &= 0,\label{bitdb:eqco1}\\
      (    \id \otimes    R_{\rhd}(x)  )F(y)(r -\tau(r)) &= 0,\label{bitdb:eqco2}\\
      (     L_{\rhd}(x)\otimes \id  -  \id \otimes  R_{\lhd}(x))E(y)(r -\tau(r)) &= 0,\label{bitdb:eqco3}\\
      (     R_{\star}(x)\otimes \id  -  \id \otimes  R_{\rhd}(x))(E(y)-F(y))(r -\tau(r)) &= 0,\quad \forall x,y \in A, \label{bitdb:eqco4}
    \end{align}
    where   $E,F: A \rightarrow
\operatorname{End}(A \otimes A)$ are given
respectively by Eqs.~\eqref{eq:e}-\eqref{eq:f}. 
 \end{enumerate}

    \end{pro}
    \begin{proof} Let  $r = \sum\limits_i a_i \otimes b_i \in A \otimes A$.

 \eqref{tdb:co1}. Let $x \in A$. Then we have
\begin{align*}
&    (\id \otimes \delta_{\rhd,r})\delta_{\rhd,r}(x) - ( \delta_{\rhd,r} \otimes \id )\delta_{\rhd,r}(x) \\
&=\!\sum_{i,j}\!\Big( a_i  \star x \!\otimes\! a_j \star b_i \!\otimes\! b_j \!+\!   a_i  \star x \!\otimes\! a_j  \!\otimes\!  b_i  \lhd b_j \!+ \!a_i \!\otimes \! a_j  \star (x \lhd b_i) \!\otimes \!b_j \!+  \! a_i \!\otimes \!a_j \!\otimes \!  (x \lhd b_i)  \lhd b_j  \\
&\quad- a_j \star (a_i  \star x) \otimes   b_j\!\otimes\! b_i \!- \!  a_j   \! \otimes\! (a_i \star x) \lhd b_j \!\otimes\!  b_i   \!- \!a_j \star a_i \!\otimes \! b_j \!\otimes\!   x \lhd  b_i \!-\!  a_j \!\otimes \!a_i \lhd b_j \otimes    x \lhd b_i   \Big) \\
&=S_{1}+ S_{2}+S_{3},
 \end{align*}
where
\begin{align*}
S_{1} &=\sum_{i,j} \Big(a_i  \star x  \otimes  a_j \star b_i  \otimes  b_j  +    a_i  \star x  \otimes  a_j   \otimes   b_i  \lhd b_j   - a_j \star (a_i  \star x) \otimes   b_j \otimes  b_i  \Big)\\
&=\sum_{i,j} \Big((R_{\star}(x)\otimes \id \otimes \id)(r_{23}\star r_{12}+r_{13}\lhd r_{23})- (a_j \rhd a_i ) \star x \otimes   b_j \otimes  b_i  \Big)\\
&= (R_{\star}(x)\otimes \id \otimes \id)(r_{23}\star r_{12}+r_{13}\lhd r_{23}-r_{12}\rhd r_{13})  \\
&= (R_{\star}(x)\otimes \id \otimes \id)\mathbf{B} (r),  \\
S_{2} &=\sum_{i,j} \Big(     a_i  \otimes  a_j  \otimes    (x \lhd b_i)  \lhd b_j -  a_j \star a_i  \otimes   b_j  \otimes    x \lhd  b_i  -   a_j  \otimes  a_i \lhd b_j \otimes    x \lhd b_i  \Big)\\
&=\sum_{i,j} \Big(( \id \otimes \id \otimes L_{\lhd}(x))(r_{13}\rhd r_{23}- r_{12}\star r_{13} - r_{23}\lhd r_{12})  \Big)\\
&= ( \id \otimes \id \otimes L_{\lhd}(x))\mathbf{A} (r),\\
S_{3} &=\sum_{i,j} \Big( a_i  \otimes   a_j  \star (x \lhd b_i)  \otimes  b_j   -    a_j     \otimes  (a_i \star x) \lhd b_j  \otimes   b_i  \Big) =0.
\end{align*}
Then
 \begin{align*}
S_{1}+S_{2}+S_{3} &= ( \id \otimes \id \otimes  L_{\lhd}(x)  )\mathbf{A} (r)+(  R_{\star}(x)\otimes   \id \otimes \id   )\mathbf{B} (r) . 
 \end{align*}
Thus $(A,\delta_{\rhd,r} )$ is an associative  coalgebra if and only if Eq.~\eqref{tdb:eqco1} holds. Moreover, we have
\begin{align*}
&    (\id \otimes \delta_{\lhd,r})\delta_{\lhd,r}(x) - ( \delta_{\lhd,r} \otimes \id )\delta_{\lhd,r}(x) \\
&=\!\sum_{i,j}\!\Big( a_i  \rhd x \!\otimes\! a_j \rhd b_i \!\otimes\! b_j \!-\!   a_i  \rhd x \!\otimes\! a_j  \!\otimes\!  b_i  \star b_j \!- \!a_i \!\otimes \! a_j  \rhd (x \star b_i) \!\otimes \!b_j \!+  \! a_i \!\otimes \!a_j \!\otimes \!  (x \star b_i)  \star b_j  \\
&\quad- a_j \rhd (a_i \rhd  x) \!\otimes  \! b_j\!\otimes\! b_i \!+ \!  a_j   \! \otimes\! (a_i \rhd x) \star b_j \!\otimes\!  b_i   \!+ \!a_j \rhd a_i \!\otimes \! b_j \!\otimes\!   x \star  b_i \!-\!  a_j \!\otimes \!a_i \star b_j \otimes    x \star b_i   \Big) \\
&=U_{1}+ U_{2}+U_{3},
 \end{align*}
where
\begin{align*}
U_{1} &=\sum_{i,j} \Big(a_i  \rhd x  \otimes  a_j \rhd b_i  \otimes  b_j  -    a_i  \rhd x  \otimes  a_j   \otimes   b_i  \star b_j   - a_j \rhd (a_i \rhd  x)  \otimes    b_j \otimes  b_i  \Big)\\
&=\sum_{i,j} \Big((R_{\rhd}(x)\otimes \id \otimes \id)(r_{23}\rhd r_{12} - r_{13}\star r_{23})- (a_j \lhd a_i ) \rhd   x \otimes   b_j \otimes  b_i  \Big)\\
&= (R_{\rhd}(x)\otimes \id \otimes \id)(r_{23}\rhd r_{12} - r_{13}\star r_{23}-r_{12}\lhd r_{13})  \\
&= -(R_{\rhd}(x)\otimes \id \otimes \id)\mathbf{C} (r),  \\
U_{2} &=\sum_{i,j} \Big(     a_i  \otimes  a_j  \otimes    (x \star b_i)  \star b_j  +  a_j \rhd a_i  \otimes   b_j  \otimes    x \star  b_i  -   a_j  \otimes  a_i \star b_j \otimes    x \star b_i  \Big)\\
&=\sum_{i,j} \Big(( \id \otimes \id \otimes L_{\star}(x))(  r_{12}\rhd r_{13} - r_{23}\star r_{12}) -   a_i  \otimes  a_j  \otimes    x \star (b_i  \lhd b_j ) \Big)\\
&= ( \id \otimes \id \otimes L_{\star}(x))(  r_{12}\rhd r_{13} - r_{23}\star r_{12} - r_{13}\lhd r_{23})  \\
&=- ( \id \otimes \id \otimes L_{\star}(x))\mathbf{B} (r),\\
U_{3} &=\sum_{i,j} \Big(    a_j     \otimes  (a_i \rhd x) \star b_j  \otimes   b_i   -  a_i  \otimes   a_j  \rhd (x \star b_i)  \otimes  b_j   \Big) =0.
\end{align*}
Then
 \begin{align*}
U_{1}+U_{2}+U_{3} &= - ( \id \otimes \id \otimes L_{\star}(x))\mathbf{B} (r)-(R_{\rhd}(x)\otimes \id \otimes \id)\mathbf{C} (r). 
 \end{align*}
Thus $(A,\delta_{\lhd,r} )$ is an associative  coalgebra if and only if Eq.~\eqref{tdb:eqco2} holds.  By a similar calculation, Eqs.~\eqref{codiass1}-\eqref{codiass3} hold   if and only if Eqs.~\eqref{tdb:eqco3}-\eqref{tdb:eqco5} hold,  respectively.

 \eqref{tdb:co2}. Let $x,y \in A$. Then we have
\begin{align*}
 &  (\operatorname{id} \otimes L_{\lhd}(x)  ) \delta_{\lhd,r}(y) -( R_{\rhd}(y) \otimes \operatorname{id}) \delta_{\rhd,r}(x)\\
 & =a_i \rhd y \otimes x \lhd b_i - a_i  \otimes x \lhd (y \star b_i)  - (a_i \star x )\rhd y \otimes b_i - a_i \rhd y \otimes x \lhd b_i  
 =0,\\
 &( L_{\lhd}(x) \otimes \operatorname{id}  ) \delta_{\rhd,r}(y)-
 (\operatorname{id} \otimes L_{\lhd}(y)  ) \tau \delta_{\rhd,r}(x) \\ 
  & = x \lhd (a_i \star y) \otimes b_i + x \lhd a_i \otimes y \lhd b_i -  x \lhd b_i \otimes y \lhd a_i\\
  &= (L_{\lhd}(x)  \otimes \id)E(y)(r -\tau(r)) = 0.
\end{align*}
Hence Eq. ~\eqref{bidicd1} holds automatically,   Eq. ~\eqref{bidicd2} holds  if and only if   Eq. ~\eqref{bitdb:eqco1} holds. By a similar calculation, we find that Eq.~\eqref{bidicd3}  holds  if and only if   Eq. ~\eqref{bitdb:eqco2} holds, Eqs.~\eqref{bidicd4}-\eqref{bidicd7}  holds  automatically, Eqs.~\eqref{bidicd8}-\eqref{bidicd9}  hold  if and only if   Eq. ~\eqref{bitdb:eqco3}-\eqref{bitdb:eqco4} hold,  respectively. 
\end{proof}
   

Consequently, we arrive at the following main result of this paper.

\begin{thm}\label{tcdppab:1}
Let  $(A,\rhd,\lhd)$   be a dialgebra and $r \in A \otimes A$. Define  linear maps ${\delta_{\rhd,r}}, {\delta_{\lhd,r}}: A \rightarrow A \otimes A$ by Eq.~\eqref{EF}. 
If $r$ is a solution of the DAYBE and the skew-symmetric part of $r$ is invariant, then $(A,\rhd,\lhd,{\delta_{\rhd,r}}, {\delta_{\lhd,r}})$ is a bi-dialgebra. In particular, if  $r$ is a symmetric solution of the DAYBE, then $(A,\rhd,\lhd,{\delta_{\rhd,r}}, {\delta_{\lhd,r}})$ is a bi-dialgebra.
\end{thm}
\begin{proof}
It follows from Lemma \ref{ryutaur} and  Proposition \ref{cobdppb}.
\end{proof}

   \subsection{Quasi-triangular bi-dialgebra}

\begin{defi}\label{quasidb}
	Let  $(A,\rhd,\lhd)$   be a dialgebra.  If $r$ is a solution of the DAYBE in $(A,\rhd,\lhd)$  and the skew-symmetric part of $r \in A \otimes  A$ is  invariant, then the coboundary bi-dialgebra $(A,\rhd,\lhd,{\delta_{\rhd,r}}, {\delta_{\lhd,r}})$ induced by $r$ is called a \textbf{quasi-triangular bi-dialgebra}.
	In particular, if $r$ is a symmetric solution of the DAYBE in $(A,\rhd,\lhd)$, then   $(A,\rhd,\lhd,{\delta_{\rhd,r}}, {\delta_{\lhd,r}})$  is called a \textbf{triangular bi-dialgebra}.
\end{defi}

\begin{pro}\label{tquasidba}
	Let  $(A,\rhd,\lhd)$   be a dialgebra and $r \in A \otimes A$.  
	Then $(A,\rhd,\lhd,{\delta_{\rhd,r}}, {\delta_{\lhd,r}})$ is a  quasi-triangular bi-dialgebra if and only if $(A,\rhd,\lhd,{\delta_{\rhd,\tau(r)}}, {\delta_{\lhd,\tau(r)}})$ is a quasi-triangular bi-dialgebra.
\end{pro}
\begin{proof}
	It follows from Lemma~\ref{ryutaur}.
\end{proof}

Next we turn to study the operator form of the DAYBE.

\begin{pro}\label{operaforDAYBE}
 Let  $(A,\rhd,\lhd)$   be a dialgebra and $r \in A \otimes A$ whose skew-symmetric part is
invariant. Then   $r$  is a solution of the DAYBE on $(A,\rhd,\lhd)$  if and only if  the following equations hold:
 \begin{align}
   \widetilde{r}\left(a^{*}\right) \rhd \widetilde{r}\left(b^{*}\right)&=\widetilde{r}\left(-R_{\lhd}^{*}\left(\widetilde{r}\left(a^{*}\right)\right) b^{*}-L_{\star}^{*}\left(\widetilde{r}^t\left(b^{*}\right)\right) a^{*}\right),\label{szxsDAYBE1}\\
\widetilde{r}\left(a^{*}\right) \lhd \widetilde{r}\left(b^{*}\right)&=\widetilde{r}\left(R_{\star}^{*}\left(\widetilde{r}\left(a^{*}\right)\right) b^{*}-L_{\rhd}^{*}\left(\widetilde{r}^t\left(b^{*}\right)\right) a^{*}\right), \quad \forall a^{*}, b^{*} \in A^*. \label{szxsDAYBE2}
   \end{align}
\end{pro}
\begin{proof}Let $r=\sum\limits_{i} a_i \otimes b_i$. For all $a^{*},b^{*},c^{*} \in A^*$, we have
\begin{align*}
\left\langle \widetilde{r}\left(a^{*}\right) \rhd\widetilde{r}\left(b^{*}\right), c^{*}\right\rangle &= -\left\langle  \widetilde{r}\left(b^{*}\right), L_{\rhd}^*(\widetilde{r}(a^{*}))c^{*}\right\rangle =  -\sum_{i}\left\langle  b^{*} ,a_i\right\rangle  \left\langle   a^{*} \otimes R_{\rhd}^{*}(b_i)c^{*} ,r  \right\rangle\\
&= -\sum_{i,j} \left\langle   a^{*}  ,a_j  \right\rangle \left\langle  b^{*} ,a_i\right\rangle \left\langle R_{\rhd}^{*}(b_i)c^{*} ,b_j \right\rangle  
 = \left\langle  a^{*}  \otimes b^{*} \otimes c^{*},r_{13}\rhd r_{23}\right\rangle,\\
\left\langle \widetilde{r}(-R_{\lhd}^{*}\left(\widetilde{r}\left(a^{*}\right)\right) b^{*}), c^{*}\right\rangle &= \left\langle  -R_{\lhd}^{*}\left(\widetilde{r}\left(a^{*}\right)\right) b^{*} \otimes c^{*}, r\right\rangle  =  -\sum_{i}\left\langle  L_{\lhd}^{*}\left(a_i\right)b^{*}  , \widetilde{r}\left(a^{*}\right)\right\rangle \left\langle   c^{*}, b_i\right\rangle \\
&=  -\sum_{i,j}\left\langle   a^{*} , a_j\right\rangle  \left\langle   L_{\lhd}^{*}\left(a_i\right)b^{*}  , b_j\right\rangle \left\langle   c^{*}, b_i\right\rangle 
 = \left\langle   a^{*} \otimes b^{*}\otimes c^*  , r_{23} \lhd r_{12} \right\rangle, \\
\left\langle \widetilde{r}(-L_{\star}^{*}\left(\widetilde{r}^t\left(b^{*}\right)\right) a^{*}), c^{*}\right\rangle &= -\left\langle  L_{\star}^{*}\left(\widetilde{r}^t\left(b^{*}\right)\right) a^{*} \otimes c^{*}, r\right\rangle 
 =  -\sum_{i}\left\langle   R_{\star}^{*}\left(a_i\right)a^{*}  , \widetilde{r}^t\left(b^{*}\right)\right\rangle \left\langle   c^{*}, b_i\right\rangle  \\
 &=   \sum_{i,j} \left\langle    a^{*}  , a_j \star a_i\right\rangle \left\langle   b^{*} , b_j\right\rangle \left\langle   c^{*}, b_i\right\rangle =   \left\langle   a^{*} \otimes b^{*}\otimes c^*  ,  r_{12} \star r_{13}\right\rangle.
\end{align*}
Then  $\mathbf{A} (r) = 0$ if and only if   Eq.~\eqref{szxsDAYBE1} holds. Similarly, we have
\begin{align*}
\left\langle \widetilde{r}\left(a^{*}\right) \lhd\widetilde{r}\left(b^{*}\right), c^{*}\right\rangle &=    \langle  a^{*}  \otimes b^{*} \otimes c^{*},\sum_{i,j} a_j \otimes a_i \otimes  b_j \lhd b_i  \rangle
=: \left\langle  a^{*}  \otimes b^{*} \otimes c^{*}, r_{13}  \lhd r_{23} \right\rangle,\\
\left\langle
\widetilde{r}( R_{\star}^{*}\left(\widetilde{r}\left(a^{*}\right)\right)
b^{*}), c^{*}\right\rangle &=  -\langle  a^{*}  \otimes b^{*}
\otimes c^{*}, \sum_{i,j} a_j \otimes   a_i \star b_j  \otimes b_i
 \rangle
=: \left\langle   a^{*} \otimes b^{*}\otimes c^*  ,  -r_{23}\star  r_{12}  \right\rangle, \\
 \left\langle \widetilde{r}(-L_{\rhd}^{*}\left(\widetilde{r}^t \left(b^{*}\right)\right) a^{*}), c^{*}\right\rangle &=   \langle  a^{*}
\otimes b^{*} \otimes c^{*},\sum_{i,j} a_j \rhd  a_i \otimes b_j
\otimes  b_i \rangle=:  \left\langle   a^{*} \otimes
b^{*}\otimes c^*  , r_{12} \rhd r_{13}\right\rangle.
\end{align*}
Thus   $\mathbf{B} (r) = 0$ if and only if   Eq.~\eqref{szxsDAYBE2} holds.   Hence it follows from  Lemma~\ref{ryutaur}.
\end{proof}

\begin{lem}\label{wrsqodi}
 Let  $(A,\rhd,\lhd)$   be a dialgebra and $r \in A \otimes A$ whose skew-symmetric part is
invariant.  Define  linear maps ${\delta_{\rhd,r}}, {\delta_{\lhd,r}}: A \rightarrow A \otimes A$ by Eq.~\eqref{EF} and $\rhd_{r},\lhd_{r}: A^* \otimes A^* \to A^*$ be the linear duals of  $\delta_{\rhd,r}$ and $\delta_{\lhd,r}$ respectively. Suppose that $r$ is a solution of the DAYBE in $(A,\rhd,\lhd)$. Then the dialgebra structure $\rhd_{r},\lhd_{r}$ on  $A^{*}$ are respectively given by
\begin{align}
a^{*} \rhd_{r} b^{*}&=-R_{\lhd}^{*}(\widetilde{r}(a^{*})) b^{*}-L_{\star}^{*}(\widetilde{r}^t(b^{*})) a^{*},  \label{wsddiybe1}\\
a^{*} \lhd_{r} b^{*}&=R_{\star}^{*}(\widetilde{r}(a^{*})) b^{*}-L_{\rhd}^{*}(\widetilde{r}^t(b^{*})) a^{*}, \quad \forall a^{*}, b^{*} \in A^{*}.\label{wsddiybe2}
\end{align}
Moreover, $\widetilde{r}: (A^*,\rhd_{r},\lhd_{r}) \to(A, \rhd,\lhd) $  is a  homomorphism of dialgebras.
\end{lem}
\begin{proof}For all $x \in A$, $a^{*}, b^{*} \in A^{*}$, we have
       \begin{align*}
\left\langle a^{*} \rhd_{r} b^{*}, x\right\rangle=&\left\langle a^{*} \otimes b^{*}, \delta_{\rhd,r}(x)\right\rangle 
= -\left\langle a^{*} \otimes L_{\lhd}^*(x)b^{*} ,  r\right\rangle  - \left\langle  R_{\star}^*(x) a^{*} \otimes b^{*},  r \right\rangle\\
=&-\left\langle L_{\lhd}^*(x)b^{*} ,  \widetilde{r}\left(a^{*}\right)\right\rangle   - \left\langle  R_{\star}^*(x) a^{*},  \widetilde{r}^t\left(b^{*}\right) \right\rangle 
= \left\langle -R_{\lhd}^{*}(\widetilde{r}(a^{*})) b^{*}-L_{\star}^{*}(\widetilde{r}^t(b^{*})) a^{*},  x \right\rangle,\\
\left\langle a^{*} \lhd_{r} b^{*}, x\right\rangle=&\left\langle a^{*} \otimes b^{*}, \delta_{\rhd,r}(x)\right\rangle 
=  \left\langle a^{*} \otimes L_{\star}^*(x)b^{*} ,  r\right\rangle  - \left\langle  R_{\rhd}^*(x) a^{*} \otimes b^{*},  r \right\rangle\\
=& \left\langle L_{\star}^*(x)b^{*},  \widetilde{r}\left(a^{*}\right)\right\rangle   - \left\langle  R_{\rhd}^*(x) a^{*} ,  \widetilde{r}^t\left(b^{*}\right) \right\rangle 
= \left\langle  R_{\star}^{*}(\widetilde{r}(a^{*})) b^{*}-L_{\rhd}^{*}(\widetilde{r}^t(b^{*})) a^{*},  x \right\rangle.
\end{align*}
Thus Eqs.~\eqref{wsddiybe1}-\eqref{wsddiybe2} holds. Moreover by Proposition \ref{operaforDAYBE}, we  have 
\begin{align*}
\widetilde{r}(a^{*} \rhd_{r} b^{*})&=\widetilde{r}\left(-R_{\lhd}^{*}(\widetilde{r}(a^{*})) b^{*}-L_{\star}^{*}(\widetilde{r}^t(b^{*})) a^{*}\right) = \widetilde{r}\left(a^{*}\right) \rhd \widetilde{r}\left(b^{*}\right),\\
\widetilde{r}(a^{*} \lhd_{r} b^{*})&=\widetilde{r}\left( R_{\star}^{*}(\widetilde{r}(a^{*})) b^{*}-L_{\rhd}^{*}(\widetilde{r}^t(b^{*})) a^{*}\right) = \widetilde{r}\left(a^{*}\right) \lhd \widetilde{r}\left(b^{*}\right).
\end{align*}
This completes the proof.
\end{proof}

\begin{thm}\label{thm:daybe}
	 Let  $(A,\rhd,\lhd)$   be a dialgebra  and $r \in A \otimes A$. Suppose  the skew-symmetric part $\alpha$ of $r$  is invariant.  
	Then $r$ is a solution of the DAYBE in $(A,\rhd,\lhd)$ if and only if  $(A^*,\rhd_r,\lhd_r)$ is a dialgebra and the linear maps $\widetilde{r}, \widetilde{r}^t: (A^*,\rhd_r,\lhd_r) \to (A,\rhd,\lhd)$ are homomorphisms of dialgebras, where $\rhd_r,\lhd_r: A^* \otimes A^* \to A^*$ are defined by Eqs.~\eqref{wsddiybe1} and \eqref{wsddiybe2}, respectively.
\end{thm}
\begin{proof}
	($\Longrightarrow$). By Definition \ref{quasidb}, $(A,\rhd,\lhd,{\delta_{\rhd,r}}, {\delta_{\lhd,r}})$ is a quasi-triangular bi-dialgebra with ${\delta_{\rhd,r}}$ and  ${\delta_{\lhd,r}}$ given by Eq.~\eqref{EF}. Thus by Lemma \ref{wrsqodi}, $(A^*,\rhd_r,\lhd_r)$ is a dialgebra and $\widetilde{r}$  is a  homomorphism of dialgebras. Applying Lemma \ref{semiinequ}, Lemma \ref{ryutaur} and Proposition \ref{operaforDAYBE}, for all $a^*,b^* \in A^*$, we have
\begin{align*}
\widetilde{r}^t(a^{*} \rhd_{r} b^{*})&=\widetilde{r}^t\left(-R_{\lhd}^{*}(\widetilde{r}(a^{*})) b^{*}-L_{\star}^{*}(\widetilde{r}^t(b^{*})) a^{*}\right)\\
&= \widetilde{r}^t\left(-R_{\lhd}^{*}(\widetilde{\alpha}(a^{*})) b^{*}-R_{\lhd}^{*}(\widetilde{r}^t(a^{*})) b^{*}-L_{\star}^{*}(\widetilde{r}(b^{*})) a^{*}+L_{\star}^{*}(\widetilde{\alpha}(b^{*})) a^{*}\right)\\
&= \widetilde{r}^t\left( -R_{\lhd}^{*}(\widetilde{r}^t(a^{*})) b^{*}-L_{\star}^{*}(\widetilde{r}(b^{*})) a^{*}\right)\\
& = \widetilde{r}^t\left(a^{*}\right) \rhd \widetilde{r}^t\left(b^{*}\right),\\
\widetilde{r}^t(a^{*} \lhd_{r} b^{*})&=\widetilde{r}^t\left(R_{\star}^{*}(\widetilde{r}(a^{*})) b^{*}-L_{\rhd}^{*}(\widetilde{r}^t(b^{*})) a^{*}\right)\\
& =\widetilde{r}^t\left(R_{\star}^{*}(\widetilde{r}^t(a^{*})) b^{*} + R_{\star}^{*}(\widetilde{\alpha}(a^{*})) b^{*}-L_{\rhd}^{*}(\widetilde{r}(b^{*})) a^{*}+L_{\rhd}^{*}(\widetilde{\alpha} (b^{*})) a^{*}\right)\\
& =\widetilde{r}^t\left(R_{\star}^{*}(\widetilde{r}^t(a^{*})) b^{*}  -L_{\rhd}^{*}(\widetilde{r}(b^{*})) a^{*} \right)\\
& = \widetilde{r}^t\left(a^{*}\right) \lhd \widetilde{r}^t\left(b^{*}\right).
\end{align*}
	Thus $\widetilde{r}^t: (A^*,\rhd_r,\lhd_r) \to (A,\rhd,\lhd)$ is also a homomorphism of dialgebra.

	($\Longleftarrow$).
	It follows from Proposition \ref{operaforDAYBE}. 
\end{proof}

We investigate $\mathcal{O}$-operators of arbitrary weight on dialgebras, which serve as the operator form of solutions to the DAYBE whose skew-symmetric parts are invariant.

Recall \cite{BGN2013} that an {\bf $A$-bimodule algebra} of an associative algebra $(A,\cdot_A)$  is a quadruple $(V,\cdot_V;l,r)$ such that $(V,\cdot_V)$ is an associative algebra, $(V;l,r)  $ is a bimodule of $(A,\cdot_A)$ and the following equations hold:
\begin{align}
l (x)(a \cdot_V  b) & = \left(l (x) a\right) \cdot_V  b,\quad 
r (x)(a \cdot_V  b)   = a \cdot_V \left(r (x) b\right), \quad 
 \left(r (x) a\right) \cdot_V  b   =   a \cdot_V \left(l (x) b\right), \label{ass:mpofAss3}
\end{align}
for all $x \in A, a,b \in V$.
\begin{defi} \label{defArep} An \textbf{$A$-bimodule dialgebra} of a dialgebra $(A,\rhd_A,\lhd_A)$  is a septuple $(V,\rhd_V,\lhd_V;l_{\rhd}, r_{\rhd},$ $ l_{\lhd},  r_{\lhd})$ such that $(V,\rhd_V,\lhd_V)$ is a dialgebra, $(V;l_{\rhd}, r_{\rhd}, l_{\lhd},  r_{\lhd})$ is a bimodule of $(A,\rhd_A,\lhd_A)$, $(V,\rhd_V;$ $l_{\rhd}, r_{\rhd})$  and $(V,\lhd_V;l_{\lhd},  r_{\lhd})$ are $A$-bimodule algebras of associative algebras $(A,\rhd_A)$ and $(A,\lhd_A)$ respectively, and the following equations hold:
\begin{align}
l_{\lhd}(x)(a \lhd_V b) &= l_{\lhd}(x)(a \rhd_V b), &
a \lhd_V \bigl(l_{\lhd}(x)(b)\bigr) &= a \lhd_V \bigl(l_{\rhd}(x)(b)\bigr), \label{di:mpofDias1}\\
a \lhd_V \bigl(r_{\lhd}(x)(b)\bigr) &= a \lhd_V \bigl(r_{\rhd}(x)(b)\bigr), &
l_{\rhd}(x)(a \lhd_V b) &= \bigl(l_{\rhd}(x)a\bigr) \lhd_V b, \label{di:mpofDias2}\\
\bigl(r_{\rhd}(x)a\bigr) \lhd_V b &= a \rhd_V \bigl(l_{\lhd}(x)b\bigr), &
r_{\lhd}(x)(a \rhd_V b) &= a \rhd_V \bigl(r_{\lhd}(x)b\bigr), \label{di:mpofDias3}\\
\bigl(l_{\lhd}(x)a\bigr) \rhd_V b &= \bigl(l_{\rhd}(x)a\bigr) \rhd_V b, &
\bigl(r_{\lhd}(x)a\bigr) \rhd_V b &= \bigl(r_{\rhd}(x)a\bigr) \rhd_V b, \label{di:mpofDias4}\\
r_{\rhd}(x)(a \lhd_V b) &= r_{\rhd}(x)(a \rhd_V b), \label{di:mpofDias5}
\end{align}
for all $x \in A, a,b \in V$.
\end{defi}
\begin{rmk}
Similar to bimodules of dialgebras, $A$-bimodule dialgebra  $(V,\rhd_V,\lhd_V;  l_{\rhd}, $ $r_{\rhd}, l_{\lhd},  r_{\lhd})$ of a dialgebra $(A,\rhd_A,\lhd_A)$ has an equivalent characterization in terms of the dialgebra  structure $\rhd_{A \oplus V}, \lhd_{A \oplus V}$ on the direct sum  $A \oplus V$  of vector spaces with $\rhd_{A \oplus V}, \lhd_{A \oplus V}$ given by
\begin{align*}
(x+a) \rhd_{A \oplus V}(y+b)&=x \rhd_{A} y+l_{\rhd}(x) b+r_{\rhd}(y) a+a  \rhd_{V} b, \\
(x+a) \lhd_{A \oplus V}(y+b)&=x \lhd_{A} y+l_{\lhd}(x) b+l_{\lhd}(y) a+a \lhd_{V} b, \quad \forall x, y \in A, a, b \in V.
\end{align*}
\end{rmk}

\begin{ex}
Let $(A,\rhd,\lhd)$ be a dialgebra. Then $(A,\rhd,\lhd;L_{\rhd}, R_{\rhd}, L_{\lhd}, R_{\lhd})$ is an  $A$-bimodule dialgebra.
\end{ex}

\begin{defi}Let $(A,\rhd_A,\lhd_A)$ be a dialgebra and $(V,\rhd_V,\lhd_V;l_{\rhd}, r_{\rhd}, l_{\lhd},  r_{\lhd})$ be an $A$-bimodule dialgebra of $(A,\rhd_A,\lhd_A)$. A linear map $T: V \rightarrow A$ is called an \textbf{$\mathcal{O}$-operator of weight $\lambda \in \mathbb{F}$ on $(A,\rhd_A,\lhd_A)$ associated to $(V,\rhd_V,\lhd_V;l_{\rhd}, r_{\rhd}, l_{\lhd},  r_{\lhd})$} if 
\begin{align}
T(a) \rhd_A T(b) &= T\left(l_{\rhd}(T(a))b+r_{\rhd}(T(b)) a+\lambda a \rhd_V b\right), \label{wofn1}\\
T(a) \lhd_A T(b) &= T\left(l_{\lhd}(T(a))b+r_{\lhd}(T(b)) a+\lambda a \lhd_V b\right),\quad \forall a, b \in V. \label{wofn2} 
\end{align}
In particular, if $V$ is a  trivial dialgebra, then we simply say $T:V\rightarrow A$ is an {\bf $\mathcal{O}$-operator on $(A,\rhd_A,\lhd_A)$ associated to the bimodule $(V;l_{\rhd}, r_{\rhd}, l_{\lhd},  r_{\lhd})$}.
\end{defi}

\begin{ex}\label{RBfna}
In particular, if $P$ is an $\mathcal{O}$-operator  of weight $\lambda$ on a dialgebra $(A,\rhd_A,\lhd_A)$ associated to $(A,\rhd_A,\lhd_A;L_{\rhd}, R_{\rhd}, L_{\lhd}, R_{\lhd})$, then  $P$  is called a {\bf Rota-Baxter operator of weight $\lambda$ on $(A,\rhd_A,\lhd_A)$}, that is, $P$ satisfies
\begin{align}
P(x) \rhd_A P (y)&=P \left(P  (x) \rhd_A y+ x \rhd_A P(y) + \lambda x\rhd_A y \right), \label{RBofdi1}\\
P(x) \lhd_A P (y)&=P \left(P  (x) \lhd_A y+ x \lhd_A P(y) + \lambda x\lhd_A y \right),\quad \forall x, y \in A.\label{RBofdi2} 
\end{align}
In particular,  $P$ is simply called a {\bf Rota-Baxter operator on $(A,\rhd_A,\lhd_A)$} when $\lambda=0$. 
\end{ex}

\begin{rmk}
Recall \cite{B2} that a {\bf Rota-Baxter operator  of weight $\lambda$ } on an associative algebra $(A,\cdot)$ is a linear map  $P$  such that the following condition holds:
\begin{align}
P(x) \cdot P (y)&=P \left(P  (x) \cdot  y+ x \cdot  P(y) + \lambda x \cdot  y \right),\quad \forall x,y\in A. \label{RBofass} 
\end{align}
Thus, if $P$ is a Rota-Baxter operator of weight $\lambda$ on a dialgebra $(A,\rhd,\lhd)$, then $P$  is also a Rota-Baxter operator of weight $\lambda$ on the associative algebras  $(A,\rhd)$ and $(A,\lhd)$, respectively.
\end{rmk}

\begin{pro}
 Let  $(A,\rhd_A,\lhd_A)$   be a dialgebra and $r \in A \otimes A$ whose skew-symmetric part $\alpha$ is
invariant.  Define   binary operations $\rhd_{A^*},\lhd_{A^*}$ on $A^*$ by
\begin{align}
a^* \rhd_{A^*} b^* = L_{\star}^*(\widetilde{\alpha}(b^*))a^*, \quad a^* \lhd_{A^*} b^* = L_{\rhd}^*(\widetilde{\alpha}(b^*))a^*,\quad \forall a^*, b^* \in A^*. \label{dualdistruchen}
\end{align}
Then $(A^*,\rhd_{A^*},\lhd_{A^*};-R_{\lhd}^{*},  -L_{\star}^{*}, R_{\star}^{*},-L_{\rhd}^{*})$  is an $A$-bimodule dialgebra of $(A,\rhd_A,\lhd_A)$.
\end{pro}
\begin{proof}Let $a^*, b^*,c^*\in A^*$. By Lemma \ref{semiinequ}, we have
\begin{align*}
a^* \rhd_{A^*}  (b^* \rhd_{A^*}  c^*) &= L_{\star}^*(\widetilde{\alpha}\left(   L_{\star}^*(\widetilde{\alpha}(c^*))b^*)\right)a^* = R_{\lhd}^*(\widetilde{\alpha}\left( a^* \right))  R_{\lhd}^*(\widetilde{\alpha}(b^*))c^* \\
&= -R_{\lhd}^*(\widetilde{\alpha} ( a^*  ) \rhd_A \widetilde{\alpha}(b^*))c^*  
 =  R_{\lhd}^*(\widetilde{\alpha} ( L_{\star}^*(\widetilde{\alpha}(b^*))a^*))c^* \\
&=  L_{\star}^*(\widetilde{\alpha}(c^*))     L_{\star}^*(\widetilde{\alpha}(b^*))a^* =(a^* \rhd_{A^*}   b^*) \rhd_{A^*}  c^*,\\
a^* \lhd_{A^*}  (b^* \lhd_{A^*}  c^*) &=L_{\rhd}^*(\widetilde{\alpha}( L_{\rhd}^*(\widetilde{\alpha}(c^*))b^* ))a^* = -R_{\star}^*(\widetilde{\alpha}(a^* )) L_{\rhd}^*(\widetilde{\alpha}(c^*))b^* \\
& =  R_{\star}^*(\widetilde{\alpha}(a^* )) R_{\star}^*(\widetilde{\alpha}(b^*))c^* =  R_{\star}^*(\widetilde{\alpha}( a^*) \lhd_A \widetilde{\alpha}(b^*  ))  c^*    \\
&= -R_{\star}^*(\widetilde{\alpha}( L_{\rhd}^*(\widetilde{\alpha}(b^*))a^*  ))  c^*   =  L_{\rhd}^*(\widetilde{\alpha}(c^*))     L_{\rhd}^*(\widetilde{\alpha}(b^*))a^*   \\
&=(a^* \lhd_{A^*}   b^*) \lhd_{A^*}  c^*.
\end{align*}
Thus $(A^*,\rhd_{A^*})$ and $(A^*,\lhd_{A^*})$ are associative algebras. Similarly, Eqs.~\eqref{dias1}-\eqref{dias3} also hold for $(A^*,\rhd_{A^*},\lhd_{A^*})$. Thus $(A^*,\rhd_{A^*},\lhd_{A^*})$ is a dialgebra. Furthermore, let $x \in A$, we have
\begin{align*}
-R_{\lhd}^{*}(x)(a^* \rhd_{A^*}  b^*) & =  -R_{\lhd}^{*}(x)L_{\star}^*(\widetilde{\alpha}(b^*))a^* = -R_{\lhd}^{*}(x)R_{\lhd}^*(\widetilde{\alpha}(a^*))b^*= R_{\lhd}^*(x \rhd \widetilde{\alpha}( a^*)     ) b^* \\
&=-R_{\lhd}^*(\widetilde{\alpha}(R_{\lhd}^{*}(x) a^*)     ) b^*  =-L_{\star}^*(\widetilde{\alpha}(b^*))\left( R_{\lhd}^{*}(x) a^*\right)  =\left(-R_{\lhd}^{*}(x) a^*\right) \rhd_{A^*}  b^*,\\
-L_{\star}^{*} (x)(a^* \rhd_{A^*}  b^*) &=  -L_{\star}^{*} (x)L_{\star}^*(\widetilde{\alpha}(b^*))a^* = L_{\star}^*(\widetilde{\alpha}(  b^* ) \rhd_{A} x)a^* =-L_{\star}^*(\widetilde{\alpha}( L_{\star}^{*} (x) b^* ))a^*  \\
& = a^* \rhd_{A^*} \left(-L_{\star}^{*} (x) b^*\right), \\
 \left(-L_{\star}^{*} (x) a^*\right) \rhd_{A^*}  b^*   &= -L_{\star}^*(\widetilde{\alpha}(b^*)) L_{\star}^{*} (x) a^* = L_{\star}^{*} (x \rhd_{A} \widetilde{\alpha}(  b^*))a^* = L_{\star}^{*} (\widetilde{\alpha}(-R_{\lhd}^{*} (x) b^*))a^* \\
 &=   a^* \rhd_{A^*} \left(-R_{\lhd}^{*} (x) b^*\right), \\
R_{\star}^{*} (x)(a^* \lhd_{A^*}  b) & = R_{\star}^{*} (x)L_{\rhd}^*(\widetilde{\alpha}(b^*))a^* =L_{\rhd}^*(\widetilde{\alpha}(b^*))   R_{\star}^{*} (x) a^* = \left(R_{\star}^{*} (x) a^*\right) \lhd_{A^*}  b^*,\\
-L_{\rhd}^{*} (x)(a^* \lhd_{A^*}  b) &  =  -L_{\rhd}^{*} (x)L_{\rhd}^*(\widetilde{\alpha}(b^*))a^*  
  = L_{\rhd}^*(\widetilde{\alpha}( b^*) \lhd_{A} x)a^*=  - L_{\rhd}^*(\widetilde{\alpha}( L_{\rhd}^{*} (x) b^*))a^*  \\
  &    = a^* \lhd_{A^*} \left(-L_{\rhd}^{*} (x) b^*\right), \\
 \left(-L_{\rhd}^{*} (x) a^*\right) \lhd_{A^*}  b^*  & =-L_{\rhd}^*(\widetilde{\alpha}(b^*))   L_{\rhd}^{*} (x) a^*  =  L_{\rhd}^*(x \lhd \widetilde{\alpha}(b^*)) a^*= L_{\rhd}^*(\widetilde{\alpha}(R_{\star}^{*} (x) b^*)) a^*   \\
 &=a^* \lhd_{A^*} \left(R_{\star}^{*} (x) b^*\right).
\end{align*}
Therefore, $(A^*,\rhd_{A^*}; -R_{\lhd}^{*}, -L_{\star}^{*})$  and $(A^*,\lhd_{A^*};R_{\star}^{*} , -L_{\rhd}^{*})$ are $A$-bimodule algebras of associative algebras $(A^*,\rhd_{A^*})$ and $(A^*,\lhd_{A^*})$ respectively. Moreover, it is straightforward to show that Eqs.~\eqref{di:mpofDias1}-\eqref{di:mpofDias5} hold. So $(A^*,\rhd_{A^*},$ $\lhd_{A^*};-R_{\lhd}^{*},  -L_{\star}^{*}, R_{\star}^{*},-L_{\rhd}^{*})$  is an $A$-bimodule dialgebra of $(A,\rhd_A,\lhd_A)$.
\end{proof}

\begin{thm}\label{ilovethisthm}
 Let  $(A,\rhd_A,\lhd_A)$   be a dialgebra and $r \in A \otimes A$ whose skew-symmetric part $\alpha$ is
invariant. Define  linear maps ${\delta_{\rhd,r}}, {\delta_{\lhd,r}}: A \rightarrow A \otimes A$ by Eq.~\eqref{EF} and binary operations $\rhd_{A^*},\lhd_{A^*}$ on $A^*$ by Eq.~\eqref{dualdistruchen}. Then
the following conditions are equivalent: 
\begin{enumerate}
\item\label{quasiTri:1} $r$ is a solution of the DAYBE in $(A,\rhd_A,\lhd_A)$  such that $(A,\rhd_A,\lhd_A,{\delta_{\rhd,r}}, {\delta_{\lhd,r}})$ is a quasi-triangular bi-dialgebra.
\item\label{quasiTri:2} $\widetilde{r}: A^* \to A$ is an $\mathcal{O}$-operator of weight $1$ on $(A,\rhd_A,\lhd_A)$ associated to $A$-bimodule dialgebra $(A^*,\rhd_{A^*},\lhd_{A^*};-R_{\lhd}^{*},  -L_{\star}^{*}, R_{\star}^{*},-L_{\rhd}^{*})$.
\item\label{quasiTri:3} $\widetilde{r}^t: A^* \to A$ is an $\mathcal{O}$-operator of weight $-1$ on $(A,\rhd_A,\lhd_A)$ associated to $A$-bimodule dialgebra $(A^*,\rhd_{A^*},\lhd_{A^*};-R_{\lhd}^{*},  -L_{\star}^{*}, R_{\star}^{*},-L_{\rhd}^{*})$.
\end{enumerate}
\end{thm}
\begin{proof}\eqref{quasiTri:1} $ \Longleftrightarrow $ \eqref{quasiTri:2}.  Suppose $r$ is a solution of the DAYBE in $(A,\rhd_A,\lhd_A)$. Let $a^*,b^* \in A^*$. Then by Proposition \ref{operaforDAYBE}, we have
 \begin{align*}
   \widetilde{r}\left(a^{*}\right) \rhd_A \widetilde{r}\left(b^{*}\right)&=\widetilde{r}\left(-R_{\lhd}^{*}\left(\widetilde{r}\left(a^{*}\right)\right) b^{*}-L_{\star}^{*}\left(\widetilde{r}^t\left(b^{*}\right)\right) a^{*}\right)  
    = \widetilde{r}\left(-R_{\lhd}^{*}\left(\widetilde{r}\left(a^{*}\right)\right) b^{*}-L_{\star}^{*}\left(\widetilde{r} \left(b^{*}\right)\right) a^{*}+L_{\star}^{*}\left(\widetilde{\alpha} \left(b^{*}\right)\right) a^{*}\right) \\
      &= \widetilde{r}\left(-R_{\lhd}^{*}\left(\widetilde{r}\left(a^{*}\right)\right) b^{*}-L_{\star}^{*}\left(\widetilde{r} \left(b^{*}\right)\right) a^{*}+a^* \rhd_{A^*} b^*\right),\\
\widetilde{r}\left(a^{*}\right) \lhd_A \widetilde{r}\left(b^{*}\right)&=\widetilde{r}\left(R_{\star}^{*}\left(\widetilde{r}\left(a^{*}\right)\right) b^{*}-L_{\rhd}^{*}\left(\widetilde{r}^t\left(b^{*}\right)\right) a^{*}\right) 
 =\widetilde{r}\left(R_{\star}^{*}\left(\widetilde{r}\left(a^{*}\right)\right) b^{*}-L_{\rhd}^{*}\left(\widetilde{r}\left(b^{*}\right)\right) a^{*}+L_{\rhd}^{*}\left(\widetilde{\alpha}\left(b^{*}\right)\right) a^{*} \right)\\
&=\widetilde{r}\left(R_{\star}^{*}\left(\widetilde{r}\left(a^{*}\right)\right) b^{*}-L_{\rhd}^{*}\left(\widetilde{r}\left(b^{*}\right)\right) a^{*}+a^* \lhd_{A^*} b^* \right).
   \end{align*}
Thus $\widetilde{r}: A^* \to A$ is an $\mathcal{O}$-operator of weight $1$ on $(A,\rhd_A,\lhd_A)$ associated to $A$-bimodule dialgebra $(A^*,\rhd_{A^*},\lhd_{A^*};-R_{\lhd}^{*},  -L_{\star}^{*}, R_{\star}^{*},-L_{\rhd}^{*})$. The converse statement can be proved by reversing the argument.

\eqref{quasiTri:1} $ \Longleftrightarrow $ \eqref{quasiTri:3}.  If $r$ is a solution of the DAYBE in $(A,\rhd_A,\lhd_A)$, then so is $\tau(r)$. Let $a^*,b^* \in A^*$. Then by Proposition \ref{operaforDAYBE}, we have
   \begin{align*}
\widetilde{r}^t\left(a^{*}\right) \rhd_A \widetilde{r}^t\left(b^{*}\right)&=\widetilde{r}^t\left(-R_{\lhd}^{*}\left(\widetilde{r}^t\left(a^{*}\right)\right) b^{*}-L_{\star}^{*}\left(\widetilde{r}\left(b^{*}\right)\right) a^{*}\right)  \\
   &= \widetilde{r}^t\left(-R_{\lhd}^{*}\left(\widetilde{r}^t\left(a^{*}\right)\right) b^{*}-L_{\star}^{*}\left(\widetilde{r}^t \left(b^{*}\right)\right) a^{*}-L_{\star}^{*}\left(\widetilde{\alpha} \left(b^{*}\right)\right) a^{*}\right) \\
      &= \widetilde{r}^t\left(-R_{\lhd}^{*}\left(\widetilde{r}^t\left(a^{*}\right)\right) b^{*}-L_{\star}^{*}\left(\widetilde{r}^t \left(b^{*}\right)\right) a^{*}- a^* \rhd_{A^*} b^*\right), \\
      \widetilde{r}^t\left(a^{*}\right) \lhd_A \widetilde{r}^t\left(b^{*}\right)&=\widetilde{r}^t\left(R_{\star}^{*}\left(\widetilde{r}^t\left(a^{*}\right)\right) b^{*}-L_{\rhd}^{*}\left(\widetilde{r} \left(b^{*}\right)\right) a^{*}\right)\\
&=\widetilde{r}^t\left(R_{\star}^{*}\left(\widetilde{r}^t\left(a^{*}\right)\right) b^{*}-L_{\rhd}^{*}\left(\widetilde{r}^t\left(b^{*}\right)\right) a^{*}-L_{\rhd}^{*}\left(\widetilde{\alpha}\left(b^{*}\right)\right) a^{*} \right)\\
&=\widetilde{r}^t\left(R_{\star}^{*}\left(\widetilde{r}^t\left(a^{*}\right)\right) b^{*}-L_{\rhd}^{*}\left(\widetilde{r}^t\left(b^{*}\right)\right) a^{*}- a^* \lhd_{A^*} b^* \right).
      \end{align*}
      Thus $\widetilde{r}^t: A^* \to A$ is an $\mathcal{O}$-operator of weight $-1$ on $(A,\rhd_A,\lhd_A)$ associated to $A$-bimodule dialgebra $(A^*,\rhd_{A^*},\lhd_{A^*};-R_{\lhd}^{*},  -L_{\star}^{*}, R_{\star}^{*},-L_{\rhd}^{*})$. The converse statement can be proved by reversing the argument.
\end{proof}

\begin{pro} \label{rbfna:X}
Let $(A, \rhd,\lhd,\mathfrak{B})$ be a skew-symmetric Frobenius dialgebra  and  $\varphi_{\mathfrak{B}}: A\rightarrow A^{*}$  be the induced linear isomorphism  by $\mathfrak{B}$.  Let $r \in A \otimes A$ whose skew-symmetric part $\alpha$ is
invariant. Then $r$ is a solution of the DAYBE if and only if $P_{r}:= \widetilde{r}\circ \varphi_{\mathfrak{B}}$ satisfies the following equations
\begin{align}
P_r(x) \rhd   P_r (y) &= P_r (P_r  (x) \rhd   y + x \rhd  P_{r}(y)-x \rhd  P_{\alpha}(y)),\label{infordiBBF1}\\
 P_r(x) \lhd   P_r (y)& = P_r (P_r  (x) \lhd   y +   x \lhd  P_{r}(y)-x \lhd  P_{\alpha}(y)),\label{infordiBBF2} \quad \forall x,y \in A.
\end{align}
\end{pro}
\begin{proof}
For all  $x,y \in A$, there exist $a^*, b^* \in A^*$ such that $x=\varphi_{\mathfrak{B}}^{-1}(a^{*})$, $y=\varphi_{\mathfrak{B}}^{-1}(b^{*})$. By  Proposition \ref{dualdj:x}, we have
\begin{align*}
&P_r(x) \rhd   P_r (y) - P_r (P_r  (x) \rhd   y)-P_r (x \rhd  P_{r}(y))+P_r (x \rhd  P_{\alpha}(y))\\
&= \widetilde{r}(a^{*}) \rhd \widetilde{r}(b^{*})-\widetilde{r}\circ \varphi_{\mathfrak{B}}(L_{\rhd}(\widetilde{r}(a^{*}))\varphi_{\mathfrak{B}}^{-1}(b^{*}))-\widetilde{r}\circ \varphi_{\mathfrak{B}}(R_{\rhd}(\widetilde{r}(b^{*}))\varphi_{\mathfrak{B}}^{-1}(a^{*}))+P_r (x \rhd  P_{\alpha}(y))\\
&= \widetilde{r}(a^{*}) \rhd \widetilde{r}(b^{*}) +\widetilde{r}( R_{\lhd}^{*}(\widetilde{r}(a^{*}))b^{*}) +\widetilde{r}(L_{\star}^*(\widetilde{r}(b^{*}))a^{*})-\widetilde{r}  (L_{\star}^*(\widetilde{\alpha}(b^{*})) a^{*} ) \\
&= \widetilde{r}(a^{*}) \rhd \widetilde{r}(b^{*}) +\widetilde{r}( R_{\lhd}^{*}(\widetilde{r}(a^{*}))b^{*}) +\widetilde{r}(L_{\star}^*(\widetilde{r}^t(b^{*}))a^{*}) ,\\
&P_r(x) \lhd   P_r (y) - P_r (P_r  (x) \lhd   y)-P_r (x \lhd  P_{r}(y))+P_r (x \lhd  P_{\alpha}(y))\\
&= \widetilde{r}(a^{*}) \lhd \widetilde{r}(b^{*})-\widetilde{r}\circ \varphi_{\mathfrak{B}}(L_{\lhd}(\widetilde{r}(a^{*}))\varphi_{\mathfrak{B}}^{-1}(b^{*}))-\widetilde{r}\circ \varphi_{\mathfrak{B}}(R_{\lhd}(\widetilde{r}(b^{*}))\varphi_{\mathfrak{B}}^{-1}(a^{*}))+P_r (x \lhd  P_{\alpha}(y))\\
&= \widetilde{r}(a^{*}) \lhd \widetilde{r}(b^{*}) -\widetilde{r}( R_{\star}^{*}(\widetilde{r}(a^{*}))b^{*}) +\widetilde{r}(L_{\rhd}^*(\widetilde{r}(b^{*}))a^{*})-\widetilde{r}(L_{\rhd}^*(\widetilde{\alpha}(b^{*}))a^{*})\\
&= \widetilde{r}(a^{*}) \lhd \widetilde{r}(b^{*}) -\widetilde{r}( R_{\star}^{*}(\widetilde{r}(a^{*}))b^{*}) +\widetilde{r}(L_{\rhd}^*(\widetilde{r}^t(b^{*}))a^{*}).
\end{align*}
By Proposition \ref{operaforDAYBE},   $r$ is a solution of the DAYBE if and only if Eqs.~\eqref{infordiBBF1}-\eqref{infordiBBF2}.
\end{proof}

 \section{Triangular and factorizable diassociative  bialgebras}\label{CBialgebra}
 
 In this section, we study two special classes of quasi-triangular bi-dialgebras. One is the triangular bi-dialgebra, which can be constructed from symmetric solutions to the DAYBE. Moreover, the notion of a pre-dialgebra is introduced to give symmetric solutions of the DAYBE in certain larger dialgebras. The other is the factorizable bi-dialgebra, which gives rise to a factorization of the underlying dialgebra. We show that the Drinfeld classical double of a bi-dialgebra naturally admits a factorizable bi-dialgebra structure. Finally, we introduce the notion of skew-symmetric Rota-Baxter Frobenius dialgebras and prove that a skew-symmetric Rota-Baxter Frobenius dialgebra of weight zero induces a triangular bi-dialgebra. Furthermore, we establish a one-to-one correspondence between factorizable bi-dialgebras and skew-symmetric Rota-Baxter Frobenius dialgebras of nonzero weights.   As an application, we give distinct factorizable bi-dialgebras arising from associative algebras and dialgebras, respectively.

  \subsection{Triangular  bi-dialgebras} 
  
We now turn to the general symmetric solutions of the DAYBE. Then Proposition \ref{operaforDAYBE} can be restated as follows.

\begin{pro}\label{wssqof:0}
 Let  $(A,\rhd,\lhd)$   be a dialgebra and $r \in A\otimes A$ be symmetric.  Then   $r$  is a solution of the DAYBE if and only if $\widetilde{r}$ is an $\mathcal{O}$-operator  on  $(A,\rhd,\lhd)$  associated to the bimodule $(A^*;-R_{\lhd}^{*},  -L_{\star}^{*}, R_{\star}^{*},-L_{\rhd}^{*})$.
\end{pro}

\begin{thm}\label{tcdppab:3}
 Let  $(A,\rhd,\lhd)$   be a dialgebra  and $(V;l_{\rhd},r_{\rhd},l_{\lhd},r_{\lhd})$ be a bimodule of $(A,\rhd,\lhd)$. Let $T: V \rightarrow A$  be a linear map which can be identified as an element in $A \otimes V^{*} \subseteq\left(A \ltimes_{-r_{\lhd}^{*}, l_{\lhd}^{*}-l_{\rhd}^{*}, r_{\rhd}^{*}-r_{\lhd}^{*},-l_{\rhd}^{*}} V^*\right) \otimes\left(A \ltimes_{-r_{\lhd}^{*}, l_{\lhd}^{*}-l_{\rhd}^{*}, r_{\rhd}^{*}-r_{\lhd}^{*},-l_{\rhd}^{*}} V^*\right)$  through  $\operatorname{Hom}(V, A) \cong A \otimes V^{*}$. 
Then  $r=T + \tau(T)$  is a solution of the DAYBE in the dialgebra  $A \ltimes_{-r_{\lhd}^{*}, l_{\lhd}^{*}-l_{\rhd}^{*}, r_{\rhd}^{*}-r_{\lhd}^{*},-l_{\rhd}^{*}} V^*$ if and only if  $T$ is an  $\mathcal{O}$-operator on   $(A,\rhd,\lhd)$  associated to the bimodule  $(V;l_{\rhd},r_{\rhd},l_{\lhd},r_{\lhd})$.
\end{thm}
\begin{proof}
It is similar to the proof of \cite[Theorem 2.5.5]{B2}.
\end{proof}

Let $A$ be a vector space and $r   \in A \otimes A$ be nondegenerate. Then we can define a nondegenerate bilinear form $\mathfrak{B}$ on $A$  by
  \begin{align}
 \mathfrak{B}(x,y):=\left\langle \widetilde{r}^{-1}(x) ,y \right\rangle,\quad \forall x,y \in A. \label{BF:inducebyr}
\end{align}
The bilinear form $\mathfrak{B}$  is called the {\bf  induced bilinear form}  by $r$.

\begin{pro}\label{wssqof}
 Let  $(A,\rhd,\lhd)$   be a dialgebra    and  $r   \in A \otimes A$ be symmetric and nondegenerate. Let $\mathfrak{B}$ be the induced bilinear form by $r$.  Then   $r$  is a solution of the DAYBE in $(A,\rhd,\lhd)$  if and only if  $\mathfrak{B}$ satisfies the “closed” conditions
\begin{align}
\mathfrak{B}(x \star y,z) &= \mathfrak{B}(y,z \rhd x) - \mathfrak{B}(x,y \lhd z),\label{closed} \quad \forall x,y,z \in A.
\end{align}
\end{pro}
\begin{proof}
Since $r   \in A \otimes A$ is symmetric, then  $\mathfrak{B}$ is symmetric.  Since $\widetilde{r}: A^* \to A$  is invertible, for all $x, y,z \in A$, there are $a^*,b^*,c^* \in A^*$ such that  $\widetilde{r}(a^*) = x$, $\widetilde{r}(b^*) = y$. Then  we have
\begin{align*}
 \mathfrak{B}(x \star y,z) &=   \mathfrak{B}(\widetilde{r}(a^*) \star \widetilde{r}(b^*),z),\\
\mathfrak{B}(y,z \rhd x) - \mathfrak{B}(x,y \lhd z) &= \left\langle    b^{*},z \rhd x \right\rangle  - \left\langle   a^{*}, y \lhd z \right\rangle  =\left\langle  -R_{\rhd}^{*}\left(\widetilde{r}\left(a^{*}\right)\right) b^{*}+L_{\lhd}^{*}\left(\widetilde{r}\left(b^{*}\right)\right) a^{*} ,z \right\rangle \\
&= \mathfrak{B}(\widetilde{r}\left(-R_{\rhd}^{*}\left(\widetilde{r}\left(a^{*}\right)\right) b^{*}+L_{\lhd}^{*}\left(\widetilde{r}\left(b^{*}\right)\right) a^{*}\right),z).
\end{align*}
By  Proposition \ref{wssqof:0},  $r$  is a solution of the DAYBE in $(A,\rhd,\lhd)$ if and only if  $\mathfrak{B}$ satisfies Eq.~\eqref{closed}.
\end{proof}

Building on Semenov-Tian-Shansky’s work \cite{STS}, we generalize the corresponding result to the framework of dialgebras.

\begin{pro} \label{rbfna1}
Let $(A, \rhd,\lhd,\mathfrak{B})$ be a skew-symmetric Frobenius dialgebra  and  $\varphi_{\mathfrak{B}}: A\rightarrow A^{*}$  be the induced linear isomorphism  by $\mathfrak{B}$.  Suppose $r \in A \otimes A$ is symmetric.  Then $r$ is a solution of the DAYBE if and only if $P_{r}:= \widetilde{r}\circ \varphi_{\mathfrak{B}}$ is a Rota-Baxter operator on dialgebra  $(A, \rhd,\lhd)$.
\end{pro}
\begin{proof}
It follows from Proposition \ref{rbfna:X}.
\end{proof}

 Recall \cite{L4} that a  {\bf dendriform algebra } is a triple  $(A, \vdash,\dashv)$ where $A$ is a vector space with two binary operations  $\vdash,\dashv : A \otimes A \to A$ such that  
\begin{align*}
 x \vdash (y \vdash z)  =   (x \vdash y + x \dashv y)  \vdash   z ,\;    \;    
  x \vdash (y \dashv z)  =    (x \vdash y)  \dashv z ,  \;    \; 
(  x \dashv y)  \dashv  z  = x \dashv  (y \vdash z + y \dashv z),
\end{align*}
for all $x,y,z \in A$. 
  
 \begin{defi}\label{predppadefi}
A  \textbf{pre-dialgebra} is a quintuple $(A, \nearrow,\searrow,\swarrow, \nwarrow )$  such that  $(A, \nearrow,\searrow)$ and $(A,\swarrow,\nwarrow )$ are dendriform algebras  satisfying the following conditions:
{\small
\begin{align}
 x \nwarrow (y \lhd z) & \!=\!x \nwarrow (y \rhd z), && \!x \swarrow(y \nwarrow z)   \!=\!  x \swarrow(y\! \searrow \!z) , &&\!
 x \swarrow (y \swarrow z )  \!=\! x \!\swarrow \!(y \!\nearrow\! z )  ,  \label{predi1}\\
 (x \searrow y) \nwarrow z  & \!=\!x\searrow (y \lhd z), &&
 ( x \nearrow  y) \nwarrow  z    \!=\! x\nearrow (y \nwarrow z)  ,  &&
 (x \rhd y)\swarrow z   \!=\! x \nearrow (y \swarrow z),  \label{predi2}\\
(x\nwarrow y)  \searrow z  & =(x\searrow y)  \searrow z , &&
 (x \swarrow y) \searrow z    =(x \nearrow y) \searrow z   , &&
 (x \lhd y) \nearrow z   =  (x \rhd y) \nearrow z,\label{predi3} 
\end{align}}
where $x \rhd y = x \nearrow y + x \searrow y $ and $x \lhd y = x \swarrow y + x \nwarrow y $ for all $x,y,z \in A$.
\end{defi}

{}

\begin{rmk}
In fact, the operad of pre-dialgebras is the arity splitting (into two pieces) of the operad of dialgebras in the sense of \cite{PBG}, or equivalently the disuccessor of the operad of dialgebras in the sense of \cite{BBG}. 
\end{rmk}
 
  Recall \cite{TS} that a  {\bf Leibniz-dendriform algebra (or pre-Leibniz algebra)} is a triple  $(A, \succ,\prec )$ where $A$ is a vector space with two binary operations  $\succ,\prec : A \otimes A \to A$ such that  
\begin{align}
(x \succ y + x \prec y ) \succ z  & =x \succ(y \succ z)- y \succ(x \succ z), \label{Leibd1}\\
(x \succ y) \prec z & =- ( y \prec x) \prec z, \label{Leibd2}\\
x \prec( y \succ z + y \prec z) & =(x \prec y) \prec z+y \succ(x \prec z), \quad \forall x,y,z \in A. \label{Leibd3}
\end{align}

Let $(A, \succ,\prec )$ be  a Leibniz-dendriform algebra. Then  there is a Leibniz algebra structure $(A,[\cdot,\cdot])$  on $A$ with $[\cdot,\cdot]$ given by
\begin{align*}
[x,y]:= x \succ y + x \prec y, \quad \forall x,y \in A.
\end{align*}
We call $(A,[\cdot,\cdot])$ the {\bf sub-adjacent Leibniz algebra} of the Leibniz-dendriform algebra $(A, \succ,\prec)$.

\begin{pro}\label{saforpd}
    Let $(A, \nearrow,\searrow,\swarrow, \nwarrow )$  be a pre-dialgebra. 
    \begin{enumerate}
\item\label{preditoDi:1} The binary operations $\rhd$ and $\lhd$ respectively given by 
\begin{align}
x \rhd y := x \nearrow y + x \searrow y,\quad x \lhd y := x \swarrow y + x \nwarrow y, \quad \forall x,y \in A, \label{pdipalj}
\end{align}
define a dialgebra $(A, \rhd,\lhd )$,  called the {\bf sub-adjacent dialgebra} of $(A, \nearrow,\searrow,\swarrow, \nwarrow )$ and $(A, \nearrow,\searrow,\swarrow, \nwarrow )$ is called the {\bf compatible pre-dialgebra structure} on $(A, \rhd,\lhd )$. Moreover, $(A; L_{\nearrow}, R_{\searrow},$ $ L_{\swarrow}, R_{\nwarrow})$ is a bimodule   of  $(A, \rhd,\lhd )$. 
\item\label{preditopreLeib:2} The binary operations $\succ$ and $\prec$ respectively given by 
\begin{align} 
x \succ y := x \nearrow y - y \nwarrow x,\quad x \prec y := x \searrow y - y \swarrow x, \quad \forall x,y \in A, \label{pditoLDa}
\end{align}
define a Leibniz-dendriform algebra $(A, \succ,\prec)$,  called the {\bf associated Leibniz-dendriform algebra} of $(A, \nearrow,\searrow,\swarrow, \nwarrow )$ and $(A, \nearrow,\searrow,\swarrow, \nwarrow )$ is called the {\bf compatible pre-dialgebra structure} on $(A, \succ,\prec)$. 
    \end{enumerate}
\end{pro} 
\begin{proof} \eqref{preditoDi:1}. It is well-known that $(A, \nearrow,\searrow)$ and $(A, \swarrow, \nwarrow)$ are dendriform algebras \cite{L4}.   Let $x,y,z \in A$, we have
\begin{align*}
 &x \lhd (y  \lhd z)-x \lhd (y \rhd z) \\
 &= x \swarrow (y \swarrow z + y \nwarrow z - y \nearrow z - y \searrow z) + x \nwarrow (y \swarrow z + y \nwarrow z - y \nearrow z - y \searrow z)    =0,\\
&(x \rhd y) \lhd z-x \rhd(y \lhd z) \\
&= (x \nearrow y + x \searrow y) \swarrow z + (x \nearrow y + x \searrow y) \nwarrow z 
 - x \nearrow (y \swarrow z + y \nwarrow z)- x \searrow (y \swarrow z + y \nwarrow z) 
 = 0,\\
&(x \lhd y) \rhd z-x \rhd(y\rhd z)\\
&= (x \swarrow y + x \nwarrow y - x \nearrow y - x \searrow y) \nearrow z + (x \swarrow y + x \nwarrow y - x \nearrow y - x \searrow y) \searrow z =0.
\end{align*}
Thus $(A, \rhd,\lhd )$ is a dialgebra. Moreover,  it is straightforward to verify that $(A; L_{\nearrow}, R_{\searrow},$ $ L_{\swarrow}, R_{\nwarrow})$ is a bimodule   of  $(A, \rhd,\lhd )$.

\eqref{preditopreLeib:2}.  Let $x,y,z \in A$, we have
\begin{align*}
&(x \succ y + x \prec y ) \succ z -x \succ(y \succ z)+ y \succ(x \succ z)  \\
&=(x \nearrow y - y \nwarrow x+ x \searrow y - y \swarrow x)\nearrow z - z \nwarrow (x \nearrow y - y \nwarrow + x \searrow y - y \swarrow x) -x \nearrow (y \nearrow z - z \nwarrow y) \\
&\quad + (y \nearrow z - z \nwarrow y)  \nwarrow x + y \nearrow (x \nearrow z - z \nwarrow x) - (x \nearrow z - z \nwarrow x)  \nwarrow y\\
&=  (y \nearrow z - z \nwarrow y)  \nwarrow x- z \nwarrow (x \nearrow y - y \nwarrow + x \searrow y - y \swarrow x)      - y \nearrow (   z \nwarrow x) + (   z \nwarrow x)  \nwarrow y\\
&=    (y \nearrow z  )  \nwarrow x - y \nearrow (   z \nwarrow x)  =0.
\end{align*}
Thus Eq.~\eqref{Leibd1} holds. Similarly, by a direct verification, Eqs.~\eqref{Leibd2}-\eqref{Leibd3} also hold. Thus $(A, \succ,\prec)$ is a Leibniz-dendriform algebra.
\end{proof}

\begin{pro}
  Let $(A, \nearrow,\searrow,\swarrow, \nwarrow )$  be a pre-dialgebra and $(A, \rhd,\lhd )$ be the sub-adjacent dialgebra of $(A, \nearrow,\searrow,\swarrow, \nwarrow )$. Let $(A, \succ,\prec)$ be the associated Leibniz-dendriform algebra  of $(A, \nearrow,\searrow,\swarrow, \nwarrow )$ and $(A, [\cdot,\cdot])$ be the sub-adjacent Leibniz algebra  of $(A, \succ,\prec)$. Then the sub-adjacent Leibniz algebra $(A,[\cdot,\cdot]')$ of  $(A, \rhd,\lhd )$ is  consistent with the sub-adjacent Leibniz algebra $(A, [\cdot,\cdot])$  of $(A, \succ,\prec)$. Hence we have the following commutative diagram:
   \begin{equation*}
\begin{gathered}
	\xymatrix@C=2.4cm{  \txt{ pre-dialgebra\\ $(A, \nearrow,\searrow,\swarrow, \nwarrow )$} \ar@{->}[d]_-{\mathrm{Prop.\;}{\ref{saforpd}}}  \ar@{->}[r]^{\hspace{1.4cm}\mathrm{Prop.\;} \ref{saforpd}   \hspace{1cm}}&   \txt{ dialgebra \\$(A, \rhd,\lhd )$}   \ar@{->}[d]_-{\mathrm{Prop.\;}{\ref{diatoLeib}}} \\
	\txt{  Leibniz-dendriform algebra \\ $(A, \succ,\prec)$}  \ar@{->}[r]^{ \quad\quad\mathrm{sub-adjacent} }  &  \txt{  Leibniz algebra  \\ $(A, [\cdot,\cdot])$ }}
\end{gathered}
\end{equation*}
\end{pro}
\begin{proof}
Suppose that $(A,[\cdot,\cdot]')$ is the sub-adjacent Leibniz algebra of $(A, \rhd,\lhd )$ and $(A,[\cdot,\cdot])$ is the sub-adjacent Leibniz algebra of $(A, \succ,\prec)$. Then for all  $x,y \in A$, we have
 \begin{align*}
[x,y] &=  x \succ y + x \prec y = x \nearrow y - y \nwarrow x  + x \searrow y - y \swarrow x = x \rhd y  - y \lhd x = [x,y]'.
 \end{align*}
This completes the proof.
\end{proof}

\begin{pro} \label{OonVector}
Let $(A, \rhd_A,\lhd_A)$ be a dialgebra.
\begin{enumerate}
\item\label{item1} Let  $T: V \rightarrow A$  be an  $\mathcal{O}$-operator on $(A,\rhd_A,\lhd_A)$ associated to the bimodule $(V;l_{\rhd},$ $ r_{\rhd}, l_{\lhd},  r_{\lhd})$. Then there exists a pre-dialgebra structure $(V, \nearrow_V,\searrow_V,\swarrow_V, \nwarrow_V)$ on  $V$, where $\nearrow_V$, $\searrow_V$, $\swarrow_V$ and  $\nwarrow_V$ are respectively defined by
\begin{align}
 &u \nearrow_{V} v:=l_{\rhd}(T(u)) v,  && u \searrow_{V} v:=r_{\rhd}(T(v)) u, \\
 &u \swarrow_{V} v:=l_{\lhd}(T(u)) v, &&  u \nwarrow_{V} v:=r_{\lhd}(T(v)) u, \quad  \forall u, v \in V.
\end{align}
Moreover, $T$ is a homomorphism of dialgebras from the sub-adjacent dialgebra $(V, \rhd_V,$ $\lhd_V)$ of  $(V, \nearrow_V,\searrow_V,\swarrow_V, \nwarrow_V)$ to $(A, \rhd_A,\lhd_A)$.
\item\label{item2} Let $P: A \to A$ be a Rota-Baxter operator on $(A, \rhd_A,\lhd_A)$. Define  binary operations $\nearrow$, $\searrow$, $\swarrow$ and  $\nwarrow$  on  $A$  by
\begin{align*}
&x \nearrow y := P(x) \rhd_A  y, && x \searrow y := x \rhd_A  P(y), \\
&x \swarrow y := P(x) \lhd_A  y, &&  x \nwarrow y := x \lhd_A  P(y), \quad  \forall x, y \in A.
\end{align*}
Then  $(A, \nearrow,\searrow,\swarrow, \nwarrow )$  is a pre-dialgebra. Moreover, $T$ is a homomorphism of dialgebras from the sub-adjacent dialgebra   of  $(A, \nearrow,\searrow,\swarrow, \nwarrow )$  to $(A, \rhd_A,\lhd_A)$.
\end{enumerate}
\end{pro}  
\begin{proof}
(\ref{item1}). By \cite[Theorem 5.2]{BGN2009}, $(V, \nearrow_{V},\searrow_{V})$ and  $(V, \swarrow_{V},\nwarrow_{V})$ are dendriform algebras.  Moreover, let $u,v,w \in A$, we have
\begin{align*}
 u\nwarrow_V (v \lhd_V w)   - u \nwarrow_V (v \rhd_V w) &= \!r_{\lhd}(T(l_{\lhd}(T(\!v\!)\!) w \!+\! r_{\lhd}(T(\!w\!)) v  \! -\! l_{\rhd}(T(\!v\!)) w\! -\! r_{\rhd}(T(\!w\!)) \!v)\!) u, \\
 &= r_{\lhd}(T(v) \lhd_A T(w)  - T(v) \rhd_A T(w) ) u = 0,\\
 u \swarrow_V(v \nwarrow_V w)   -   u \swarrow_V(v  \searrow_V  w)&=l_{\lhd}(T(u)) r_{\lhd}(T(w)) v -l_{\lhd}(T(u)) r_{\rhd}(T(w)) v =0, \\
 u \swarrow_V (v \swarrow_V w)  - u  \swarrow_V  (v  \nearrow_V  w)&=l_{\lhd}(T(u)) l_{\lhd}(T(v)) w - l_{\lhd}(T(u)) l_{\rhd}(T(v)) w =0.
\end{align*}
Thus Eq.~\eqref{predi1} hold. It is straightforward to show that   Eqs.~\eqref{predi2}-\eqref{predi3} also hold. Thus  $(V, \nearrow_V,\searrow_V,\swarrow_V, \nwarrow_V)$ is a  pre-dialgebra.  Moreover, we have
\begin{align*}
T(u \rhd_V v) &= T(u \nearrow_V v + u \searrow_V v) = T(l_{\rhd}(T(u)) v+ r_{\rhd}(T(v)) u) =  T(u) \rhd_A  T(v), \\
T(u \lhd_V v) &= T(u \swarrow_V v + u \nwarrow_V v) = T(l_{\lhd}(T(u)) v+ r_{\lhd}(T(v)) u) =  T(u) \lhd_A   T(v).
\end{align*}

(\ref{item2}). It follows from Example \ref{RBfna} and Item \eqref{item1}.
\end{proof}

\begin{ex} \label{ex:predppa}
 Continuing with Example \ref{Ex:dias} (1), let $(A, \rhd,\lhd)$  be the 2-dimensional dialgebra  with a basis $\{e_1,e_2\}$ whose nonzero products are given by Eq.~\eqref{exfordias}.  Let $P : A \to A$ be a linear map given by
\begin{align*}
P(e_1) = 0, \quad P(e_2) = e_{1}.
\end{align*}
Then  $P$ is a Rota-Baxter   operator on $(A, \rhd,\lhd)$. Thus by  Proposition \ref{OonVector} \eqref{item2}, there is a 2-dimensional pre-dialgebra $(A, \nearrow,\searrow,\swarrow, \nwarrow )$  whose nonzero products are explicitly given by 
\begin{align}
e_2 \swarrow e_2 =   e_1. \label{ex:predppa:cfb}
\end{align}

\end{ex}

Next we give a necessary and sufficient condition on a dialgebra admitting a compatible pre-dialgebra structure.

\begin{thm}\label{iogpdpp:0}
Let $(A, \rhd,\lhd)$   be a  dialgebra. There is a compatiable pre-dialgebra structure on $(A, \rhd,\lhd)$  if and only if there exists an invertible  $\mathcal{O}$-operator on  $(A, \rhd,\lhd)$.
\end{thm} 

\begin{proof} Suppose that $(A, \nearrow,\searrow,\swarrow, \nwarrow )$ is a pre-dialgebra whose the sub-adjacent dialgebra is  $(A, \rhd,\lhd)$. Then we have
\begin{align*}
 x \rhd y  &=   x \nearrow y + x \searrow  y  = \id(L_{\nearrow}(\id(x))y+ R_{\searrow}(\id(y)) x), \\
 x \lhd y  &=   x \swarrow y + x \nwarrow y  = \id(L_{\swarrow}(\id(x))y+ R_{\nwarrow}(\id(y)) x), \quad \forall x,y \in A.
\end{align*}
Thus the identity map $\mathrm{id}:  A \rightarrow A$ is an invertible $\mathcal{O}$-operator on  $(A, \rhd,\lhd)$ associated to the bimodule  $(A; L_{\nearrow}, R_{\searrow}, L_{\swarrow}, R_{\nwarrow})$.

 Conversely, if  $T: V \to A$  is an invertible  $\mathcal{O}$-operator on $(A, \rhd,\lhd)$ associated to the  bimodule $(V;l_{\rhd},r_{\rhd},$ $l_{\lhd},r_{\lhd})$, then by Proposition \ref{OonVector}, there is a compatiable pre-dialgebra structure on  $(A, \rhd,\lhd)$   given by
\begin{align}
 x \nearrow y&:=T\left(l_{\rhd}(x)\left(T^{-1}(y)\right)\right),&& x \searrow y:=T\left(r_{\rhd}(y)\left(T^{-1}(x)\right)\right), \label{ReOondpp1}\\
x \swarrow y&:=T\left(l_{\lhd}(x)\left(T^{-1}(y)\right)\right), && x \nwarrow y:=T\left(r_{\lhd}(y)\left(T^{-1}(x)\right)\right),\quad \forall x, y \in A. \label{ReOondpp2} 
\end{align}
This completes the proof.
\end{proof}

A  nondegenerate symmetric “closed” bilinear form on a dialgebra   gives rise to a pre-dialgebra.


 \begin{pro}\label{syvofConclcyclic}
Let $\mathfrak{B}$  be the nondegenerate symmetric bilinear form on a dialgebra $(A, \rhd,\lhd)$ satisfying the closed condition \eqref{closed}. Then there exists a compatiable pre-dialgebra structure $(A, \nearrow,\searrow,\swarrow, \nwarrow )$ on $(A, \rhd,\lhd)$ with $ \nearrow,\searrow,\swarrow, \nwarrow $ given by
\begin{align*}
\mathfrak{B}(x \nearrow y,z) &= \mathfrak{B}(y, z \lhd x),&&\mathfrak{B}(x \searrow y,z)  =  \mathfrak{B}(x, y \star z),\\
\mathfrak{B}(x \swarrow y,z)&= -\mathfrak{B}(y,z \star x),&&\mathfrak{B}(x \nwarrow y,z)  =    \mathfrak{B}(x,y \rhd z),\quad \forall x,y,z \in A.
\end{align*}
 Moreover,  the  bimodules  $(A;L_{\rhd}, R_{\rhd},
L_{\lhd}, R_{\lhd})$    and
$(A^*;-R_{\nwarrow}^{*}, -L_{\nearrow}^{*}+L_{\swarrow}^{*},R_{\searrow}^{*}-R_{\nwarrow}^{*}, - L_{\nearrow}^{*})$  of the
dialgebra $(A, \rhd,\lhd)$   are equivalent.
 \end{pro}
 \begin{proof}Let  $\varphi_{\mathfrak{B}}: A\rightarrow A^{*}$ be the induced linear isomorphism  by  $\mathfrak{B}$. Then for all $u^*,v^*,w^* \in A^*$, there exist $x,y,z \in A$ such that $u^* = \varphi_{\mathfrak{B}}(x), v^* = \varphi_{\mathfrak{B}}(y), w^* = \varphi_{\mathfrak{B}}(z)$. Therefore
\begin{align*}
\langle \varphi_{\mathfrak{B}}(x \rhd y), z \rangle  &=  \langle  -R_{\lhd}^* (x)\varphi_{\mathfrak{B}}(y),z\rangle - \langle
L_{\star}^*(y)\varphi_{\mathfrak{B}}(x), z \rangle,\\
\langle \varphi_{\mathfrak{B}}(x \lhd y), z \rangle  &=  \langle   R_{\star}^* (x)\varphi_{\mathfrak{B}}(y),z\rangle - \langle
L_{\rhd}^*(y)\varphi_{\mathfrak{B}}(x), z \rangle, 
\end{align*}
Then we obtain  
\begin{align*}
 \varphi_{\mathfrak{B}}^{-1}(u^*) \rhd \varphi_{\mathfrak{B}}^{-1}(v^*)  &= \varphi_{\mathfrak{B}}^{-1}\big(-R_{\lhd}^* (\varphi_{\mathfrak{B}}^{-1}(u^*))v^* -L_{\star}^*(\varphi_{\mathfrak{B}}^{-1}(v^*))u^*\big),\\
 \varphi_{\mathfrak{B}}^{-1}(u^*) \lhd \varphi_{\mathfrak{B}}^{-1}(v^*)  &= \varphi_{\mathfrak{B}}^{-1}\big( R_{\star}^* (\varphi_{\mathfrak{B}}^{-1}(u^*))v^* -L_{\rhd}^*(\varphi_{\mathfrak{B}}^{-1}(v^*))u^*\big).
\end{align*}
Hence $\varphi_{\mathfrak{B}}^{-1}: A^* \to A$ is an invertible $\mathcal{O}$-operator   on $(A, \rhd,\lhd)$  associated to the  bimodule $(A^*;-R_{\lhd}^{*},  -L_{\star}^{*}, R_{\star}^{*},$ $-L_{\rhd}^{*})$. By Proposition \ref{OonVector},   there is
a compatiable pre-dialgebra structure $(A, \nearrow,\searrow,\swarrow, \nwarrow )$  on 
$(A, \rhd,\lhd)$  with $\nearrow,\searrow,\swarrow$ and $\nwarrow$   given by
Eqs.~\eqref{ReOondpp1}-\eqref{ReOondpp2} in taking $T := \varphi_{\mathfrak{B}}^{-1}, l_{\rhd}:=-R_{\lhd}^*, r_{\rhd}:= -L_{\star}^*, l_{\lhd}:=
R_{\star}^*, r_{\lhd}:= -L_{\rhd}^*$. Thus we have
\begin{align*}
\mathfrak{B}(x \nearrow y, z) &=  \langle \varphi_{\mathfrak{B}} (x \nearrow y), z\rangle
=  -\langle    R_{\lhd}^*(x)\left(\varphi_{\mathfrak{B}}(y)\right), z\rangle =
  \mathfrak{B}(y, z \lhd x), \\
\mathfrak{B}(x \searrow y, z) &=  \langle \varphi_{\mathfrak{B}} (x \searrow y), z\rangle
=  -\langle    L_{\star}^*(y)\left(\varphi_{\mathfrak{B}}(x)\right), z\rangle =
  \mathfrak{B}(x, y \star z), \\
\mathfrak{B}(x \swarrow y, z) &=  \langle \varphi_{\mathfrak{B}} (x \swarrow y), z\rangle
=   \langle    R_{\star}^*(x)\left(\varphi_{\mathfrak{B}}(y)\right), z\rangle =
  -\mathfrak{B}(y,z \star x), \\
\mathfrak{B}(x \nwarrow y, z) &=  \langle \varphi_{\mathfrak{B}} (x \nwarrow y), z\rangle
=  - \langle    L_{\rhd}^*(y)\left(\varphi_{\mathfrak{B}}(x)\right), z\rangle =
   \mathfrak{B}(x,y \rhd z).
\end{align*}
The rest is straightforward. 
\end{proof}

\begin{rmk}
Just as a nondegenerate Connes cocycle on an associative algebra induces a compatible dendriform algebra structure on it \cite[Theorem 4.1.1]{B2}, a nondegenerate  closed symmetric bilinear form on a dialgebra plays the same role as the Connes cocycle does and gives rise to a compatiable pre-dialgebra. Thus in the sense of Proposition \ref{syvofConclcyclic}, it can be regarded as the analogue (symmetric version) of the   Connes cocycle for dialgebras.
\end{rmk}

\begin{cor} \label{pretoBialgebra}
    Let $(A, \nearrow,\searrow,\swarrow, \nwarrow )$ be a pre-dialgebra and $(A, \rhd,\lhd)$  be the sub-adjacent dialgebra of $(A, \nearrow,\searrow,\swarrow, \nwarrow )$. Let  $\left\{e_{1}, \ldots, e_{n}\right\}$  be a basis of  $A$  and  $\left\{e_{1}^{*}, \ldots, e_{n}^{*}\right\}$  be the dual basis. Then
\begin{align}
r=\sum_{i=1}^{n}\left(e_{i} \otimes e_{i}^{*}+e_{i}^{*} \otimes e_{i}\right)
\end{align}
is a symmetric solution of the DAYBE in the dialgebra  $\hat{A}$ defined by
\begin{align*}
\hat{A}:=A \ltimes_{-R_{\nwarrow}^{*}, -L_{\nearrow}^{*}+L_{\swarrow}^{*},R_{\searrow}^{*}-R_{\nwarrow}^{*}, - L_{\nearrow}^{*}} A^{*}.
\end{align*}
\end{cor} 
\begin{proof} By Theorem \ref{iogpdpp:0}, the identity map $\mathrm{id} :  A \rightarrow A$  is an invertible $\mathcal{O}$-operator on $(A, \rhd,\lhd)$  associated to the bimodule  $(A; L_{\nearrow}, R_{\searrow}, L_{\swarrow}, R_{\nwarrow})$. Hence it follows from Theorem \ref{tcdppab:3}. 
\end{proof}

We conclude this subsection by providing an example illustrating the above construction.

\begin{ex}
 Continuing with Example \ref{ex:predppa}, let $(A, \nearrow,\searrow,\swarrow, \nwarrow )$  be the 2-dimensional pre-dialgebra  whose nonzero products are explicitly given by Eq.~\eqref{ex:predppa:cfb}.  Let $\left\{e_{1}^{*},  e_{2}^{*}\right\}$  be the dual basis of $\left\{e_{1},   e_{2}\right\}$.
Thus by  Corollary \ref{pretoBialgebra},  $r= \sum\limits_{i=1}^{2} (e_{i} \otimes e_{i}^{*}+e_{i}^{*} \otimes e_{i})$ is a symmetric solution of the DAYBE in the dialgebra  $\hat{A}:=A \ltimes_{-R_{\nwarrow}^{*}, -L_{\nearrow}^{*}+L_{\swarrow}^{*},R_{\searrow}^{*}-R_{\nwarrow}^{*}, - L_{\nearrow}^{*}} A^{*}$ whose nonzero products are explicitly given by  
\begin{align*}
e_2 \lhd e_2 = e_1,\quad e_1^* \rhd e_2 = -e_2^*.
\end{align*}
By Theorem \ref{tcdppab:1}, there is a 4-dimensional bi-dialgebra $(\hat{A},{\delta_{\rhd,r}}, {\delta_{\lhd,r}})$  where ${\delta_{\rhd,r}}, {\delta_{\lhd,r}}$ are given by
 \begin{align*}
&\delta_{\rhd,r}(e_1) = 0, &&\delta_{\rhd,r}(e_2) =  -e_1 \otimes e_2^*    ,&&\delta_{\rhd,r}(e_1^*) = 0 ,&&\delta_{\rhd,r}(e_2^*) = 0,\\
&\delta_{\lhd,r}(e_1) = 0, &&\delta_{\lhd,r}(e_2) = 0,&&\delta_{\lhd,r}(e_1^*) =  e_2^* \otimes e_2^* ,&&\delta_{\lhd,r}(e_2^*) = 0.
\end{align*}

\end{ex}

  \subsection{Factorizable  bi-dialgebras} 
  
  \begin{defi}\label{factorizabledb}
	A quasi-triangular bi-dialgebra $(A,\rhd,\lhd,{\delta_{\rhd,r}}, {\delta_{\lhd,r}})$     is called {\bf factorizable} if the skew-symmetric part $\alpha$ of $r$ is nondegenerate, which means that the linear map $\widetilde{\alpha}:= \widetilde{r} - \widetilde{r}^t: A^* \to A$  is a linear isomorphism of vector spaces. 
\end{defi}

\begin{pro}\label{fquasidba}
	Let  $(A,\rhd,\lhd)$   be a dialgebra and $r \in A \otimes A$.  
	Then $(A,\rhd,\lhd,{\delta_{\rhd,r}}, {\delta_{\lhd,r}})$ is a  factorizable bi-dialgebra if and only if $(A,\rhd,\lhd,{\delta_{\rhd,\tau(r)}}, {\delta_{\lhd,\tau(r)}})$ is a factorizable bi-dialgebra.
\end{pro}
\begin{proof}
	It follows from Proposition~\ref{tquasidba}.
\end{proof}

Consider the linear map 
\begin{equation*}
	A^* \xrightarrow{\;\;\;\; \widetilde{r} \oplus  \widetilde{r}^t\;\;\;\;} A \oplus A \xrightarrow{(a, b) \mapsto a - b} A.
\end{equation*}

The following result justifies the terminology of a factorizable bi-dialgebra.
\begin{pro}\label{prop:fpt}
	Let $(A,\rhd,\lhd,{\delta_{\rhd,r}}, {\delta_{\lhd,r}})$ be a factorizable bi-dialgebra.
	Then $\mathrm{Im}(\widetilde{r} \oplus \widetilde{r}^t)$ is a  diassociative subalgebra of the direct sum dialgebra $A \oplus A$, 
	which is isomorphic to the dialgebra $(A^*, \rhd_r, \lhd_r)$, where $\rhd_r, \lhd_r: A^* \otimes A^* \to A^*$ are respectively defined by Eqs.~\eqref{wsddiybe1} and \eqref{wsddiybe2}.
	Moreover, any $x \in A$ has a unique decomposition 
	$x = x_{+} - x_{-}$ with $(x_{+}, x_{-}) \in \mathrm{Im}(\widetilde{r} \oplus \widetilde{r}^t)$.
\end{pro}
\begin{proof}
	By Theorem~\ref{thm:daybe}, both $\widetilde{r} $ and $\widetilde{r}^t$ are homomorphisms of dialgebras.
	Therefore, $\mathrm{Im}(\widetilde{r} \oplus \widetilde{r}^t)$ is a  diassociative subalgebra of the direct sum dialgebra $A \oplus A$.
	Since $\widetilde{\alpha} = \widetilde{r} - \widetilde{r}^t$ is a linear isomorphism of vector spaces, it follows that the dialgebra $\mathrm{Im}(\widetilde{r} \oplus \widetilde{r}^t)$ is isomorphic to the dialgebra $(A^*, \rhd_r, \lhd_r)$.
	Moreover, we have
	\begin{equation*}
		\widetilde{r}\widetilde{\alpha}^{-1}(x) - \widetilde{r}^t\widetilde{\alpha}^{-1}(x) = (\widetilde{r} - \widetilde{r}^t) \widetilde{\alpha}^{-1}(x)= x, \quad \forall x \in A,
	\end{equation*} 
	which shows that $x = x_{+} - x_{-}$ with $x_{+} =\widetilde{r}\widetilde{\alpha}^{-1}(x)$ and $x_{-} =  \widetilde{r}^t\widetilde{\alpha}^{-1}(x)$.
	The uniqueness also follows from the fact that $\widetilde{\alpha}: A^* \to A$ is a linear isomorphism of vector spaces.
\end{proof}

Let $(A, \rhd_A,\lhd_A,\delta_{\rhd},\delta_{\lhd})$   be an arbitrary bi-dialgebra and $\rhd_{A^*},\lhd_{A^*}: A^*  \otimes A^*  \to  A^*$ be the linear duals of  $\delta_{\rhd}$ and  $\delta_{\lhd}$ respectively. By  Theorem \ref{sandengjia}, there is a bi-dialgebra structrue on the direct sum  $\mathfrak{D}: = A \oplus A^*$  of vector spaces with $\rhd_{\mathfrak{D}},\lhd_{\mathfrak{D}}$ given by 
\begin{align}
(x+a^*) \rhd_{\mathfrak{D}} (y+b^*) &=x \rhd_A y\!-\!\mathcal{R}_{\lhd}^{*}(a^*) y\!-\!\mathcal{L}_{\star}^{*}(b^*) x-R_{\lhd}^{*}(x) b^*-L_{\star}^{*}(y) a^*\!+\!a^* \rhd_{A^*} b^*,\\
(x+a^*) \lhd_{\mathfrak{D}} (y+b^*) &=x \lhd_A y\!+\! \mathcal{R}_{\star}^{*}(a^*) y-\mathcal{L}_{\rhd}^{*} (b^*) x\!+\!R_{\star}^{*}(x) b^*-L_{\rhd}^{*}(y) a^*\!+\!a^* \lhd_{A^*} b^*, 
\end{align}
where $x,y \in A, a^*, b^* \in A^*$. Then $ (\mathfrak{D}, \rhd_{\mathfrak{D}},\lhd_{\mathfrak{D}})$ is called a {\bf Drinfeld
classical double} of the bi-dialgebra $(A, \rhd_A,\lhd_A,\delta_{\rhd},\delta_{\lhd})$.
 
\begin{thm} \label{Dfincdouble}
Let $(A, \rhd_A,\lhd_A,\delta_{\rhd},\delta_{\lhd})$ be a bi-dialgebra  and 
 $ (\mathfrak{D}, \rhd_{\mathfrak{D}},\lhd_{\mathfrak{D}})$ be the Drinfeld classical double of  $(A, \rhd_A,\lhd_A,\delta_{\rhd},\delta_{\lhd})$. Let $\{e_{1}, e_{2}, \ldots, e_{n}\}$  be a basis of $A$  and  $\{e_{1}^{*}, e_{2}^{*}, \!\ldots\!,$ $e_{n}^{*}\}$  be the dual basis. Set 
 $$
 r=\sum_{i}^{n} e_{i} \otimes e_{i}^{*} \in A \otimes A^{*} \subset   \mathfrak{D} \otimes \mathfrak{D}.
 $$
  Then   $\left(\mathfrak{D}, \rhd_{\mathfrak{D}},\lhd_{\mathfrak{D}}, \delta_{\rhd,r},\delta_{\lhd,r} \right)$  is a factorizable bi-dialgebra with $\delta_{\rhd,r},\delta_{\lhd,r} : \mathfrak{D} \rightarrow \mathfrak{D} \otimes \mathfrak{D}$  given by Eq.~\eqref{EF}.
\end{thm}
\begin{proof}
First, we need to verify that the skew-symmetric part $\alpha =\sum_{i}  e_{i} \otimes e_{i}^{*} - e_{i}^{*} \otimes e_{i}$ of $r$ is  
invariant. Note that $\widetilde{\alpha}( a^* + x) = -x + a^*$ for all $x \in A, a^* \in A^*$ which implies that  $\widetilde{\alpha}$ is a linear isomorphism of vector spaces. Moreover, by a straightforward
computation, we have
\begin{align*}
	 \widetilde{\alpha}\left( L_{\lhd_{\mathfrak{D} }}^*(x+a^*) (y^* + b)\right)&= -b \star_A x -\mathcal{R}_{\rhd}^*(y^*)x-\mathcal{L}_{\lhd}^*(a^*)b + y^* \star_{A^*} a^* + R_{\rhd}^*(b)a^*+L_{\lhd}^*(x)y^*\\
	&= \widetilde{\alpha}\left(y^* + b\right) \star_{\mathfrak{D} } (x+a^*) ,\\
	\widetilde{\alpha}\left(L_{\star_{\mathfrak{D} }}^*(x+a^*)(y^* + b)\right) &=  b \rhd_A x +\mathcal{R}_{\lhd}^*(y^*)x-\mathcal{L}_{\star}^*(a^*)b - y^* \rhd_{A^*} a^* - R_{\lhd}^*(b)a^*+L_{\star}^*(x)y^*\\
	&=-\widetilde{\alpha}\left(y^* + b\right)\rhd (x+a^*), \quad \forall x,b \in A, a^*, y^* \in A^*.
\end{align*}
where $\rhd_{A^*},\lhd_{A^*}$ are the linear duals of  $\delta_{\rhd}$, $\delta_{\lhd}$ respectively. Thus by Lemma \ref{semiinequ}, $\alpha$ is invariant. Furthermore,
\begin{align*}
\mathbf{A} (r) &= r_{13} \rhd_{\mathfrak{D}} r_{23} -  r_{23} \lhd_{\mathfrak{D}} r_{12} -  r_{12} \star_{\mathfrak{D}} r_{13}\\
&= \sum_{i,j}e_{i} \otimes  e_{j} \otimes e_{i}^{*} \rhd_{\mathfrak{D}} e_{j}^{*} - e_{j} \otimes  e_{i}\lhd_{\mathfrak{D}} e_{j}^{*}  \otimes e_{i}^{*} -  e_{i}\star_{\mathfrak{D}} e_{j} \otimes  e_{i}^{*}  \otimes e_{j}^{*} \\
&= \sum_{i,j}e_{i} \otimes  e_{j} \otimes e_{i}^{*} \rhd_{A^*} e_{j}^{*} + e_{j} \otimes    \mathcal{L}_{\rhd}^{*} (e_{j}^{*}) e_{i}\otimes e_{i}^{*}  -e_{j} \otimes R_{\star}^{*}(e_{i}) e_{j}^{*}   \otimes e_{i}^{*} -  e_{i}\star_{A} e_{j} \otimes  e_{i}^{*}  \otimes e_{j}^{*} \\
&=0.
\end{align*}
Thus $r$ is a solution of the DAYBE in  $ (\mathfrak{D}, \rhd_{\mathfrak{D}},\lhd_{\mathfrak{D}})$. Therefore, $\left(\mathfrak{D}, \rhd_{\mathfrak{D}},\lhd_{\mathfrak{D}}, \delta_{\rhd,r},\delta_{\lhd,r} \right)$  is a factorizable bi-dialgebra.
\end{proof}

 \subsection{Skew-symmetric Rota-Baxter Frobenius dialgebras} 
In this section, we introduce the notion of a skew-symmetric Rota-Baxter Frobenius dialgebra  and show that a skew-symmetric Rota-Baxter Frobenius dialgebra of zero weight induces a triangular bi-dialgebra. Moreover, we show that there is a one-to-one correspondence between factorizable bi-dialgebras and skew-symmetric Rota-Baxter Frobenius dialgebras of nonzero weights. Finally, we construct distinct factorizable bi-dialgebras from associative algebras and dialgebras, respectively.

By equipping skew-symmetric Frobenius dialgebras with Rota-Baxter operators that satisfy certain compatibility conditions, we introduce the notion of skew-symmetric Rota-Baxter Frobenius dialgebras.

\begin{defi}
A {\bf  skew-symmetric  Rota-Baxter Frobenius dialgebra  of weight $\lambda \in \mathbb{F}$} is a quintuple $(A,  \rhd,\lhd,\mathfrak{B},P)$ such that $(A,  \rhd,\lhd,\mathfrak{B})$ is a skew-symmetric Frobenius dialgebra and $P$ is a Rota-Baxter operator of weight $\lambda$ on $(A,  \rhd,\lhd)$ satisfying the following compatibility condition:
\begin{align}
\mathfrak{B}(P(x),y) +  \mathfrak{B}(x,P(y)) + \lambda\mathfrak{B}(x,y)=0, \quad \forall x,y \in A. \label{xrxtjforrbB}
\end{align}
\end{defi}

\begin{pro}\label{PandtauP}
	Let $(A,  \rhd,\lhd,\mathfrak{B})$ be a skew-symmetric Frobenius dialgebra  and $P: A \rightarrow A$ be a linear map. Then $(A,  \rhd,\lhd,\mathfrak{B},P)$ is a  skew-symmetric  Rota-Baxter Frobenius dialgebra of weight $\lambda$ if and only if $(A,  \rhd,\lhd,-\mathfrak{B}, \hat{P})$ is a  skew-symmetric  Rota-Baxter Frobenius dialgebra  of weight $\lambda$, where $\hat{P}:= - \lambda \id - P$.
\end{pro}
\begin{proof}
It is similar to the proof of \cite[Proposition 4.14]{BLST}.
\end{proof}
 
Let $V, W$ be two vector spaces and $P: V \to W$ be a linear map.
	Denote the dual map by $P^*: W^* \to V^*$, which is defined by 
	\begin{equation*}
		\langle v, P^*(w^*)\rangle = \langle P(v), w^*\rangle, \;\; \forall v \in V, w^* \in W^*.
	\end{equation*}

\begin{pro}\label{px:ppplgphpl}
 Let $A$ be a vector space and $\mathfrak{B}_{d}$ be the  bilinear form on the direct sum  $A \oplus A^*$  of vector spaces  defined by Eq.~\eqref{mtorbl}. Let $P: A \to A$ be a linear map and $\hat{P}:= - \lambda \id - P$ with $\lambda \in  \mathbb{F}$.
 \begin{enumerate}
 \item\label{sfdifromAs1} If $(A,\cdot)$ is an associative  algebra and $P$ is a Rota-Baxter operator of weight $\lambda$ on   the  associative  algebra  $(A,\cdot)$, then  $( A \rightthreetimes_{-R_{\cdot}^*,-L_{\cdot}^*} A^* ,\mathfrak{B}_d,P+\hat{P}^*)$ is a  skew-symmetric Rota-Baxter Frobenius dialgebra.
 \item\label{sfdifromdi2} If $(A,\rhd,\lhd)$ is a  dialgebra  and $P$ is a Rota-Baxter operator of weight $\lambda$ on the dialgebra $(A,\rhd,\lhd)$, then $(A
\ltimes_{-R_{\lhd}^{*},  L_{\lhd}^{*}-L_{\rhd}^{*}, R_{\rhd}^{*}-R_{\lhd}^{*},-L_{\rhd}^{*}}
A^*,\mathfrak{B}_d,P+\hat{P}^*)$ is a  skew-symmetric Rota-Baxter Frobenius dialgebra.
  \end{enumerate}
 \end{pro}

\begin{proof}
\eqref{sfdifromAs1}. By Proposition \ref{p:ppplgphpl},  $(A \rightthreetimes_{-R_{\cdot}^*,-L_{\cdot}^*} A^* ,\mathfrak{B}_d)$ is a skew-symmetric  Frobenius dialgebra. Let $x,y  \in A, a^*,b^*  \in A^*$, we have
\begin{align*}
(P+\hat{P}^*)(x + a^*) \rhd  (P+\hat{P}^*)(y + b^*) &= (P(x)+\hat{P}^*(a^*))  \rhd  (P(y)+\hat{P}^*(b^*)) \\
& = P(x) \cdot P(y) -R_{\cdot}^*(P(x))\hat{P}^*(b^*),\\
 (P+\hat{P}^*)((P+\hat{P}^*)(x + a^*) \rhd (y + b^*))  &=P( P(x) \cdot y)-\hat{P}^*(R_{\cdot}^*(P(x)) b^* )\\
  (P+\hat{P}^*)((x + a^*) \rhd (P+\hat{P}^*)(y + b^*))  &=P( x \cdot P(y))-\hat{P}^*(R_{\cdot}^*(x) \hat{P}^*(b^*) ),\\
  \lambda  (P+\hat{P}^*)((x + a^*) \rhd  (y + b^*))  &=   \lambda P( x \cdot y)-\lambda \hat{P}^*(R_{\cdot}^*(x)  b^* ).\\
 (P+\hat{P}^*)(x + a^*) \lhd  (P+\hat{P}^*)(y + b^*) & = P(x) \cdot P(y) -L_{\cdot}^*(P(y))\hat{P}^*(a^*),\\
 (P+\hat{P}^*)((P+\hat{P}^*)(x + a^*) \lhd (y + b^*))  &=P( P(x) \cdot y)-\hat{P}^*(L_{\cdot}^*( y) \hat{P}^*(a^*) )\\
  (P+\hat{P}^*)((x + a^*) \lhd (P+\hat{P}^*)(y + b^*))  &=P( x \cdot P(y))-\hat{P}^*(L_{\cdot}^*(P(y)) a^*),\\
  \lambda  (P+\hat{P}^*)((x + a^*) \lhd  (y + b^*))  &=   \lambda P( x \cdot y)-\lambda \hat{P}^*(L_{\cdot}^*(y)  a^* ). 
\end{align*}
Note that
\begin{align*}
R_{\cdot}^*(P(x))\hat{P}^*(b^*) &= \hat{P}^*(R_{\cdot}^*(P(x)) b^* ) +  \hat{P}^*(R_{\cdot}^*(x) \hat{P}^*(b^*)  + \lambda \hat{P}^*(R_{\cdot}^*(x)  b^* ),\\
L_{\cdot}^*(P(y))\hat{P}^*(a^*)&=  \hat{P}^*(L_{\cdot}^*( y) \hat{P}^*(a^*) ) +  \hat{P}^*(L_{\cdot}^*(P(y)) a^*) + \lambda \hat{P}^*(L_{\cdot}^*(y)  a^* ).
\end{align*}
Then $P+\hat{P}^*$ is a Rota-Baxter operator of weight $\lambda $ on the  dialgebra  $ A \rightthreetimes_{-R_{\cdot}^*,-L_{\cdot}^*} A^* $.
Moreover,  it is straightforward to check that $P+\hat{P}^*$ satisfies Eq.~\eqref{xrxtjforrbB}. Thus $( A \rightthreetimes_{-R_{\cdot}^*,-L_{\cdot}^*} A^* ,\mathfrak{B}_d,P+\hat{P}^*)$ is a  skew-symmetric Rota-Baxter Frobenius dialgebra.

\eqref{sfdifromdi2}. It is straightforward to check that $P+\hat{P}^*$ is a  Rota-Baxter operator of weight $\lambda $ on the  dialgebra  $A
\ltimes_{-R_{\lhd}^{*},  L_{\lhd}^{*}-L_{\rhd}^{*}, R_{\rhd}^{*}-R_{\lhd}^{*},-L_{\rhd}^{*}}
A^*$.
\end{proof}

\begin{lem}\label{rbfna0}
	Let $A$ be a vector space and $\mathfrak{B}$ be a nondegenerate skew-symmetric bilinear form.
	Let $\varphi_{\mathfrak{B}}: A \to A^*$ be the induced linear isomorphism by $\mathfrak{B}$ and $r_\mathfrak{B} \in A \otimes A$ be the 2-tensor form of $\varphi_{\mathfrak{B}}^{-1}$ given by Eq.~\eqref{inducedr}. Suppose that $ r \in A \otimes A$ and $\alpha: = r - \tau(r)$ is the skew-symmetric part of $r$. Then $P_{r}:= \widetilde{r} \circ \varphi_{\mathfrak{B}}$ satisfies Eq.~\eqref{xrxtjforrbB} if and only if $\alpha=-\lambda r_{\mathfrak{B}}$  where $\lambda \in   \mathbb{F}$.
\end{lem}
\begin{proof}
Let  $x, y \in A$, there exist  $a^{*}, b^{*} \in A^{*}$  such that  $\varphi_{\mathfrak{B}}^{-1}\left(a^{*}\right)=x, \varphi_{\mathfrak{B}}^{-1}\left(b^{*}\right)=y$. Then
\begin{align*}
\mathfrak{B}\left(P_{r}(x), y\right)&=-\mathfrak{B}\left(y, P_{r}(x)\right)=-\left\langle \varphi_{\mathfrak{B}}(y), \widetilde{r}\circ \varphi_{\mathfrak{B}} (x)\right\rangle=-\left\langle r, a^{*} \otimes b^{*}\right\rangle,\\
\mathfrak{B}(x, P_{r}(y))&=\left\langle\varphi_{\mathfrak{B}}(x),  \widetilde{r}\circ \varphi_{\mathfrak{B}}(y)\right\rangle=\left\langle r, b^{*} \otimes a^{*}\right\rangle=\left\langle \tau(r), a^{*} \otimes b^{*}\right\rangle, \\
\lambda \mathfrak{B}(x, y)&=-\lambda \mathfrak{B}(y, x)=-\lambda\left\langle \varphi_{\mathfrak{B}}(y),\varphi_{\mathfrak{B}}^{-1}\left(\varphi_{\mathfrak{B}}(x)\right)\right\rangle=-\lambda\left\langle r_{\mathfrak{B}}, a^{*} \otimes b^{*}\right\rangle .
\end{align*}
Hence  $P_{r}:= \widetilde{r} \circ \varphi_{\mathfrak{B}}$ satisfies Eq.~\eqref{xrxtjforrbB}  if and only if  $\alpha=-\lambda  r_{\mathfrak{B}}$.
\end{proof}

As a direct consequence, a skew-symmetric Rota-Baxter Frobenius dialgebra of zero  weight  gives rise to a triangular bi-dialgebra in the following sense.

\begin{pro} \label{rbfna2}
Let $(A,  \rhd,\lhd,\mathfrak{B},P)$ be a  skew-symmetric  Rota-Baxter Frobenius dialgebra of weight $0$ and  $\varphi_{\mathfrak{B}}: A \to A^*$ be the induced linear isomorphism by $\mathfrak{B}$. Then $(A, \rhd,\lhd, \delta_{\rhd,r}, \Delta_{\lhd,r})$ is a triangular  bi-dialgebra where $r\in A \otimes A$ is the 2-tensor form of $P \circ \varphi_{\mathfrak{B}}^{-1}$ given by
	\begin{equation}
		\langle r, a^* \otimes b^*\rangle := \langle (P\circ \varphi_{\mathfrak{B}}^{-1})(a^*), b^*\rangle, \;\; \forall a^*, b^* \in A^*. \label{eq:pnbf}
	\end{equation}
\end{pro}
\begin{proof}
	It follows from Proposition~\ref{rbfna:X} and Lemma~\ref{rbfna0} with $r -\tau (r)=0$.
\end{proof}

The following theorem shows that a factorizable bi-dialgebra naturally gives rise to a   skew-symmetric  Rota-Baxter Frobenius dialgebra of nonzero weight.

\begin{thm}\label{thm:fpb2qrp}
	Let $(A, \rhd,\lhd, \delta_{\rhd,r}, \Delta_{\lhd,r})$  be a factorizable bi-dialgebra with $ \widetilde{\alpha} = \widetilde{r}- \widetilde{r}^t$.
	Then $(A, \rhd,\lhd,  \mathfrak{B}, P)$ is a skew-symmetric  Rota-Baxter Frobenius dialgebra  of weight $\lambda$, where the bilinear form $\mathfrak{B}$  and  linear map $P: A \to A$ are respectively defined by
	\begin{align*}
			\mathfrak{B}(x, y) &= -\lambda\langle \widetilde{\alpha}^{-1}(x), y \rangle, \;\; \forall x, y \in A,\\
		P &=  -\lambda \widetilde{r}  \circ \widetilde{\alpha}^{-1}, \quad \lambda \neq 0.
	\end{align*}
\end{thm}
\begin{proof}
	Clearly, $\mathfrak{B}$ is a nondegenerate symmetric bilinear form. Let $\varphi_{\mathfrak{B}}: A \to A^*$ be the induced linear isomorphism by $\mathfrak{B}$. Then we have $\varphi_{\mathfrak{B}} = - \lambda \widetilde{\alpha}^{-1}$. It follows immediately that  $\widetilde{\alpha} = -\lambda \varphi_{\mathfrak{B}}^{-1}$, and hence $\alpha = -\lambda r_\mathfrak{B}$.
	Noting that the skew-symmetric part $\alpha$ of $r$ is  invariant, we conclude that $r_\mathfrak{B}$ is skew-symmetric and invariant as well. Thus by Proposition~\ref{tensorformIndias}, $(A, \rhd,\lhd,  \mathfrak{B})$ is a skew-symmetric  Frobenius dialgebra. 
	
	Moreover by Proposition~\ref{rbfna:X}, $P$ is a Rota-Baxter operator of weight $\lambda$ on the dialgebra $(A, \rhd,\lhd)$ and it satisfies Eq.~\eqref{xrxtjforrbB} by Lemma~\ref{rbfna0}. 
	Thus, $(A, \rhd,\lhd,  \mathfrak{B},P)$ is a skew-symmetric  Rota-Baxter Frobenius dialgebra  of weight $\lambda$.
\end{proof}

Serving as a counterpart to Theorem~\ref{thm:fpb2qrp}, the theorem below demonstrates that a skew-symmetric Rota-Baxter Frobenius dialgebra of nonzero weight induces a factorizable bi-dialgebra, which refines the one-to-one correspondence between factorizable bi-dialgebras and skew-symmetric  Rota-Baxter Frobenius dialgebras of nonzero weight.

\begin{thm}\label{thm:fdiasb2qrpx}
	Let $(A, \rhd,\lhd,  \mathfrak{B}, P)$ be a skew-symmetric  Rota-Baxter Frobenius dialgebra of weight $\lambda \neq 0$. Let $\varphi_{\mathfrak{B}}: A \to A^*$ be the induced linear isomorphism by $\mathfrak{B}$ and  $r \in A \otimes A$ be the 2-tensor form of $P \circ \varphi_{\mathfrak{B}}^{-1}$ given by Eq.~\eqref{eq:pnbf}. Then $(A, \rhd,\lhd, \delta_{\rhd,r}, \delta_{\lhd,r})$ is a factorizable bi-dialgebra with ${\delta_{\rhd,r}}, {\delta_{\lhd,r}}: A \rightarrow A \otimes A$ given by Eq.~\eqref{EF}.
\end{thm}
 \begin{proof}
Let $r_{\mathfrak{B}}$ be the 2-tensor form of $\varphi_{\mathfrak{B}}^{-1}$ and $\alpha$ be the skew-symmetric part  of $r$. By Proposition~\ref{tensorformIndias} and  Lemma~\ref{rbfna0}, $r_{\mathfrak{B}}$ is  skew-symmetric and invariant, $\alpha=-\lambda r_{\mathfrak{B}}$ and $\widetilde{\alpha} = -\lambda \varphi_{\mathfrak{B}}^{-1}$, which show that  $\alpha$ is also invariant and $\widetilde{\alpha}$ is a linear isomorphism.  Moreover, since  $P$  is a Rota-Baxter operator of weight  $\lambda$ on  $(A, \rhd,\lhd)$, it follows from
Proposition~\ref{rbfna:X}  that  $r$  is a solution of the DAYBE in  $(A, \rhd,\lhd)$. Therefore, $(A, \rhd,\lhd, \delta_{\rhd,r}, \delta_{\lhd,r})$ is a factorizable bi-dialgebra.
 \end{proof}

\begin{pro}
	Let  $(A, \rhd,\lhd, \delta_{\rhd,r}, \delta_{\lhd,r})$ be a factorizable bi-dialgebra, which corresponds to a skew-symmetric  Rota-Baxter Frobenius dialgebra   $(A, \rhd,\lhd,  \mathfrak{B}, P)$ of weight $\lambda \neq 0$ via Theorems~\ref{thm:fpb2qrp} and \ref{thm:fdiasb2qrpx}. 
	Then the factorizable bi-dialgebra $(A, \rhd,\lhd, \delta_{\rhd,\tau(r)}, \delta_{\lhd,\tau(r)})$ corresponds to the skew-symmetric Rota-Baxter Frobenius dialgebra  $(A, \rhd,\lhd, -\mathfrak{B}, \hat{P})$ of weight $\lambda \neq 0$ where $\hat{P} := - \lambda \id - P$. 
	In conclusion, we have the following commutative diagram.
	\begin{displaymath}
		\xymatrix{ 
			(A, \rhd,\lhd, \delta_{\rhd,r}, \delta_{\lhd,r})\ar@{<->}[rr]^{\mathrm{Prop.}~\ref{fquasidba}}\ar@{->}[d]^{\mathrm{Thm.}~\ref{thm:fpb2qrp}}&  &  (A, \rhd,\lhd, \delta_{\rhd,\tau(r)}, \delta_{\lhd,\tau(r)}) \ar@{->}[d]^{\mathrm{Thm.}~\ref{thm:fpb2qrp}} \\
		(A, \rhd,\lhd,  \mathfrak{B}, P) \ar@<1.5ex>[u]^{\mathrm{Thm.}~\ref{thm:fdiasb2qrpx}} \ar@{<->}[rr]^{\mathrm{Prop}.~\ref{PandtauP} }& &   (A, \rhd,\lhd, -\mathfrak{B}, \hat{P}) \ar@<1.5ex>[u]^{\mathrm{Thm.}~\ref{thm:fdiasb2qrpx}}}
	\end{displaymath}
\end{pro}
\begin{proof}
	By Theorem~\ref{thm:fpb2qrp}, the factorizable bi-dialgebra $ (A, \rhd,\lhd, \delta_{\rhd,\tau(r)}, \delta_{\lhd,\tau(r)})$ induces a  skew-symmetric  Rota-Baxter Frobenius dialgebra  $(A, \rhd,\lhd,  \mathfrak{B}^{\prime}, P^{\prime})$ of weight $\lambda \ne 0$, where
	\begin{align*}
		\mathfrak{B}^{\prime}(x, y)&=   \lambda\langle \widetilde{\alpha}^{-1}(x), y \rangle= -\mathfrak{B}(x, y), \quad \forall x, y \in A,\\
		P^{\prime} &=  \lambda  \widetilde{\tau(r)} \circ \widetilde{\alpha}^{-1}=  \lambda  \widetilde{r}^t \circ \widetilde{\alpha}^{-1}=   \lambda \widetilde{r} \circ \widetilde{\alpha}^{-1} - \lambda\widetilde{\alpha} \circ \widetilde{\alpha}^{-1} = -P -\lambda \mathrm{id}  = \hat{P}.
	\end{align*}
	Thus, $  (A, \rhd,\lhd, \delta_{\rhd,\tau(r)}, \delta_{\lhd,\tau(r)})$ induces a skew-symmetric  Rota-Baxter Frobenius dialgebra $(A, \rhd,\lhd,\mathfrak{B}, \hat{P})$ of weight $\lambda$ by Theorem~\ref{thm:fpb2qrp}.
	Conversely, by a similar argument, we show that the skew-symmetric  Rota-Baxter Frobenius dialgebra $(A, \rhd,\lhd,\mathfrak{B}, \hat{P})$ of weight $\lambda$ induces the factorizable bi-dialgebra $(A, \rhd,\lhd, \delta_{\rhd,\tau(r)}, $ $\delta_{\lhd,\tau(r)})$ via Theorem~\ref{thm:fdiasb2qrpx}.
\end{proof}

As an application of Theorem \ref{thm:fdiasb2qrpx}, we construct distinct factorizable bi-dialgebras from associative algebras and dialgebras, respectively.

\begin{cor}\label{jpx:ppplgphpl}
 Let $A$ be a vector space, $P: A \to A$ be a linear map and $\hat{P}:= - \lambda \id - P$.   Let $\left\{e_{1}, \cdots, e_{n}\right\}$ be a basis of  $A$  and  $\left\{e_{1}^{*}, \cdots, e_{n}^{*}\right\}$  be the dual basis.   Suppose $\lambda \ne 0$. Set
 \begin{align*}
r=\sum_{i}^n \hat{P} (e_i) \otimes  e_i^{*}-e_i^*\otimes  P(e_i).
\end{align*}
 \begin{enumerate}
 \item\label{xy:sfdifromAs1}  If $(A,\cdot)$ is an associative  algebra and $P$ is a Rota-Baxter operator of weight $\lambda$ on the associative algebra $(A,\cdot)$,  then  $( A \rightthreetimes_{-R_{\cdot}^*,-L_{\cdot}^*} A^* ,  \delta_{\rhd,r},\delta_{\lhd,r})$ is a factorizable bi-dialgebra with ${\delta_{\rhd,r}}, {\delta_{\lhd,r}}$ given by Eq.~\eqref{EF}.
 \item\label{xy:sfdifromdi2} If $(A,\rhd,\lhd)$ is a  dialgebra   and $P$ is a Rota-Baxter operator of weight $\lambda$ on the dialgebra $(A,\rhd,\lhd)$,  then  $(A
\ltimes_{-R_{\lhd}^{*},  L_{\lhd}^{*}-L_{\rhd}^{*}, R_{\rhd}^{*}-R_{\lhd}^{*},-L_{\rhd}^{*}}
A^*, \delta_{\rhd,r},\delta_{\lhd,r})$   is a factorizable bi-dialgebra  with ${\delta_{\rhd,r}}, {\delta_{\lhd,r}}$ given by Eq.~\eqref{EF}.
  \end{enumerate}

\end{cor}

\begin{proof}\eqref{x:sfdifromAs1}. By Proposition \ref{px:ppplgphpl}, $( A \rightthreetimes_{-R_{\cdot}^*,-L_{\cdot}^*} A^* ,\mathfrak{B}_d,P+\hat{P}^*)$ is a  skew-symmetric Rota-Baxter Frobenius dialgebra of weight $\lambda$ with $\mathfrak{B}_{d}$ given by Eq.~\eqref{mtorbl}. Let $\varphi_{\mathfrak{B}_{d}}: A \to A^*$ be the induced linear isomorphism by $\mathfrak{B}_{d}$. Note that 
\begin{align*}
\varphi_{\mathfrak{B}_{d}}\left(x+a^{*}\right)=   a^{*}-x, \quad \forall x \in A, a^* \in A^*.
\end{align*}
By Theorem \ref{thm:fdiasb2qrpx}, we obtain a linear transformation  $\widetilde{r}$  on  $A \oplus A^{*}$ given by
\begin{align*}
\widetilde{r}\left(x+a^{*}\right)=\left(P+\hat{P}^*\right)\varphi_{\mathfrak{B}_{d}}^{-1}\left(x+a^{*}\right)=-P(x)+\hat{P}^*( a^{*}) .
\end{align*}
Then we have
\begin{align*}
r &= \sum_{i,j}(e_i  + e_j^*)\otimes \widetilde{r}(e_i^* + e_j) = \sum_{i,j}(e_i  + e_j^*)\otimes (- P(e_j) + \hat{P}^*( e_i^{*}) ) \\
&= \sum_{i,j}-e_j^*\otimes  P(e_j) +  e_i \otimes \hat{P}^*( e_i^{*})  = \sum_{i}^n\hat{P} (e_i) \otimes  e_i^{*}-e_i^*\otimes  P(e_i).
\end{align*}
By Theorem \ref{thm:fdiasb2qrpx},  $( A \rightthreetimes_{-R_{\cdot}^*,-L_{\cdot}^*} A^* ,  \delta_{\rhd,r},\delta_{\lhd,r})$ is a factorizable bi-dialgebra.

\eqref{x:sfdifromdi2}.  It is similar to the proof of Item \eqref{x:sfdifromAs1}
\end{proof}

\begin{cor}\label{EX:jpx:ppplgphpl}
 Let $A$ be a vector space. Let $\left\{e_{1}, \cdots, e_{n}\right\}$ be a basis of  $A$  and  $\left\{e_{1}^{*}, \cdots, e_{n}^{*}\right\}$  be the dual basis.    Set $r=\sum\limits_{i}^n  e_i^*\otimes   e_i$.
 \begin{enumerate}
 \item\label{x:sfdifromAs1}  If $(A,\cdot)$ is an associative  algebra,  then  $( A \rightthreetimes_{-R_{\cdot}^*,-L_{\cdot}^*} A^* ,  \delta_{\rhd,r},\delta_{\lhd,r})$ is a factorizable bi-dialgebra with ${\delta_{\rhd,r}}, {\delta_{\lhd,r}}$ given by Eq.~\eqref{EF}.
 \item\label{x:sfdifromdi2} If $(A,\rhd,\lhd)$ is a  dialgebra,  then  $(A
\ltimes_{-R_{\lhd}^{*},  L_{\lhd}^{*}-L_{\rhd}^{*}, R_{\rhd}^{*}-R_{\lhd}^{*},-L_{\rhd}^{*}}
A^*, \delta_{\rhd,r},\delta_{\lhd,r})$   is a factorizable bi-dialgebra  with ${\delta_{\rhd,r}}, {\delta_{\lhd,r}}$ given by Eq.~\eqref{EF}.
  \end{enumerate}
\end{cor}
\begin{proof}
By Corollary \ref{jpx:ppplgphpl}, it follows   from the fact that $-\id$ is a Rota-Baxter operator of weight $1$ on the associative  algebra $(A,\cdot)$  and  the  dialgebra $(A,\rhd,\lhd)$, respectively.
\end{proof}

\begin{ex} 
(1) Continuing with Example \ref{amaExspe} (2), there is a dialgebra $ A \rightthreetimes_{-R_{\cdot}^*,-L_{\cdot}^*} A^*$ defined by Eqs.~\eqref{ex:diasexZ1}-\eqref{ex:diasexZ2}. Thus   by Corollary \ref{EX:jpx:ppplgphpl},  $( A \rightthreetimes_{-R_{\cdot}^*,-L_{\cdot}^*} A^* ,  \delta_{\rhd,r},\delta_{\lhd,r})$ is a factorizable bi-dialgebra with ${\delta_{\rhd,r}}, {\delta_{\lhd,r}}$ explicitly given by 
 \begin{align*}
&\delta_{\rhd,r}(e_1) = 0, &&\delta_{\rhd,r}(e_2) =0  ,&&\delta_{\rhd,r}(e_1^*) =  e_1^* \otimes e_2^* ,&&\delta_{\rhd,r}(e_2^*) = e_2^* \otimes e_2^*,\\
&\delta_{\lhd,r}(e_1) = 0, &&\delta_{\lhd,r}(e_2) = 0,&&\delta_{\lhd,r}(e_1^*) =  e_1^* \otimes e_2^* ,&&\delta_{\lhd,r}(e_2^*) = e_2^* \otimes e_2^*.
\end{align*}
 
 (2) Continuing with Example \ref{coregularrep} (2), there is a dialgebra $A \ltimes_{-R_{\lhd}^{*},  L_{\lhd}^{*}-L_{\rhd}^{*}, R_{\rhd}^{*}-R_{\lhd}^{*},-L_{\rhd}^{*}}
A^*$ defined by Eqs.~\eqref{exfordias}, \eqref{ex:diasexZ3} and \eqref{ex:diasexZ4}. Thus   by Corollary \ref{EX:jpx:ppplgphpl}, $(A
\ltimes_{-R_{\lhd}^{*},  L_{\lhd}^{*}-L_{\rhd}^{*}, R_{\rhd}^{*}-R_{\lhd}^{*},-L_{\rhd}^{*}}
A^*, \delta_{\rhd,r},\delta_{\lhd,r})$ is a factorizable bi-dialgebra with ${\delta_{\rhd,r}}, {\delta_{\lhd,r}}$ explicitly given by 
{}
 \begin{align*}
&\delta_{\rhd,r}(e_1) = 0, \quad \delta_{\rhd,r}(e_2) =e_1^* \otimes e_1  - e_2^* \otimes e_1  ,\quad\delta_{\rhd,r}(e_1^*) =  e_2^* \otimes e_2^* ,&& \delta_{\rhd,r}(e_2^*) = e_2^* \otimes e_2^*,\\
&\delta_{\lhd,r}(e_1) = 0, \quad \delta_{\lhd,r}(e_2) = e_1^* \otimes e_1  - e_2^* \otimes e_1, \quad \delta_{\lhd,r}(e_1^*) =  e_1^* \otimes e_2^* - e_2^* \otimes e_1^* + e_2^* \otimes e_2^*, && 
  \delta_{\lhd,r}(e_2^*) = e_2^* \otimes e_2^*.
\end{align*}
\end{ex}

\medskip

 \noindent

\medskip

 \noindent


\end{document}